\documentclass[letterpaper]{article}
\usepackage{
amsmath,
amsfonts,
latexsym, 
amsrefs,
amssymb
}
\usepackage{amsmath} 
\usepackage{amssymb}
\usepackage{amsthm}
\usepackage{mathtools}
\usepackage{bm}

\usepackage{tocloft}

\usepackage{verbatim}

\usepackage[shortlabels]{enumitem}

\usepackage{hyperref}
\usepackage{cleveref}

\numberwithin{equation}{section}

\newtheorem{theorem}{Theorem}[section]
\newtheorem{proposition}[theorem]{Proposition}
\newtheorem{lemma}[theorem]{Lemma}
\newtheorem{corollary}[theorem]{Corollary}

\theoremstyle{remark}
\newtheorem{remark}[theorem]{Remark}
\theoremstyle{definition}
 
\newtheorem{definition}[theorem]{Definition}

\usepackage{tikz}
\usepackage{tikz,fullpage}
\usetikzlibrary{arrows,%
                petri,%
                topaths, automata}%

\usepackage{graphicx}
\newcommand{\I}{\mathrm{i}}
\def\beq{\begin{equation}}
\def\eeq{\end{equation}}
\def\bel{\begin{lemma}}
\def\eel{\end{lemma}}

\newcommand{\unn}[2]{[\![#1,#2]\!]}

\DeclareMathOperator{\Var}{Var}

\def\matn{\mathrm{Mat}_{N}}

\def\iu{{\mathrm{i}}}
\def\msc{m_{\mathrm{sc} } }

\def\eer{\normalcolor}

\newcommand{\1}{\mathds{1}}
\usepackage{mathrsfs}
\usepackage{wrapfig}

\usepackage{color}

\newcommand{\Id}{{\mathrm{Id}}}

\numberwithin{equation}{section}
\newcommand{\bla}{\bm{\lambda}}
\newcommand{\bzeta}{\bm{\zeta}}
\newcommand{\bmu}{\bm{\mu}}

\newcommand{\boxi}{\bm{\xi}}

\newcommand{\bnu}{\mbox{\boldmath $\nu$}}

\newcommand{\eps}{\varepsilon}

\newcommand{\diff}{\mathop{}\mathopen{}\mathrm{d}}
\newcommand{\rd}{{\rm d}}

\newcommand{\bx}{{\bf{x}}}

\newcommand{\bz} {{\bf {z}}}

\newcommand{\be}{\begin{equation}}
\newcommand{\ee}{\end{equation}}

\newcommand{\e}{{\varepsilon}}

\newcommand{\la}{\lambda}

\newcommand{\cC}{{\cal C}}

\DeclareMathOperator{\rlog}{ {\rm  rlog}}
\DeclareMathOperator{\ilog}{ {\rm  ilog}}

\def\RR{{\mathbb R}}

\renewcommand{\cal}{\mathcal}

\newcommand{\wt}{\widetilde}

\newcommand{\ii}{\mathrm{i}} 

\renewcommand{\epsilon}{\varepsilon}
\renewcommand{\leq}{\leqslant}
\renewcommand{\geq}{\geqslant}

\renewcommand{\le}{\leq}
\renewcommand{\ge}{\geq}

\renewcommand{\P}{\mathbb{P}}
\newcommand{\E}{\mathbb{E}}
\newcommand{\R}{\mathbb{R}}
\newcommand{\C}{\mathbb{C}}
\newcommand{\N}{\mathbb{N}}

\newcommand{\pa}[1]{\left({#1}\right)}

\newcommand{\abs}[1]{\lvert #1 \rvert}

\newcommand{\absa}[1]{\left\lvert #1 \right\rvert}

\DeclareMathOperator{\tr}{Tr}

\DeclareMathOperator{\re}{Re}
\DeclareMathOperator{\im}{Im}
\renewcommand{\Im}{\operatorname{Im}}
\renewcommand{\Re}{\operatorname{Re}}

\DeclareMathOperator{\sgn}{sgn}

\DeclareMathOperator{\OO}{O}
\DeclareMathOperator{\oo}{o}

\makeatletter
\renewcommand{\subsection}{\@startsection
{subsection}
{2}
{0mm}
{-\baselineskip}
{0 \baselineskip}
{\normalfont\bf\itshape}} 
\makeatother

\usepackage{dsfont}
\usepackage{stmaryrd}

\renewcommand{\phi}{\varphi}

\def\sc{\rho_{\mathrm{sc}} }

\usepackage{bbm}

\DeclareMathOperator{\GOE}{GOE}
\DeclareMathOperator{\GUE}{GUE}

\def\one{{\mathbbm 1}}

\renewcommand{\tilde}{\widetilde}

\def\@empty{}

\def\author#1{\par
    {\centering{\authorfont#1}\par\vspace*{0.05in}}
}

\def\titlefont{\fontsize{13}{15}\bfseries\boldmath\selectfont\centering{}}
\def\authorfont{\fontsize{13}{15}}
\def\abstractfont{\fontsize{8}{10}}

\let\affiliationfont\rhfont

\def\address#1{\par
    {\centering{\affiliationfont#1\par}}\par\vspace*{11pt}
}

\def\body{
\setcounter{footnote}{0}
\def\thefootnote{\alph{footnote}}
\def\@makefnmark{{$^{\rm \@thefnmark}$}}
}

\def\title#1{
    \thispagestyle{plain}
    \vspace*{-14pt}
    \vskip 79pt
    {\centering{\titlefont #1\par}}%
    \vskip 1em
}

\renewenvironment{abstract}{\par%
    \vspace*{6pt}\noindent 
    \abstractfont
    \noindent\leftskip10pt\rightskip10pt
}{%
  \par}

\usepackage{tocloft}

\makeatletter
\renewcommand{\section}{\@startsection
{section}
{1}
{0mm}
{-2\baselineskip}
{1\baselineskip}
{\normalfont\large\scshape\centering}} 
\makeatother

\begin{document}

~\vspace{-1.4cm}

\title{Optimal rigidity and maximum of the characteristic polynomial of Wigner matrices} 

~\vspace{0.5cm}

\noindent\begin{minipage}[b]{0.33\textwidth}

 \author{Paul Bourgade}

\address{Courant Institute\\ New York University\\
   bourgade@cims.nyu.edu}

 \end{minipage}
\noindent\begin{minipage}[b]{0.33\textwidth}
 \author{Patrick Lopatto}

\address{Brown University \\ Division of Applied Mathematics\\
  patrick\textunderscore lopatto@brown.edu}

 \end{minipage}
\noindent\begin{minipage}[b]{0.33\textwidth}
 \author{Ofer Zeitouni}

\address{Weizmann Institute of Science \\and  Courant Institute\\
  ofer.zeitouni@weizmann.ac.il}

 \end{minipage}

\begin{abstract}
We determine to leading order the maximum of the characteristic polynomial for  Wigner matrices and $\beta$-ensembles. 
In the special case of Gaussian-divisible Wigner matrices,  our method provides universality of the maximum up to tightness.
These are the first universal results on the Fyodorov--Hiary--Keating conjectures for these models, and in particular answer the question of optimal rigidity for the spectrum of Wigner matrices. 

Our proofs combine dynamical techniques for universality of eigenvalue statistics with ideas surrounding the maxima of log-correlated fields and Gaussian multiplicative chaos. 
\end{abstract}

\vspace{0.3cm}

\tableofcontents

\section{Introduction}


\subsection{Universality in the Fyodorov--Hiary--Keating program.}\ In 2012,  Fyodorov, Hiary, and Keating (FHK)  initiated a new line of research on the connection between random matrix theory and the Riemann zeta function.
Motivated by ideas from statistical mechanics, they  conjectured that the extremal statistics for both characteristic polynomials of random matrices and  the zeta function on the critical line are identical to those of logarithmically correlated fields \cite{FyoHiaKea2012}. 
Such fields arise whenever one superimposes randomness equally on all length scales, and are characterized by correlations proportional to the logarithm of the inverse distance between two points. The branching random walk and the two-dimensional Gaussian free field are paradigmatic examples.

The FHK conjecture states that for a Haar-distributed $N\times N$ unitary matrix $U_N$, the random variable $X_N$ determined by the equality
\begin{equation}\label{eqn:FHK}
\max_{|z|=1}\log|\det(z-{\rm U}_N)|=\log N-\tfrac{3}{4}\log\log N+X_N
\end{equation}
converges to a variable $X_\infty$ in distribution as $N \rightarrow \infty$,
where $X_\infty$ is distributed as the sum of two independent Gumbel random variables. 
Recently, there have been significant advances towards proving \eqref{eqn:FHK} and analogous results for certain other random matrix ensembles, as we discuss below. However, all previous results have been limited to specific models, which admit representations either as a determinantal point process or a sparse matrix model

In this paper, we consider the FHK conjecture for a far broader class of random matrices. 
We study the following generalization that encompasses  real symmetric and complex Hermitian random matrices with independent entries (Wigner matrices), and systems of interacting particles at inverse temperature $\beta > 0$ and governed by a general potential $V\colon \mathbb{R} \rightarrow \mathbb{R}$ ($\beta$-ensembles). 
Under quite general conditions, the limiting spectrum of a Wigner matrix or $\beta$-ensemble is deterministic and supported on some compact interval $[A,B]$.  
We make this {\it one-cut} hypothesis in the statement below,  although similar asymptotics should hold in the bulk in the complementary {\it multicut} case.  
Fix a small $\eps >0$ and set $I = [A+\eps, B-\eps]$. We let $\lambda_1\leq\dots\leq \lambda_N$ denote with the eigenvalues of the matrix or the particles of the $\beta$-ensemble, as appropriate, and set $\det(E)=\prod_{i=1}^N(E-\lambda_i)$.   \\

\noindent{\bf Problem 1.}\ 
{\it Consider any $\beta$-ensemble or any Wigner matrix; in the matrix case, set $\beta=1$ if it is real symmetric or $\beta=2$ if it is complex Hermitian. 
With the above conventions, show that
\begin{equation}\label{eqn:prob1}
\sqrt{\frac{\beta}{2}}\cdot \max_{E\in I}\big(\log|\det(E)|-\E\big[\log|\det(E)|\big]\big)=\log N-\tfrac{3}{4}\log\log N+Z_N, \text{ where } \lim_{N\rightarrow \infty} Z_N\overset{{\rm (d)}}{=}Z_{\infty}
\end{equation} 
for a random variable $Z_{\infty}$ satisfying the tail decay asymptotic $c\, y e^{-2y}\leq\mathbb{P}(Z_{\infty}>y)\leq c^{-1} y e^{-2y}$ as $y\to\infty$, for some fixed $c>0$. (The exact distribution of $Z_{\infty}$ may depend on the matrix entries, or on $V$ and $\beta$.)}\\

\noindent Such a prediction was made in \cite{fyodorov2016distribution} for the Gaussian Unitary Ensemble (GUE). 
Both \cite{fyodorov2016distribution} and our paper focus on the bulk of the spectrum, since this corresponds to the original FHK setting \eqref{eqn:FHK}.

The first contribution of this paper is to establish that the first order term in the conjecture \eqref{eqn:prob1} is correct, both for Wigner matrices and $\beta$-ensembles defined by a general class of potentials (\Cref{thm:maxWigner} and \Cref{thm:Beta}).  We also establish this conjecture up to tightness for the class of Gaussian-divisible Wigner matrices, in the sense that the maximum for these ensembles can be coupled with the maximum for the Gaussian Orthogonal Ensemble (GOE) up to an error of order 1.

Our first order result is new even for the GOE; previous studies were limited to the GUE \cite{LamPaq2019}. 
Because the general models identified in Problem 1 are not integrable for $\beta \neq 2$,  and do not admit a sparse representation using independent 
variables
for non-quadratic $V$, the techniques previously used to prove FHK 
asymptotics are not applicable.  
Instead, we adopt dynamical ideas based around Dyson Brownian motion, 
which have not previously been applied to FHK asymptotics due the singular, non-local character of the relevant observable.

The method we develop for Problem 1 also leads to a sharp characterization of a fundamental property of random matrices, eigenvalue rigidity. This term refers to the observation that eigenvalues of such matrices behave as repelling particles, with interactions that suppress their fluctuations and trap them near deterministic locations. 
We fix a small constant $c >0$ and consider the following problem.\\

\noindent{\bf Problem 2.}\ {\it  For general self-adjoint random matrices or $\beta$-ensembles,  how large is
$
\displaystyle\max_{cN \le i \le (1-c)N }\big|\lambda_i-\E[\lambda_i]\big|?
$}
\\

This  can be understood as asking for either of the following two things: 
\begin{enumerate} 
\item[(i)] An estimate giving the exact size of the maximum on a set of {\it high probability}, i.e.  $1-{\rm o}(1)$.  
\item[(ii)] A bound that captures the correct order of this maximum with {\it overwhelming probability}, i.e. 
$1-{\rm O}(N^{-D})$ for any $D>0$.
\end{enumerate}
The second contribution of this paper is to answer both versions of this question. 
For (i),  we identify the size of the maximum, including the correct constant prefactor, for Wigner matrices (the first part of \Cref{thm:Wigner}) and $\beta$-ensembles (\Cref{cor:Beta}). These are the first optimal rigidity results for matrix ensembles that are not unitary invariant. Previous works in this direction relied on reducing the rigidity question to one about a Riemann--Hilbert problem; such a translation is only possible for integrable ensembles \cite{claeys2021much,ChaFahWebWon2021}. 
For (ii), we obtain the Gaussian decay of the distribution of $\lambda_i-\E[\lambda_i]$ well beyond the fluctuations regime, in the second part of  \Cref{thm:Wigner}.  This solves the longstanding question of rigidity on the scale $(\log N)/N$ with overwhelming probability.

Our results are obtained by a novel combination of 
methods coming from
the study of universality for random matrices (in particular, heat flow, coupling and homogenization),  with  ideas coming from the theory of logarithmically correlated  fields.
We now give precise statements of our main results, and defer a complete survey of the existing literature to \Cref{s:related} below.

\subsection{Results.}\ \label{subsec:Results}
We begin with the definition of Wigner matrices.
\begin{definition}\label{def:wig}
A  Wigner matrix $H=H(N)$ is a  real symmetric or complex Hermitian $N\times N$ matrix  whose upper-triangular elements $\{H_{ij}\}_{i \le j}$  are independent random variables with mean zero and variances  $\E\left[ |H_{ij}|^2 \right] = N^{-1}$ for all $i\neq j$, and $\E\left[ |H_{ii}|^2 \right] = CN^{-1}$ for all $i$, where $C> 0$ is a constant. We have $H_{ij} = \overline{H}_{ij}$ for $i >j$, and in the case that $H$ is complex Hermitian, we suppose that the 
variables 
$\{\Im H_{ij}\}_{i,j}, \{\Re H_{ij})\}_{i,j}$ are all independent and satisfy
$\E[(\Im H_{ij})^2] = \E[(\Re H_{ij})^2]$ for all $i\neq j$. 
Further, there exists a constant $c>0$ such that, for all $i,j \in \unn{1}{N}$ and $x>0$,
\begin{equation}\label{eqn:tail}
\mathbb{P}\left(|\sqrt{N}H_{ij}|>x\right)\leq c^{-1}\exp\left({-x^{c}}\right).
\end{equation}
Moreover,  a  symmetric Wigner matrix is called Gaussian-divisible if it has the same distribution as $\sqrt{1-\e^2}H+\e {G}$, where $H$ is a Wigner matrix as defined above,  independent of the GOE matrix $G$.  Here $\e\in(0,1)$ does not depend on $N$.
\end{definition}

We recall that the empirical spectral density of a Wigner matrix converges to the semicircle law as $N \rightarrow \infty$, see e.g.\ \cite{anderson2010introduction}.  This distribution has density
\begin{equation}\label{e:scdensity}\rho_{\rm sc} (x) = \frac{\sqrt{ (4 - x^2)_+ }}{2 \pi }.
\end{equation}
We consider the principal branch of the logarithm, extended to the negative real numbers by continuity from above, given by $\log(r e^{\ii \theta}) = \log(r) + \ii \theta$ for any $r>0$ and $\theta \in (-\pi,\pi]$.
As is usual, we define $z^\alpha$ by $\exp({\alpha \log(z)})$.  In particular for real $\lambda$ and $E$
we have
$\Re\log(E-\lambda)=\log|E-\lambda|$ and $\im\log(E-\lambda)=\pi\mathds{1}_{\lambda>E}$
Given a probability measure $\nu$ with  bounded density  and a matrix $H$ with eigenvalues $\lambda_1\leq\dots\leq\lambda_N$,  for real $E$ we also define
\begin{equation}
	\label{def:L_N}
	L_N(E) 
	=
	 \sum_{j=1}^N \log\pa{E-\lambda_j} 
	- N \int_\R \log\pa{E-x}  \rd\nu(x),
\end{equation}
which is the logarithm of the characteristic polynomial up to a centering shift.
The following theorem is our main result on the maximum of the characteristic polynomial for Wigner matrices.

\begin{theorem}\label{thm:maxWigner}
Let $H$ be a symmetric Wigner matrix as in Definition \ref{def:wig} and set $\nu=\rho_{\rm sc}(x)\rd x$ in \eqref{def:L_N}. Then for any $\e,\kappa> 0$ we have
\begin{align*}
&\mathbb{P}\left(\sup_{|E|<2-\kappa} \frac{\re L_N(E)}{\sqrt{2}\log N}\in[1-\e,1+\e]\right)=1-\oo(1),\\
&\mathbb{P}\left(\sup_{|E|<2-\kappa} \frac{\im L_N(E)}{\sqrt{2}\log N}\in[1-\e,1+\e]\right)=1-\oo(1).
\end{align*}
The same result holds for Hermitian Wigner matrices after replacing the  $\sqrt{2}$ factors with $1$.
\end{theorem}

\begin{remark}
For the imaginary part of the logarithm,  a similar estimate on the minimum holds,  by considering the $\sup$ for the Wigner matrix $-H$:
$$
\mathbb{P}\left(\inf_{|E|<2-\kappa} \frac{\im L_N(E)}{\sqrt{2}\log N}\in[-1-\e,-1+\e]\right)=1-\oo(1).
$$
No such statement holds for the real part, as $\inf_{|E|<2-\e}\re L_N(E)=-\infty$.
\end{remark}

\noindent For Gaussian-divisible Wigner matrices,  universality actually holds up to tightness.

\begin{theorem}\label{thm:maxWignerDivisible}
Let $H$ be a Gaussian-divisible symmetric Wigner matrix as in Definition \ref{def:wig}. 
Then for any $\kappa>0$,  there exists a coupling between $H$ and a GOE such that the following sequence of random variables is tight:
\[
\Big(\sup_{|E|<2-\kappa} {\rm Re} L^{H}_N(E)-\sup_{|E|<2-\kappa} {\rm Re} L^{{\rm GOE}}_N(E)\Big)_{N\geq 1}.
\]
\end{theorem}

\begin{corollary}\label{cor:tight}
Conditionally to the tightness of the following random variables for $H$ in the integrable Gaussian orthogonal ensemble, 
$$\sup_{|E|<2-\kappa} \big({\rm Re}L^{H}_N(E)-\sqrt{2}(\log N-\frac{3}{4}\log\log N)),$$
tightness also holds for $H$ in the universal class of Gaussian-divisible symmetric Wigner matrices.
\end{corollary}

\noindent Natural analogues of Theorem \ref{thm:maxWignerDivisible} and Corollary \ref{cor:tight} hold for the Hermitian symmetry class.

\begin{remark}
For the imaginary part of the logarithm,  the same statements Theorem \ref{thm:maxWignerDivisible} and Corollary \ref{cor:tight} are an immediate consequence of the homogenization of the Dyson Brownian motion from \cite[Theorem 3.1]{Bou2018}, 
and an elementary bound on macroscopic linear statistics of Wigner matrices.   The result is more subtle for the real part of the logarithm, as it involves a non-local observable of the spectrum.
\end{remark}

\begin{remark} We emphasize that tightness for the Gaussian ensembles is still 
  elusive,  despite the proof of this result for the circular ensembles \cite{ChhMadNaj2018}. 
Only the first order is established: for the GUE in \cite{LamPaq2019} for the real part,  in \cite{claeys2021much} for the imaginary part,  and for the GOE in Theorem \ref{thm:maxWigner}.
\end{remark}

Our second result considers optimal rigidity of the particles.
The first part establishes a high probability rigidity estimate with an optimal deviation including the multiplicative constant.
The second establishes a rigidity estimate with much stronger control on the low probability exceptional set, which is still of optimal order in $N$.

For a given probability measure $\nu$ as in (\ref{def:L_N}), 
the $i$-th quantiles of $\nu$, denoted $\gamma_i=\gamma_i(N,\nu)$  for $1\leq i\leq N$, are defined through the relation
\begin{equation}\label{quantiles}
\int_{-\infty}^{\gamma_i}\rd\nu=\frac{i-\tfrac{1}{2}}{N}.
\end{equation}

\begin{theorem}\label{thm:Wigner}
Let $H$ be a symmetric Wigner matrix as in Definition \ref{def:wig}.  The following holds.
\begin{enumerate}[(i)]
\item For every $\kappa,\eps >0$, we have
\[
\mathbb{P}\left(\max_{\kappa N\leq k\leq (1-\kappa)N} \frac{\pi}{\sqrt{2}}\cdot \frac{\rho_{\rm sc}(\gamma_k)\, N \big(\lambda_k-\gamma_k\big)}{\log N}\in[1-\e,1+\e]\right)=1-\oo(1),
\]
\item  
For any $\kappa,\e,A>0$ there exists $C>0$ such that the following holds for all $N \in \mathbb{N}$. For all $k \in  [ \kappa N , (1-\kappa)N]$ and
$u\in[0,A\sqrt{\log N}]$,
\begin{equation}\label{eqn:sharpdecay}
\mathbb{P}\left(|\lambda_k-\gamma_k|>u\cdot\frac{\sqrt{2}}{\pi\rho_{\rm sc}(\gamma_k)}\cdot\frac{\sqrt{\log N}}{N}\right)\leq C e^{-(1-\e) u^2}.
\end{equation}
\end{enumerate}
For Hermitian Wigner matrices,  (i) and (ii) also hold after replacing  the $\sqrt{2}$ factor with $1$.
\end{theorem}

A union bound in (ii) proves the optimal rigidity scale $(\log N)/ N$ in the bulk of the spectrum:   for every $D>0$ there exists $C >1$ such that  for all $N \in \N$,
\begin{equation}\label{eqn:optRig}
\mathbb{P}\left(\max_{\kappa N\leq k\leq (1-\kappa)N} \left|\lambda_k-\gamma_k\right|
\ge \frac {C \log N}{N}
\right)\le C N^{-D}.
\end{equation}

We next turn to our results on $\beta$-ensembles. 
We recall that the $\beta$-ensemble of dimension $N$, inverse temperature $\beta >0$, and potential $V\colon \R \rightarrow \R$ is the probability measure on the subset 
$
\Delta_N =\{ \bm \lambda = (\lambda_1, \dots \lambda_N) \in \RR^N  : \lambda_1 \le \lambda_2 \le \dots \le \lambda_N\}
$
given by 
\begin{equation}
	\label{eq:beta_ensembles_density}
	\diff \mu_N (\lambda_1,\dots,\lambda_N)
	= \frac{1}{Z_N}
	\prod_{1 \leq k < l \leq N} 
	\absa{\lambda_k - \lambda_l}^\beta 
	\exp\left(- \frac{\beta N}{2} \sum_{k=1}^N V(\lambda_k)\right) 
	\rd \lambda_1 \dots \rd \lambda_N,
\end{equation}
where $Z_N = Z^{(\beta,V)}_N$ is a normalizing factor.
In this paper, $\beta>0$ is fixed and our assumptions on $V$  are the following.
\begin{enumerate}[label=(A\arabic*)]
\item $V$ is analytic on $\R$.  \label{assumption:analytic}%
\item \label{assumption:V'_at_infinity} At least one of the following growth conditions hold for $V$:
\begin{enumerate}[(i)]
\item Sub-quadratic:
\begin{equation} \label{eq:alternative_assumption}
	\liminf_{x \to \pm \infty} \frac{V(x)}{2 \ln \abs{x}} > 1
	\qquad \text{and} \qquad 
	\limsup_{x \to \pm \infty} \frac{\abs{V'(x)}}{\abs{x}} < \infty
\end{equation}

\item Super-linear:
There exist constants $M_0,C,c > 0$ such that 
$$ 
	V'(x) \geq c \quad \text{and} \quad
	\sup_{y \in [M_0,x]} \frac{V'(y)}{y} \leq C V(x) 
	\quad \text{for all} \quad x \geq M_0,
$$
and similar estimates apply for $x\leq -M_0$, i.e. the above holds for $\tilde V(x):=V(-x)$.
\end{enumerate}

\item \label{assumption:off-criticality} Under the previous assumptions, it is known that the expectation of the empirical spectral measure, given by $\E[N^{-1}\sum_{i=1}^N\delta_{\lambda_i}]$, converges weakly to an absolutely continuous probability measure $\mu_V$ with a continuous density, which we denote by $\rho_V$ (see \cite[Theorem 1]{de1995statistical} and \cite[Proposition 1]{albeverio2011} for details).
We assume that $\rho_V$ is supported on a single interval $[A,B]$
and is positive on $(A,B)$, with square root singularities at $A$ and $B$. This means that there exists a $c> 0$ and a function $r(E)\colon \R \rightarrow \R$  and
\beq\label{eqn:r}
\rho_V(E) =\frac{1}{\pi}  \sqrt{(E-A)(B-E)} r(E)\, \one_{[A,B] }.
\eeq
Moreover,  we assume that $r$ does not vanish on $[A,B]$ and has an analytic extension to $\mathbb{C}$.\footnote{We remark that  assumption \ref{assumption:off-criticality} is satisfied by a large class of potentials. For example, it suffices for $V$ to be convex and twice differentiable \cite[Example 1]{de1995statistical}.}
\item \label{assumption:large_dev} Let 
\begin{equation*}
L_V(x)  = \frac{V(x)}{2} - \int_\R \log\abs{x-t} \, \rd \mu_V(t).
\end{equation*}
There exists a constant $\ell_V$ such that $L_V(x) = \ell_V$ for $x \in [A,B]$, and $L_V(x) > \ell_V$ for $x \notin [A,B]$. 
\end{enumerate}

The following is an analogue of Theorem \ref{thm:maxWigner} for $\beta$-ensembles.

\begin{theorem}\label{thm:Beta} Let $(\lambda_1,\dots,\lambda_N)$ be distributed according to the density \eqref{eq:beta_ensembles_density} with a potential $V$ satisfying the above hypotheses. Take $\nu=\mu_V$ in the definition \eqref{def:L_N}.  Then for any $\e,\kappa>0$ we have
\begin{align*}
&\mathbb{P}\left(\sup_{A+ \kappa < E < B -\kappa } \sqrt{\frac{\beta}{2}}\frac{\re L_N(E)}{\log N}\in[1-\e,1+\e]\right)=1-\oo(1),\\
&\mathbb{P}\left(\sup_{A+ \kappa < E < B -\kappa} \sqrt{\frac{\beta}{2}}\frac{\im L_N(E)}{\log N}\in[1-\e,1+\e]\right)=1-\oo(1).
\end{align*}
\end{theorem}

We next state the analogue of the first part of \Cref{thm:Wigner} for $\beta$-ensembles, which follows from the previous theorem.\footnote{An analogue of the rigidity scale with overwhelming probability,  i.e.  (\ref{eqn:optRig}) was already shown in \cite{BouModPai2020}.} 

\begin{corollary}\label{cor:Beta}
Under the same assumptions as Theorem \ref{thm:Beta} (in particular $\nu=\mu_V$ in (\ref{quantiles})), we have for every $\kappa >0$ that 
\[
\mathbb{P}\left(\max_{\kappa N \leq k\leq (1-\kappa)N} \pi\sqrt{\frac{\beta}{2}}\cdot \frac{\rho_V(\gamma_k)\, N\big(\lambda_k-\gamma_k\big)}{\log N}\in[1-\e,1+\e]\right)=1-\oo(1).
\]
\end{corollary}

\begin{remark}
In  Theorem \ref{thm:Wigner} and Corollary \ref{cor:Beta} the quantiles $\gamma_k$ can be replaced by $\E[\lambda_k]$,   therefore answering the rigidity question as stated in Problem 2. Indeed,  the bound $|\E[\lambda_k]-\gamma_k|=\OO(N^{-1})$ holds thanks to  \cite[Theorems 1.4 and 1.5]{LanSos2018}.
\end{remark}

\begin{remark}
  \label{rem-meso}
All results in this article have direct analogues for maxima on mesoscopic intervals which are  supported in the bulk of the spectrum,  and the proofs are the same up to notational changes.  
For example,  in the case of symmetric Wigner matrices,  for any deterministic interval $I=I(N)\subset[-2+\kappa,2-\kappa]$ such that $\log|I|/\log N\to-1+\alpha$, $\alpha\in(0,1)$,  and any $\e>0$ fixed,  we have
\begin{equation}\label{eqn:meso}
\mathbb{P}\left(\sup_{E\in I} \frac{\re L_N(E)}{\sqrt{2}\log N}\in[\alpha-\e,\alpha+\e]\right)=1-\oo(1).
\end{equation}
\end{remark}

\subsection{Related Works.}\
\label{s:related}
 The FHK conjectures were first stated in \cite{FyoHiaKea2012,fyodorov2014freezing} for 
 Haar-distributed 
unitary random matrices
and 
the Riemann zeta function.
See \cite{bailey2022maxima} for a recent survey. See also \cite{fyodorov2016distribution} for the case of the GUE.
While we do not discuss in this paper relations
with $\zeta$,  we remark that the FHK conjecture for it,
up to tightness of the analogue of the variable $Z_N$,
has been established in \cite{arguin2020fyodorov1,arguin2020fyodorov2} after
initial progress in \cite{Naj16,ArgBelBouRadSou2016,Har2019};
see \cite{Har2019II} for a survey.

On the random matrix side,  the sharpest results available are for 
the circular $\beta$-ensembles. The leading and second order terms, for
$\beta=2$, were computed in \cite{ArgBelBou2017} and
\cite{PaqZei2018}. For general $\beta$, the FHK conjecture up to tightness
of the random variable $Z_N$,
was obtained in \cite{ChhMadNaj2018}, and
the convergence was recently established in 
\cite{PaqZei2022}.
All these works rely on uncovering 
hierarchical structures in the spectra of random matrices, permitting the use of methods originally developed  
for branching processes \cite{Bra1978,Bra1983}.

For other ensemble of random matrices, much less is known. 
As demonstrated in \cite{lambert2020maximum,claeys2021much} and also used by 
us, 
obtaining 
the leading order of the FHK conjectures (more precisely, a lower bound on
the leading order) is closely 
related to proving convergence of powers of (a rescaled) version of   
the characteristic polynomial towards
the Gaussian multiplicative chaos (GMC);
we refer the reader to 
\cite{Web2015,nikula2020multiplicative,lambert2021mesoscopic,lambert2023law,BouFal2022}
for some works concerning the GMC for random matrices and further
results in this direction. 
For $\beta=2$, convergence toward the GMC of the characteristic polynomial
(using Riemann--Hilbert techniques) was obtained in \cite{BerWebWon2017}, 
in the so-called $L^2$ phase, which is not sufficient for obtaining leading
order information on the maximum. 
In the context of more general Hermitian matrices,
related results on the distribution of the characteristic polynomial of 
Gaussian $\beta$-ensembles (which again are not sufficient for controlling 
the maximum) were 
proved in \cite{augeri2020clt,lambert2020strong1,lambert2020strong2,ChaFahWebWon2021}. 

The fundamental reason one expects extreme values statistics such as (\ref{eqn:FHK}) and a limiting Gaussian multiplicative chaos from random matrices is that they lie in the class of {\it logarithmically correlated fields}.  
For $\beta$-ensembles and Wigner matrices,  this log-correlation has been proved in the sense of distributional convergence first, as follows e.g.\  from \cite{Joh1998,LytPas2009}, and more recently in the 
pointwise sense  \cite{BouMod2018,BouModPai2020}.  For \textit{Gaussian} log-correlated
fields, a rich theory
concerning the extremes is available, with the same universal scaling as in 
\eqref{eqn:FHK}. In particular, the fluctuations of the analogue of 
$Z_N$ are always of the form of two independent random variables, one
being Gumbel and the second depending on the long-range behavior of the covariance. We refer the reader to \cite{BolDeuGia01,Bis20,Zei16} for an account of the theory in the canonical case of the Gaussian free field (from leading order computation to convergence of the maximum and details on the process of extrema), 
and to
\cite{ding2017convergence} for the universal description of the limit. 
Extending the theory beyond the Gaussian case (where extra tools, including
comparison theorems, are available) toward its natural universality class
has been a major challenge and has attracted
a lot of recent activity. Beyond the models of
random matrices and the Riemann zeta function already discussed, 
we mention here the sine-Gordon model \cite{BauHof2022} (where a renormalization
procedures enables coupling to the Gaussian free field, yielding a full convergence result),
the cover time for planar random
walk \cite{DemPerRosZeo04,BelKis14,BelRosZei20} (where tightness
has been proved),  
the maxima of 
Ginzburg--Landau fields \cite{BelWu20},  
the maxima of characteristic polynomials
of permutation matrices \cite{CooZei20}, where at this time only
leading order information is available,  and the model of  two dimensional
random polymers \cite{CarSunZyg20,CosZei23}, where not even the
leading order convergence has been demonstrated.\\

\noindent We next turn to the topic of rigidity of eigenvalues, which
has a long history,  going back at least to 
\cite{de1995statistical}. The importance of obtaining
some a-priori rigidity estimates for Wigner matrices
was highlighted in 
\cite{ErdSchYau2011}, as part of their celebrated proof 
of the universality of spacing distribution for the Wigner ensemble. 
This work established the
upper bound  $|\lambda_k-\gamma_k|\le N^{-1/2-\varepsilon}$ for some $\varepsilon>0$.

Sharper estimates on rigidity for Wigner matrices
were obtained in the
seminal work \cite{ErdYauYin2012}, which bounded the fluctuations of the eigenvalues by $N^{-1+\eps}$ for every $\eps >0$, with overwhelming probability. 
This  result was then refined to show the bound ${\rm O}((\log N)^C/N)$ 
with overwhelming probability for some (potentially large) constant $C>1$ \cite{TaoVu2013,GotNauTikTim2013,cacciapuoti2015bounds}. 
To our knowledge, the sharpest result on rigidity
prior to this work is contained
in \cite{GotNauTikTim2018}, who obtained the
rigidity scale ${(\log N)^2}/{N}$. Our result on Gaussian decay far in the tail distribution, given in (\ref{eqn:sharpdecay}), is new even for the Gaussian ensembles.

A question related in spirit  
to the rigidity question is that of the maximal 
spacing between successive eigenvalues,
going back to a question of Diaconis \cite{Dia03}. For the maximal spacing of GUE and CUE matrices
matrices, the first order of the gap
was computed in \cite{BenBou2013}, and convergence of the rescaled maximal gap
was
established in \cite{FengWei2018II}. Both of these works 
use determinantal methods. 
Universality and comparison results were obtained more recently in 
\cite{Bou2018} and \cite{LanLopMar2018}.

Finally,  many other aspects of extreme value theory for random matrices have been very active recently.  Notably we refer to important progress on the spectral radius of non-Hermitian random matrices,  an example where universal fluctuations are known (which are not in the log-correlated universality class, see \cite{CipErdWu2023} and the references therein).

\subsection{Proof Ideas.}\ \label{s:ideas}
{We now describe the main contours of the proofs in this paper. 
Even though our presentation of the main results starts with Wigner matrices, 
we describe the proofs first for $\beta$-ensembles (see \Cref{thm:Beta}), 
since the 
Wigner case is then based on a 
comparison  which takes as input the results
for the GOE and GUE.  For concreteness,  we focus
on the description of the
proofs for $\Re L_N(E)$.

{The standard approach for estimating $\sup_{A+\kappa<E<B-\kappa} \Re L_N(E)$
from above
has two components: first, one replaces the supremum over
the interval $[A+\kappa,B-\kappa]$ by a maximum over a finite collection
of $E_i$'s, of spacing of order $1/N$. That this is enough has already been
shown e.g.\ in  \cite[Corollary  5.4]{LamPaq2019} (based on an idea from
\cite{ChhMadNaj2018}).\footnote{We will actually use a different method, that
applies also to $\Im L_N$ and also allows one to work with mesoscopic
intervals as in Remark \ref{rem-meso}. Our method builds on local laws up to microscopic scales
from \cite{BouModPai2020}.} After achieving this reduction,
one uses a union bound together with a tail estimate on the law 
of $\Re L_N(E)$ with $E$ fixed and deterministic. In particular,
one needs to control exponential moments of the latter variable. 
Unlike the case treated in \cite{LamPaq2019}, we do not have at our disposal
an integrable structure, and so explicit computations are not possible. Instead,
we would like 
to use  exponential estimates from
\cite{BouModPai2020} (in an improved form described in
\Cref{thm:GaussFluct}).}

{Unfortunately, the estimates in \Cref{thm:GaussFluct} do not apply directly for
$E_i$, but rather only for $E_i +\ii \eta_0$ where $\eta_0=(\log N)^{1000}/N$ 
is as in \eqref{eqn:etao1}. Because of that, we need to modify the above
procedure and first move away from the real axis. Continuity arguments allow
one to move to distance $1/N$ from the real line; to go beyond that, we need
to use the very precise
local law with Gaussian tail from \cite[Remark 2.4]{BouModPai2020} 
(which, after
integration, give control on $L_N(E_i+\ii \eta_0)$). 
This recent theorem provides essentially optimal bounds on the centered 
moments of the Stieltjes transform on all scales $\Im z >0$, and is crucial for our work; weaker estimates, such as those available in
the previous literature, would not have sufficed.} 

{For the lower bound, due 
to the log-correlated structure of the field $L_N(\cdot)$, 
one could follow methods based on second moment analysis, including the insertion
of appropriate barriers, as described e.g.\ in \cite{Zei16}
for the Gaussian setup. There are several obstacles to that approach, 
including the need to obtain very precise decoupling inequalities for pairs of
macroscopically separated energies  $E_i$, based on Fourier transforms.
Instead, we use the GMC approach (introduced in similar contexts
in \cite{lambert2020maximum,claeys2021much}). Here again, the
proof starts with the preliminary step of moving the problem off the real line 
and into the upper half plane (by distance $\eta_0$), 
in order to improve the regularity of $L_N$;
this step is achieved using 
a Poisson integral representation of the harmonic function 
$\rlog_z(x)={\rm Re}\log(z-x)$. Then, in the main technical step, we demonstrate that for every $\gamma \in ( -\sqrt{2}, \sqrt{2})$, the random field 
with density
\vspace{-0.2cm}
\begin{equation}\label{e:gmc0}
F(E) = \frac{e^{\sqrt{\beta}\gamma \Re \tilde L_N(E+\ii\eta_0)}}{\E[e^{\sqrt{\beta}\gamma\Re \tilde L_N(E+\ii\eta_0)}]} 
\end{equation}

\vspace{-0.2cm}
\noindent with respect to Lebesgue measure, 
converges to a Gaussian multiplicative chaos as $N$ grows. 
Here $\tilde L_N$ is an appropriate centering of $L_N$. 
Following the general criteria in 
\cite{claeys2021much}, this again follows from the controls provided
by 
\Cref{thm:GaussFluct}. Once convergence to GMC has been achieved, the required
lower bound follows (essentially because the GMC is supported on points $E$
with 
$\Re L_N(E+\ii \eta_0)> \sqrt{\beta}\gamma/2-\delta$).}  \\

\noindent {We now turn to the proof of our result on the log-characteristic 
polynomials of Wigner matrices, \Cref{thm:maxWigner}. This is fundamentally 
a universality result, stating that results established in \Cref{thm:Beta} for 
the GOE/GUE does not depend on the distribution of the matrix entries. We adopt a dynamical approach to this question, in line with the general framework that has been developed  to resolve the Wigner--Dyson--Mehta conjecture and other problems regarding the universality of local spectral statistics \cite{erdos2017dynamical}. 

Our primary input is a method to couple characteristic polynomials.
We consider the matrix-valued stochastic differential equation
\vspace{-0.3cm}
\begin{equation}\label{OUintro}
\rd H_{t}=\frac{1}{\sqrt{N}} \rd B_t-\frac{1}{2}H_t\,\rd t
\end{equation}

\vspace{-0.1cm}
\noindent with initial data $H_0$ given by a Wigner matrix, where $B_t$ is a matrix of Brownian motions  that are independent up to the symmetry $B_{ij} = B_{ji}$. The dynamics are chosen so that if $H_0$ is a GOE, then its distribution remains invariant for $t>0$. It is well known that if the eigenvalues $(\lambda_i(t))_{i=1}^N$ of $H_t$ evolve according to the Dyson Brownian motion, given by 
\Cref{eqn:eigenvaluesSymmetric}:
\vspace{-0.3cm}
\begin{align}
\rd\la_k=\frac{\rd \beta_{k}}{\sqrt{N}}+\left(\frac{1}{N}\sum_{\ell\neq k}\frac{1}{\la_k-\la_\ell}-\frac{1}{2}\lambda_k\right)\rd t,\label{eqn:eigenvaluesSymmetricintro}
\end{align}
%
where the  $\{ \beta_k \}_{k=1}^N$ are independent, standard Brownian motions. To enforce a coupling, we let $(\mu_i(t))_{i=1}^N$ be a solution to \eqref{eqn:eigenvaluesSymmetricintro} with the same choice of driving Brownian motions with the initial data $t=0$ given by a GOE. Then the process $(\lambda_i(t) - \mu_i(t))_{i=1}^N$ satisfies a deterministic system of differential equations, which may be studied in detail using homogenization and the method of characteristics \cite{BouErdYauYin2016,Bou2018}. Our main result on coupling, \Cref{l:relax},
is the estimate is that for any $z=E+ \iu \eta$ with $\eta \in (N^{-1}, 1)$, any time $t > \exp(-\bar C_0(\log\log N)^2)$ for appropriate
$\bar C_0$, we have 
\vspace{-0.1cm}
\begin{equation}\label{eqn:relaxintro}
\mathbb{P}\left(\max_{-2+\kappa<E<2+\kappa}
\left|\sum\log\big(z-\mu_k(t)\big)-\sum\log\big(z-\la_k(t)\big)\right| > (\log N)^{1/2+\e}\right)=\oo(1).
\end{equation}

\vspace{-0.1cm}
 \noindent The crucial point here is that while $t$ is relatively
large, we are able
to approach the real line up to the microscopic distance $1/N$, which is precisely the
distance beyond which deterministic arguments do not yield  control on
the difference between
$L_N(E)$ and $L_N(E+\ii \eta)$. Note that the scale $1/N$ is below
the scale of rigidity (which is of order $\log N/N$, as we prove in 
\Cref{thm:Wigner}). The ability to nevertheless perform the coupling 
(using in a crucial way  overcrowding estimates from \cite{Ngu2018} 
and a-priori suboptimal 
rigidity estimates from \cite{ErdYauYin2012}) goes significantly beyond the earlier dynamics-based
coupling of characteristic polynomials. See e.g.\  \cite{Bou2018} for the existing sharpest result which requires ${\Im z}\geq N^{-1+\e}$.

{Recalling that \eqref{eqn:eigenvaluesSymmetricintro} is the eigenvalue evolution under \eqref{OUintro}, and that the desired result for the GOE follows from \Cref{thm:Beta}, 
we see that \eqref{eqn:relaxintro} implies matrices of the form $\sqrt{1-t} H + \sqrt{t} W$ satisfy the conclusions of \Cref{thm:maxWigner}, where $H$ is  Wigner matrix, $W$ is a GOE, and $t$ decays sufficiently slowly. 
It remains to extend the result from these \emph{weakly Gaussian-divisible} matrices to the entire Wigner class. For this, we use a standard comparison argument based on four-moment matching \cite{TaoVu2011}. It is well known that this technique shows that weakly Gaussian-divisible matrices are ``dense'' in the set of all Wigner matrices, in these sense that universality for sufficiently regular observables follows from establishing the expected behavior in the weakly Gaussian-divisible case.  While the extremal statistic $\sup_{|E|<2-\kappa}{\rm Re}L_N(E)$ is non-local and not regular enough to directly apply results from the literature,   ideas originally developed in \cite{LanLopMar2018} permit the comparison to proceed, and  complete the proof of \Cref{thm:maxWigner}.}

For Theorem \ref{thm:maxWignerDivisible}, no density argument is needed but the relaxation step becomes particularly delicate as it needs to reach the tightness precision.  
Even worse,  the maximum of the characteristic polynomials differences considered in (\ref{eqn:relaxintro}) is probably not tight as $N\to\infty$, even for $t\asymp 1$. The main insight consists in proving that
$\sum\log\big(E-\mu_k(t)\big)$ is very close to $\sum\log\big(E+X_N-\la_k(t)\big)$, up to error of order $1$  where $X_N$ is a {\it random shift}. This shift is small enough so that
it only changes the size of the centering of the log-characteristic polynomial by an order 1.  Choosing for $E$ the location of the maximum for GOE completes the relaxation
of the maximum, which actually requires many other ingredients as explained in the proof of Proposition \ref{l:relax2}.\\

\noindent {We now turn to the proofs of rigidity. While the best previous rigidity for Wigner matrices was proved directly by resolvent methods \cite{ErdYauYin2012}, as a precursor to the local study of Dyson Brownian motion, we reverse this usual logic and derive optimal rigidity as a consequence of refined estimates on the local dynamics.

The estimates with optimal constant, \Cref{thm:Wigner} (i) and \Cref{cor:Beta}, are  equivalent to  the corresponding results for $\Im L_N(E)$, \Cref{thm:maxWigner} and \Cref{thm:Beta},  see (\ref{eqn:equivRig}) for this classical equivalence. 
{On the other hand, obtaining \Cref{thm:Wigner} (ii), which asserts rigidity for Wigner matrices with overwhelming probability, requires further novelties. The traditional four-moment comparison method is effective only for statements that hold with probability $1- N^{-c}$, and therefore does not provide density of weakly Gaussian-divisible matrices for  \Cref{thm:Wigner} (ii).  
However,  
iterative comparisons of moments of linear statistics  have appeared in random matrix theory in \cite{KnoYin2013,TaoVu2013},  which were recently strengthened
in the context of eigenvector statistics towards comparison beyond the order of magnitude, up to optimal constants  \cite{BenLopFluc2022,BenLop2022}.  We adapt this method to obtain the sharp Gaussian decay
in (\ref{eqn:sharpdecay}) from the case of weakly Gaussian-divisible matrices.  For this ensemble,  we use our coupling \eqref{eqn:relaxintro}, along with estimates specific to the GOE/GUE, to provide an optimal-order upper bound on the large moments (of order $\log N$) of the eigenvalues  counting function (Lemma \ref{l:newcomparisonGD}).
With this estimate in hand for weakly Gaussian-divisible matrices, we use an inductive moment comparison (see \Cref{l:newcomparison}), to obtain a similar estimate on the $\log N$ moment for arbitrary Wigner matrices. The desired rigidity result,
and the precise tail bounds in \Cref{thm:Wigner}, 
then follow by Markov's inequality.}\\

\noindent To conclude this section on the developed methods,  we mention that the upcoming work \cite{CipolloniLandon}  obtains the analogue of Theorem \ref{thm:maxWigner}
for non-Hermitian matrices with independent entries,  with an approach relying on  Fourier transforms of linear statistics instead of GMC and dynamics (the special case of Ginibre matrices was proved,  also at leading order, in  \cite{lambert2020maximum}). While this approach is particularly robust for the leading order of the characteristic polynomial and likely applies to Theorem \ref{thm:maxWigner},  the dynamical method seems essential to our results on tightness (Theorem \ref{thm:maxWignerDivisible}) and rigidity with overwhelming probability (Theorem \ref{thm:Wigner}, part (ii)).

\subsection{Further Comments.}\
Since this paper is already long and uses a multitude of tools, we
have not discussed the edge of the spectrum, nor have we treated the 
case of $\beta$-ensembles with non-analytic potential.  These extensions require
work but seem within the reach of the tools developed here.

It is natural to expect that \eqref{eqn:prob1} holds. For Gaussian
$\beta$-ensembles, the Jacobi representation in terms of 
independent variables, 
used in \cite{augeri2020clt,lambert2020strong1,lambert2020strong2}, could potentially be useful, as it was in the C$\beta$E case \cite{ChhMadNaj2018,PaqZei2022}.
In view of 
Theorem \ref{thm:maxWignerDivisible}, any progress in that direction
for $\beta=1,2$
would immediately  translate to Gaussian divisible ensembles, at least at the level of tightness.
More generally,
we expect that the ideas developed here will be a 
useful basis for work on the higher order terms in Problem 1,
or for studying
FHK-type asymptotics and rigidity for other ensembles, 
including matrices of general Wigner type, 
adjacency matrices of random graphs, models arising in free probability,
and non-Hermitian matrices. 

In view of the first universal results on FHK asymptotics for Wigner matrices and $\beta$-ensembles,  another natural question  concerns universal limiting measures for random characteristic polynomials.  We also expect that some methods from this paper would help  towards the convergence of  $|\det (E-H)|^\gamma\rd E$ to a Gaussian multiplicative chaos.

\subsection{Organization.}\ \label{s:outline} 
\Cref{s:preliminary} fixes our notation conventions and states essential results from previous works. In \Cref{sec:maxgauss}, we prove \Cref{thm:Beta} and \Cref{cor:Beta}. 
\Cref{sec:Relax} studies the short-time relaxation to equilibrium of Dyson 
Brownian motion for FHK-type observables, and \Cref{sec:MomMat} provides moment matching lemmas for these observables. The results in these sections are then used in \Cref{s:wigmax} to prove \Cref{thm:maxWigner} and Theorem \ref{thm:Wigner}. 
\Cref{sec:moments} establishes \Cref{p:moments}, which controls diverging moments of the Stieltjes transform of Wigner matrices; these are used
in \Cref{sec:Relax}. 
\Cref{sec:gauss} proves \Cref{thm:GaussFluct}, on the Fourier--Laplace transform of the log-characteristic polynomial of $\beta$-ensembles near the real axis; this is used in \Cref{sec:maxgauss}.

\begin{remark} Throughout, we suppress the dependence of the constants in our arguments on the constants in Definition~\ref{def:wig} and the potential $V$  from \eqref{eq:beta_ensembles_density}, since this dependence never affects our arguments. One could give explicit (suboptimal)  error bounds in all our results in terms of these parameters, but for simplicity, we do not pursue this direction. Additionally, for brevity, we prove our results for real symmetric Wigner matrices, since the complex Hermitian case differs only in notation.
\end{remark} 

\subsection{Acknowledgments.}\ P.\ B.\ was supported by the NSF grant DMS-2054851.  P.\ L.\ was supported by NSF postdoctoral
fellowship DMS-2202891.  O.\  Z.\  was supported by Israel Science Foundation grant number 421/20.

\section{Preliminary Results}\label{s:preliminary}

We begin by recalling some fundamental concepts and notation. For deterministic sequences $X=X_N$ and $Y=Y_N>0$, we write $X = O(Y)$ if there exists a constant $C>1$ such that $|X| \le CY$ for all $N \ge 1$, and $X = o(Y)$ if $ \lim_{N\rightarrow \infty} X/Y = 0$.
We let $\mathbb{H} = \{ z : \Im z > 0 \}$ denote the complex upper half plane, and often use the notation $z = E + \I \eta$ for $z \in \mathbb{H}$, so that $E$ and $\eta$ stand for the real and imaginary parts of $z$, respectively.
We often identify the complex plane $\C$ with $\R^2$, and use the notation $[E_1, E_2] \times [\eta_1, \eta_2]$ to denote the set $\{z \in \C : \Re z \in [E_1, E_2] , \Im z \in [\eta_1, \eta_2] \}$. 
Our convention is that $\mathbb{N}$ denotes the set $\{1,2,\dots \}$.
The function $\log$ always denotes the natural logarithm.  

We will frequently define constants that depend on some number of parameters. These will be introduced as $C(x_1, \dots, x_n)$, for parameters $x_1, \dots, x_n$, and subsequently referred to as $C$ (suppressing the dependence on the parameters in the notation). These constants may change line to line without being renamed (while retaining the dependence on the same set of parameters). We usually write $C>1$ for large constants, and $c>0$ for small constants.

For $z=E+\ii\eta \in \mathbb{H}$ we use the notations 
\begin{equation}\label{eqn:logz}
{\rm Re}\log(z-\lambda)=\log|z-\lambda|,\ \  
{\rm Im}\log(z-\lambda)=\frac{\pi}{2}+\arctan\frac{\lambda-E}{\eta},
\end{equation}
which are coherent with our convention $\log(r e^{\ii \theta}) = \log(r) + \ii \theta$ for any $r>0$ and $\theta \in (-\pi,\pi]$.

We recall that Wigner matrices were defined in \Cref{def:wig}.  The Gaussian Orthogonal (resp.\  Unitary) Ensemble of dimension $N$, $\GOE_N$ (resp.\ $\GUE_N$), 
is defined as the $N\times N$ Wigner matrix with independent  entries $H_{ij}$,  $i\leq j$,  such that $\sqrt{N} H_{ij}$ is a real (resp. complex) Gaussian random variable with mean zero and variance $1 + \one_{i\neq j}$ (resp.  $\re H_{ij}$ and $\im H_{ij}$ are independent, each with variance $(1 + \one_{i\neq j})/2$).

We say that an event $\mathcal F = \mathcal F(N)$ holds with overwhelming probability if for any $D>0$, there exists a constant $C(D)>1$ such that $\P( \mathcal F^c) \le C N^{-D}$. 

\subsection{Semicircle law.}\ 
Let $\matn$ denote the set of $N\times N$ real symmetric matrices. Given $M \in \matn$, we index the eigenvalues $\lambda_i$ of $M$ in increasing order: $\lambda_1 \le \lambda_2 \le \dots \le \lambda_N$. 
The resolvent of $M$ is defined as by $G(z) = ( M - z \Id)^{-1}$  for $z \in \mathbb{H}$. 
The Stieltjes transform of $M$ is defined for $z \in \mathbb{H}$ by
\begin{align}\label{e:st}
m_N (z) = \frac{1}{N} \tr G = \frac{1}{N} \sum_i \frac{1}{ \lambda_i - z},
\end{align}
and  Stieltjes transform of the semicircle law is given by
\begin{align}\label{e:msc}
\msc (z) = \int_{\R} \frac{\rho_{\rm sc} (x)\, \rd x}{ x - z } = \frac{-z + \sqrt{z^2 -4}}{2}.
\end{align}
Here $\sqrt{z^2-4}:=\sqrt{z-2}\sqrt{z+2}$ is 
defined through the principal branch of the square root, extended to negative real numbers by
$\sqrt{-x}=\ii\sqrt{x}$ for $x>0$.

We now recall some elementary bounds on $\msc (z)$.

\begin{lemma}[{\cite[Lemma 6.2]{erdos2017dynamical}}]\label{l:sc}
There exists a constant $c>0$ such that the following holds. For all $z = E + \iu \eta$ such that $E \in [ -10, 10]$ and $\eta \in (0, 10]$,
\beq
c \le | \msc(z) | \le 1 - c \eta.
\eeq
Set $\kappa = | |E| - 2 |$. If $|E| \le 2$, then 
\beq
c\sqrt{\kappa + \eta} \le \Im \msc(z) \le c^{-1} \sqrt{\kappa + \eta}.
\eeq
\end{lemma}

\subsection{Wigner matrices.}\
We recall the following fundamental estimates on Wigner matrices from \cite[Theorem 2.1]{ErdYauYin2012} and \cite[Theorem 2.2]{ErdYauYin2012}.
In this theorem and throughout this paper, we will often use the control parameter
\begin{equation}\label{eqn:phi2}
\varphi=\exp\left(C_0(\log\log N)^2\right),
\end{equation}
where $C_0>0$ is a constant depending only on the constant $c$ from (\ref{eqn:tail}), whose value is fixed by the following lemma.

\begin{theorem}\label{l:lscl}
Let $H$ be a Wigner matrix. Then there exists $C_0>0$ such that the following claims hold.
\begin{enumerate}
\item There exists $c>0$ such that 
\begin{equation}\label{e:sclaw}
\P\left( \bigcup_{z \in \mathbb{H}} \left\{ |m_N(z) - \msc(z)| \ge \frac{\varphi}{N\eta} \right\} \right)\le c^{-1} \exp\left( - \varphi^c \right)
\end{equation}
and
\begin{equation}\label{e:gentries}
\P\left( \bigcup_{z \in \mathbb{H}} \left\{ \max_{i,j \in \unn{1}{N}} | G_{ij}(z) - \delta_{ij} \msc(z)| \ge \varphi \sqrt{\frac{\Im \msc(z)}{N\eta}} + \frac{\varphi}{N\eta} \right\} \right)\le c^{-1} \exp\left( -  \varphi^c \right).
\end{equation}
\item There exists $c>0$ such that, defining $\hat k=\min(k, N + 1-k)$ and $\gamma_k$ as in (\ref{quantiles}) with $\nu=\sc$, 
\begin{equation}\label{e:rigidity}
\P\left(\exists k \in \unn{1}{N} : | \lambda_k - \gamma_k|
 \ge \varphi\, \hat k^{-\frac{1}{3}} N^{-\frac{2}{3}} \right)\le c^{-1} \exp\left( -  \varphi^c \right).
\end{equation}

\end{enumerate}
\end{theorem}
\begin{remark}
In \cite{ErdYauYin2012}, \eqref{e:sclaw} and \eqref{e:gentries} were shown for $z$ in a compact spectral domain.
The extension to all $z \in \mathbb{H}$ follows from \cite[Section 10]{benaych2016lectures}.
\end{remark}

\subsection{Generic $\beta$-ensembles.}\
We recall that $\beta$-ensembles were defined above in \eqref{eq:beta_ensembles_density}, and we retain the notation from the previous section.
For $z\in \mathbb{H}$, let
\begin{equation}\label{e:szdef}
	s(z)= s_N(z) = \frac{1}{N} \sum_{k=1}^N \frac{1}{\lambda_k - z},\ \ \ 
	m_V(z) = \int_{\mathbb{R}}\frac{\rd \mu_V(x)}{x-z}.
\end{equation}

The following will be key to the proof of Theorem \ref{thm:Beta}. It follows from \cite[Remark 2.4]{BouModPai2020}
\begin{theorem} \label{th:local_law_bulk} Under the assumptions  (A1), (A.2) (i),  (A3) and (A4)  there exist constants $C, c , \tilde{\eta}>0$  such that for any $q \geq 1$, $N \geq 1$ and  $z = E + \ii \eta$ with $0 < \eta \leq \tilde{\eta}$ and $A-\eta \leq E \leq B+\eta$, we have
\[
	\E\big[ \absa{ s(z) - m_V(z) }^q \big]
	\leq \frac{(Cq)^{q/2}}{(N\eta)^q} 
		+ \frac{C^q e^{-cN}}{\abs{z-A}^{q/2} \abs{z-B}^{q/2}}.
\]
\end{theorem}

We will also use the following rigidity estimate, which directly follows from  \cite[Lemma 3.8]{BouModPai2020}.

\begin{lemma}\label{lem:rig}
 There exists $c(V)>0$ such that for any $N\geq 1$ and $k\in\llbracket 1,N\rrbracket$ we have
\[
\mu\left(|\lambda_k-\gamma_k|>N^{-\frac{2}{3}}(\widehat k)^{-\frac{1}{3}}(\log N)^{23}\right)\leq c^{-1}e^{-c(\log N)^5}.
\]
\end{lemma}

\subsection{Resolvent identities.}\
For $M \in \matn$ and any differentiable $f : \RR \rightarrow \RR$, we set
\begin{equation}\label{e:derivdef}
\partial_{ij} f(H) = \frac{d}{dt}\bigg\rvert_{t=0} f\left(H + t \Delta^{(ij)}\right),
\end{equation}
where $\Delta^{(ij)} \in \matn$ is the matrix whose entries are zero except in the $(i,j)$ and $(j,i)$ positions, in which case they equal one: $\Delta^{(ij)}_{kl} = (\delta_{ik} \delta_{jl} + \delta_{jk} \delta_{il} )(1 + \delta_{ij})^{-1}$.

Given $M, \widetilde M \in \matn$, with resolvents $G$ and $\widetilde G$, respectively, it follows immediately from the definitions that
\begin{equation}\label{resolventexpansion}
G - \widetilde G = G ( \widetilde M - M ) \widetilde G.
\end{equation}
Additionally,  the following resolvent identities are well known. The first can be found as \cite[(3.6)]{benaych2016lectures}. The second is a straightforward consequence of \eqref{resolventexpansion} and the definition \eqref{e:derivdef}. The third is an immediate consequence of the spectral decomposition of $G(z)$ \cite[(2.1)]{benaych2016lectures}.

\begin{lemma} Given $M \in \matn$, let $G(z) = ( M - z \Id)^{-1}$ denote its resolvent. 
\begin{enumerate}
\item For any $i\in \unn{1}{N}$,  
\begin{equation}\label{e:ward}
\sum_{1 \le j \le N} |G_{ij}|^2 = \frac{ \Im G_{ii}}{\Im z}.
\end{equation}
\item For $i,j,k,l \in \unn{1}{N}$,
\begin{equation}\label{e:gpartial}
\partial_{kl}G_{ij} = - (G_{ik} G_{lj} + G_{il} G_{kj}) (1 + \delta_{kl})^{-1}.
\end{equation}
\item For $i,j\in \unn{1}{N}$, 
\begin{equation}\label{e:gtrivial}
\left| G_{ij}(z) \right| \le \eta^{-1}.
\end{equation}
\end{enumerate} 
\end{lemma}

\subsection{Eigenvalue overcrowding.}\ We recall the following overcrowding estimate.  It is contained in \cite[Theorem 1.12]{Ngu2018}, which is stated in greater generality for Wigner matrices with subgaussian entries, but we need only the special case of the Gaussian ensemble.

\begin{theorem}\label{thm:Hoi}
Let $\{ \mu_i \}_{i=1}^N$ denote the eigenvalues of $\GOE_N$.
For any $0<\gamma<1$ there exist constants $C(\gamma),c(\gamma),\gamma_0(\gamma)>0$ such that for any $k\in\llbracket \gamma_0^{-1},\gamma_0N\rrbracket$, $\e>0$ and $E\in\mathbb{R}$, we have
\begin{equation}
\mathbb{P}\left(\left|\left\{\mu_i\in\left[E-\frac{\e k}{N},E+\frac{\e k}{N}\right]\right\}\right|\geq k\right)\leq (C\e)^{\frac{1}{2}(1-\gamma)k^2}+e^{-N^c}.
\end{equation}
A similar bound holds for $\GUE_N$, with $\frac{1}{2}(1-\gamma)k^2$ replaced by $(1-\gamma)k^2$.
\end{theorem}

\section[Maximum for Log-Gases]
{Maximum for Log-Gases}
\label{sec:maxgauss}
This section contains the proofs of \Cref{thm:Beta} and \Cref{cor:Beta}. The proof of \Cref{thm:Beta} is contained in the first three subsections, and the proof of \Cref{cor:Beta} is given in \Cref{sec:elementary}. 

We give the
proof of Theorem \ref{thm:Beta} here in the case where we suppose  (A.2) (i) holds.
Under the assumption (A.2) (ii),  some extra care is needed such as working under conditional measures. We explain the necessary changes in \Cref{s:otherpotentials}.

\subsection{Upper bound for the real part.}\  
By monotonicity of $\eta \mapsto \log|E+\ii\eta-\lambda|$,  $\eta>0$, and the estimate $\int\log|E-\lambda|\rd\rho_V(\lambda)=\int\log|E+\ii\e-\lambda|\rd\rho_V(\lambda)=\OO(\e)$ uniformly in $E$, 
there exists a fixed $C>0$ such that
\begin{equation}\label{eqn:UBRe1}
\sup_{E\in [A+\kappa,B-\kappa]}\Re L_N(E)\leq \sup_{E\in [A+\kappa,B-\kappa]}\Re L_N\left(E+\frac{\ii}{N}\right)+C.
\end{equation}
Let $J=[A+\kappa,B-\kappa]\cap N^{-1-c}{\mathbb{Z}}$, where $c>0$ is an arbitrary small constant.  For any $E\in[A+\kappa,B-\kappa]$, let $E'$ be the closest point in $J$,  $z=E+\frac{\ii}{N}$ and $z'=E'+\frac{\ii}{N}$.  
Then from $\log (1+\e)=\e+\OO(\e^2)$ and recalling the definition of $s(z)$ from \eqref{e:szdef},  this implies (in this section we abbreviate $m=m_V$)
\begin{multline}\label{eqn:Taylor}
\Re L_N(z)-\Re L_N(z')=\OO((z-z')N(s(z')-m(z')))+\OO\left((z-z')^2\sum\frac{1}{|z'-\lambda_i|^2}\right)+\OO(1)\\
=N^{-c}\OO(|s(z')-m(z')|)+N^{-2c}\OO(\Im s(z'))+\OO(1).
\end{multline}
Next,  \Cref{th:local_law_bulk} together with Markov's inequality gives 
\begin{equation}\label{eqn:tailnew}
\max_{E\in[A+\kappa,B-\kappa]} \P \left( |s(z)-m(z)| > (\log N)^{7/10} \right) \le N^{-200}.
\end{equation}
for large enough $N$.  Together with the 
boundedness of $m_V$ on compact sets of $\mathbb{C}$ (see \eqref{eqn:2mplusV}), 
this gives
\begin{equation}\label{eqn:microBound}
\mathbb{P}\left(\exists E'\in J: |s(z')|\geq (\log N)^{7/10}\right)\le N^{-100}. 
\end{equation}
We conclude that 
\begin{equation}\label{eqn:UBIm1bis}
\mathbb{P}\left(\sup_{E\in [A+\kappa,B-\kappa]}\Re L_N(E)\leq \sup_{E\in J}\Re L_N(E+\ii N^{-1})+(\log N)^{9/10}\right)\geq
1-\OO(N^{-100}). 
\end{equation}

We now control the increments of $L_N$ along the line segment $\{\Re z = E, N^{-1}<\Im z<\eta_0\}$  using Markov's inequality, where we  set 
\begin{equation}\label{eqn:etao1}
\eta_0=\frac{(\log N)^{1000}}{N}
\end{equation}
throughout this section.
  For 
$E\in J$, we denote $z=z(E)=E+\ii/N$ and $\tilde z=E+\ii \eta_0$. Then for any fixed $\e>0$ and $p\in \N$, we have by a union bound that
\begin{multline}\label{e:ubound}
\mathbb{P}\left(\exists E\in J:\Re L_N(z)>\Re L_N(\tilde z)+\e\log N\right)\\
\leq C N^{1+c}(\e\log N)^{-2p}\max_{E\in J}\E\left(\int_{[N^{-1},\eta_0]^{2p}}\prod_{i=1}^p(N(s-m_V)(E+\ii\eta_i))\prod_{i=p+1}^{2p}(N\overline{(s-m_V)}(E+\ii\eta_i))\right)\rd\eta_1\dots\rd\eta_{2p}.
\end{multline}

We now suppose that $p = O( N (\log \log N)^{-1} )$. 
Theorem \ref{th:local_law_bulk} gives, for $E\in[A+\kappa,B-\kappa]$,
\begin{equation}\label{eqn:ttLocalLaw}
\E[(|(s-m_{V})(E+\ii\eta)|^{p}]\leq 
\frac{(Cp)^{p/2}}{(N\eta)^p}
			+ C^p e^{-\tilde c N}\leq \frac{(Cp)^{p/2}}{(N\eta)^p},
\end{equation}
for some $C=C(V,\kappa)$,
where the latter inequality holds because we assume $\eta<\eta_0$. 
Equation (\ref{eqn:ttLocalLaw})  and H{\"o}lder's inequality give
\begin{equation}\label{eqn:tLocalLaw}
\E\left[\prod_{i=1}^p\big|N(s-m_{V})(E_k+\ii\eta_i)\big|\prod_{i=p+1}^{2p}|N\overline{(s-m_{V})}(E_k+\ii\eta_i)|\right]\leq   (Cp)^{p}\prod_{i=1}^{2p}\frac{1}{\eta_i}.
\end{equation}
Inserting the previous display in \eqref{e:ubound}, we obtain
\begin{equation}\label{eqn:UBRe2}
\mathbb{P}\left(\exists E\in J:\Re L_N(z)>\Re L_N(\tilde z)+\e\log N\right)\leq \frac{N^{1+c} (Cp)^{p}(\log\log N)^{2p}}{(\e\log N)^{2p}}\leq N^{1+c}\frac{(Ap)^{p}(\log\log N)^{2p}}{(\log N)^{2p}}\leq N^{-100},
\end{equation}
where $A$ is a new constant depending on $C$ and $\e$, and the latter inequality is obtained by setting $p=B\frac{\log N}{\log\log N}$ for sufficiently large $B$.
We note that for the above reasoning,  the Gaussian-like moment growth $(Cq)^{q/2}$ in Theorem \ref{th:local_law_bulk} is crucial (as opposed to an exponential-like growth of $(Cq)^{q}$) . 

Moreover, from Markov's inequality and Theorem \ref{thm:GaussFluct},  for any fixed $\lambda >0$ we have 
\begin{multline*}
\mathbb{P}\left(\exists E\in J:\Re L_N(\tilde z)>(1+\e)\sqrt{\frac{2}{\beta}}\log N\right)\leq N^{1+c}\, e^{-\lambda (1+\e)\sqrt{\frac{2}{\beta}}\log N} \max_{E\in J}\E\big[e^{\lambda \Re L_N(\tilde z)}\big]\\
\leq C\, N^{1+c}\, \max_{E\in J}e^{\frac{\sigma(\lambda,0,\tilde z)}{2}+\mu(\lambda,0,\tilde z)-\lambda(1+\e)\sqrt{\frac{2}{\beta}}\log N},
\end{multline*}
where we refer to (\ref{eqn:sigma}) and (\ref{eqn:shift}) for the definitions of $\sigma$ and $\mu$.
From Lemma \ref{lem:varshift}, we have $\mu(\lambda,0,\tilde z)=\OO(1)$ 
and $\sigma(\lambda,0,\tilde z)=(1+\oo(1))\lambda^2\frac{\log N}{\beta}$
uniformly in $N$, $E\in J$,  and $\lambda$ in any compact subset of $\mathbb{R}_+$.  Choosing $\lambda=\sqrt{2\beta}$ this implies that 
\begin{equation}\label{eqn:UBRe3}
\mathbb{P}\left(\exists E\in J:\Re L_N(\tilde z)>(1+\e)\sqrt{\frac{2}{\beta}}\log N\right)\leq e^{-(2\e-c-\oo(1))\log N}\to 0.
\end{equation}
With the choice  $0<c<2\e$,  equations (\ref{eqn:UBIm1bis}),  (\ref{eqn:UBRe2}), (\ref{eqn:UBRe3}) conclude the proof that
\[
\mathbb{P}\left(\exists E\in[A+\kappa,B-\kappa] :\Re L_N(E)>(1+\e)\sqrt{\frac{2}{\beta}}\log N+2\e\log N\right)\to 0.
\]

\subsection{Upper bound for the imaginary part.}\ 
The proof of the upper bound for $\Im L_N$ the same as the one for the real part up to the following  complication: There is no analogue of (\ref{eqn:UBRe1}), as $\eta \mapsto \sum_i\Im \log (E+\ii\eta-\lambda_i)$ is not monotone.
To circumvent this problem,  we observe that the error made by shifting $\Im \log$ from the real axis to a scale $\eta$ can be bounded in terms of a linear combination of the real and imaginary parts of the Stieltjes transform at scale $\eta$.

More precisely,  note that,  from $\arctan(x)-\arctan(+\infty)=-\int_x^\infty\frac{\rd u}{1+u^2}=-\frac{1}{x}+{\rm O}(\frac{1}{x^2})$ as $x\rightarrow \infty$, we have
\[
\arctan((\lambda-E)/\eta)-\frac{\pi}{2}\cdot \sgn(\lambda-E)=- \frac{\eta(\lambda-E)}{(\lambda-E)^2+\eta^2}+{\rm O}\left( \frac{\eta^2}{(\lambda-E)^2+\eta^2}\right).
\]
As a consequence,  for $z=E+\frac{\ii}{N}$,
\begin{equation}\label{e:Lbd1}
\Big|\Im L_N(E)-\Im L_N(z)\Big|\leq 10|s(z)|.
\end{equation}
As in the previous paragraph,   let $J=[A+\kappa,B-\kappa]\cap N^{-1-c}{\mathbb{Z}}$, where $c>0$ is an arbitrary small constant.  For any $E\in[A+\kappa,B-\kappa]$, let $E'$ be the closest point in $J$ and $z'=E'+\frac{\ii}{N}$.  Then the mean value theorem yields
\begin{equation}\label{e:mvt}
s(z)-s(z')=\OO\Big(N^{-1-c}\max_{x\in[A+\kappa,B-\kappa]}\big|s'(x+\ii N^{-1})\big|\Big),
\end{equation}
with an implicit constant uniform in the choice of $E$.

Note that  $|s'(z)|\leq  \eta^{-1}\Im s(z) \le \eta^{-2}$ (where the last inequality follows from \eqref{e:gtrivial}),  so $s(z)$ is $\eta^{-2}$-Lipschitz continuous. Taking a union bound over a mesh with spacing size $O(N^{-10})$,  with (\ref{eqn:tailnew}) we obtain that 
\[
\P \left(\max_{x\in[A+\kappa,B-\kappa]}   \Im s(x+\ii N^{-1}) > (\log N)^{7/10} \right) \le N^{-100}.
\]
Together with \eqref{e:Lbd1} and \eqref{e:mvt}, and again using $|s'(z)|\leq  \eta^{-1}\Im s(z)$, this gives 
\begin{equation}\label{e:int0}
\mathbb{P}\left(\forall E\in[A+\kappa,B-\kappa], \Big|\Im L_N(E)-\Im L_N(z)\Big|\leq 10(|s(z')|+N^{-c/2})\right)\geq 1-N^{-100}.
\end{equation}
Moreover, the same estimate as (\ref{eqn:Taylor}) holds for $\Im L_N$, so that
\begin{equation}\label{e:int0new}
\mathbb{P}\left(\forall E\in[A+\kappa,B-\kappa], \Big|\Im L_N(E)-\Im L_N(z')\Big|\leq 100(|s(z')|+N^{-c/2})\right)\geq 1-N^{-100}.
\end{equation}
Together with \eqref{eqn:microBound}, we conclude that
\begin{equation}\label{eqn:UBIm1bb}
\mathbb{P}\left(\sup_{E\in [A+\kappa,B-\kappa]}\Im L_N(E)\leq \sup_{E\in J}\Im L_N(E+\ii N^{-1})+(\log N)^{9/10}\right)\geq
1-\OO(N^{-100}). 
\end{equation}

Further, the analogues of \eqref{eqn:UBRe2} and \eqref{eqn:UBRe3} hold,  with the same proofs up to notational changes, and together with \eqref{eqn:UBIm1bb} they conclude the proof of the upper bound for the imaginary part in Theorem \ref{thm:Beta}.

\subsection{Lower bound for the real and imaginary parts.}\ \label{eqn:lowerB}
We start with the proof for the real part. The proof for the imaginary part is essentially the same, and is described at the end of this section.\\

\noindent{\it First step:  shift to the upper half plane.}\ For any $z= E + \iu \eta$ with $\eta>0$, 
by harmonicity of $z\in\mathbb{H}\mapsto\log|z-\lambda|$
we have
\[
\log|z-\lambda|=\int_{\mathbb{R}}\log|x-\lambda|\cdot \frac{\eta}{\eta^2+(E-x)^2}\frac{\rd x}{\pi}.
\]
This implies 
\begin{equation}\label{eqn:Harmo}
\Re L_N(z)=\int_{\mathbb{R}}\Re L_N(x)\cdot \frac{\eta_0}{\eta_0^2+(E-x)^2}\frac{\rd x}{\pi}
\end{equation}
for $z\in \R+\ii\eta_0$, with $\eta_0$ defined in (\ref{eqn:etao1}).
On the other hand, 
\begin{multline}
\int_{[A+\kappa,B-\kappa]^c}|\Re L_N(x)|\cdot \frac{\eta_0}{\eta_0^2+(E-x)^2}\frac{\rd x}{\pi}\leq \sum_i\int_{[A+\kappa,B-\kappa]^c}|\log|x-\lambda_i|-\log|x-\gamma_i||\cdot \frac{\eta_0}{\eta_0^2+(x-E)^2}\frac{\rd x}{\pi}\\
+
\int_{[A+\kappa,B-\kappa]^c}|N\int\log|x-\lambda|\rd\rho_V(\lambda)-\sum_i\log|x-\gamma_i||\cdot \frac{\eta_0}{\eta_0^2+(x-E)^2}\frac{\rd x}{\pi}.\label{eqn:cut}
\end{multline}

For $x\in[A+\kappa,B-\kappa]^{\rm c}$ and $E\in[A+2\kappa,B-2\kappa]$, we have $|x-E|>\kappa$,  so for such $x$ and $E$ we have $\frac{1}{\eta_0^2+(x-E)^2}\leq \frac{C}{1+x^2}$.  Therefore the first sum on the right-hand side of \eqref{eqn:cut} is smaller than
\begin{equation}\label{eqn:Harmo2}
C\eta_0 \sum_i\int_{\mathbb{R}}\frac{|\log|x-\lambda_i|-\log|x-\gamma_i||}{1+x^2}\rd x\leq C\eta_0\sum_i|\lambda_i-\gamma_i|\cdot|\log|\lambda_i-\gamma_i||.
\end{equation}
This implies that on the rigidity event from \Cref{lem:rig}, 
the first term on the right-hand side of (\ref{eqn:cut}) is $\OO( (\log N)^{50} \eta_0)= \oo(1)$; an error $\oo(\log N)$ would be enough for the proof of the leading order of the maximum, so there is substantial margin here. 
From rigidity of $\beta$-ensembles, 
\Cref{lem:rig}, we conclude that with probability $1-\OO(N^{-10})$,  for any 
$z\in [A+2\kappa,B-2\kappa]+\ii\eta_0$ we have
\begin{equation}\label{eqn:Harmo3}
\left|\Re L_N(z)-\int_{[A+\kappa,B-\kappa]}\Re L_N(x)\cdot \frac{\eta_0}{\eta_0^2+(E-x)^2}\frac{\rd x}{\pi}\right|\leq 1.
\end{equation}
Note that for $E\in[A+2\kappa,B-2\kappa]$ we have $\int_{[A+\kappa,B-\kappa]} \frac{\eta_0}{\eta_0^2+(E-x)^2}\frac{\rd x}{\pi}=1+\OO(\eta_0)$, so from the above equation we conclude that 
\begin{equation}\label{eqn:LBstep1}
\mathbb{P}\left(\sup_{z\in[A+2\kappa,B-2\kappa]+\ii\eta_0}\Re L_N(z)\leq \sup_{E\in[A+\kappa,B-\kappa]} \Re L_N(E)+1\right)=1-\oo(1).
\end{equation}

\noindent{\it Second step: lower bound for the smoothed field.} Similarly to \cite{claeys2021much} and \cite{lambert2023law}, the proof of the lower bound for $\max_{z\in[A+2\kappa,B-2\kappa]+\ii\eta_0}\Re L_N(z)$ will be an straightforward corollary of the convergence of the corresponding field 
to a Gaussian multiplicative chaos measure, in the full subcritical phase. 

More precisely,  consider a centered Gaussian field $G_\eta$ defined on  $[A+2\kappa,B-2\kappa]$ with covariance $\E[G_\eta(E_1)G_\eta(E_2)]=\sigma(1,0,(z_1,z_2))$, where we denote $z_i=E_i+\ii\eta$ and refer to  \eqref{eqn:sigma} for the definition of $\sigma$. The existence of this field for $\eta>\eta_0$ follows from positivity of the covariance, which is a byproduct of \Cref{thm:GaussFluct} (the limit of positive matrices is positive). 
From Lemma \ref{lem:varshift}, and noting that the covariance $\sigma$ is defined in terms of the kernel $c$ studied in that lemma, we have 
\be
\E[G_\eta(E_1)G_\eta(E_2)]
= -\frac{1}{\beta}\log|z_1-\bar z_2| + O_\kappa(1)
\ee
uniformly for $E_1, E_2 \in [A+2\kappa, B - 2\kappa]$. 

It is well known that for any $|\gamma|<\sqrt{2}$, there exists a random measure $\mu_{\gamma}$, called the Gaussian multiplicative chaos with parameter $\gamma$, such that the following holds. For any continuous $f\colon \R \rightarrow \R$ with compact support in $(A+2\kappa,B-2\kappa)$, we have the distributional convergence
\[
\lim_{\eta \rightarrow 0^+} \int f(E)\frac{e^{\sqrt{\beta}\gamma G_\eta(E)}}{\E[e^{\sqrt{\beta}\gamma G_\eta(E)}]}\rd E
= \int f \,\rd\mu_\gamma.
\]
The limiting  random variable can be written $\int f \rd\mu_\gamma$, for a certain  random measure $\mu_{\gamma}$, called the Gaussian multiplicative chaos with parameter $\gamma$.
We refer for example to \cite[Section 2.1]{claeys2021much}
for a modern treatment of the existence and non-triviality of this limit. 

In the following result,  we denote
$
\tilde L_N(z)=L_N(z)-\mu(1,0,z)
$,
with $\mu$ defined in (\ref{eqn:shift})

\begin{proposition}\label{prop:GMC}
For any $|\gamma|<\sqrt{2}$ and any continuous $f$ with compact support in $(A+2\kappa,B-2\kappa)$, the following convergence in distribution holds: 
\[
\lim_{N\rightarrow \infty} \int_{A+2\kappa}^{B-2\kappa} f(E)\frac{e^{\sqrt{\beta}\gamma \Re \tilde L_N(E+\ii\eta_0)}}{\E[e^{\sqrt{\beta}\gamma\Re \tilde L_N(E+\ii\eta_0)}]}\rd E
=
\int f \rd\mu_\gamma.\]
\end{proposition}

\begin{proof}
The proof is a direct application of the general criterion for convergence of a non-Gaussian random field to a Gaussian multiplicative chaos;  see  \cite[Theorem 2.4]{claeys2021much} (which is a restatement from
\cite[Theorem 1.7]{LamOstSim2018}),  and the key technical input,  the Laplace transform of the log-characteristic polynomial from Theorem \ref{thm:GaussFluct}. 
More precisely, \cite[Theorem 2.4]{claeys2021much}  states that a sufficient condition for the conclusion of Proposition \ref{prop:GMC} is that there is a constant $c>0$ such that for any fixed $p$ and $\zeta_1,\dots,\zeta_p\in\mathbb{R}$,  uniformly in $\bz\in([A+2\kappa,B-2\kappa]\times [\eta_0,c])^p$,  we have
\begin{equation}\label{eqn:enough3}
\E_{\mu}\left[e^{\sum_{k=1}^p\zeta_k{\rm Re}\tilde L_N(z_k)}\right]=
\E\left[e^{\sum_{k=1}^p\zeta_kL(z_k)}\right]
(1+\oo(1))
\end{equation}
as $N\to\infty$, where $L$ is a centered Gaussian field defined on $[A+2\kappa,B-2\kappa]\times [\eta_0,c]$ characterized by ${\rm Var}[\sum_{k=1}^p\zeta_kL(z_k)]=\sigma(\bzeta,0,\bz)$.
Equation (\ref{eqn:enough3}) is a direct consequence of Theorem \ref{thm:GaussFluct}.
Note that we consider $\tilde L_N$ here instead of $L_N$ because \cite[Theorem 2.4]{claeys2021much} does not explicitly cover the possibility of a limiting shift.
\end{proof}
We now apply \cite[Theorem 3.4]{claeys2021much}.  The assumptions of this theorem are easily verified: 
 \cite[Assumption 3.1]{claeys2021much} follows from \ref{eqn:enough3} and 
 \cite[Assumption 3.3]{claeys2021much} follows from Proposition \ref{prop:GMC} and a standard approximation argument to allow indicators for $f$.  Then \cite[Theorem 3.4]{claeys2021much} gives,
for any fixed $\e>0$,
\[
\mathbb{P}\left(\max_{z\in[A+2\kappa,B-2\kappa]+\ii\eta_0}\Re \tilde L_N(z)\geq \left(\sqrt{\frac{2}{\beta}}-\e\right)\log N\right)=1-\oo(1).
\]
From Lemma \ref{lem:varshift},  $\mu(1,0,z)$ is uniformly bounded,  so from the previous equation there exists $C>0$ such that
\begin{equation}\label{eqn:LBstep2}
\mathbb{P}\left(\max_{z\in[A+2\kappa,B-2\kappa]+\ii\eta_0}\Re  L_N(z)\geq \left(\sqrt{\frac{2}{\beta}}-\e\right)\log N-C\right)=1-\oo(1).
\end{equation}
Equations (\ref{eqn:LBstep1}) and (\ref{eqn:LBstep2}) conclude the proof of the lower bound in Theorem \ref{thm:Beta}, for $\Re L_N$.\\

\noindent{\it Lower bound for the imaginary part.}\ 
The proof for $\Im L_N$ is identical because Theorem \ref{thm:GaussFluct}  also gives its joint Laplace transform, with limiting covariance $\sigma(0,\bzeta,\bz)=\sigma(\bzeta,0,\bz)+\OO(1)$,  a fact easily checked with   (\ref{eqn:sigma}) and Lemma \ref{lem:varshift}.  The only small difference is about the analogue of (\ref{eqn:cut}), i.e. bounding $\int_{[A+\kappa,B-\kappa]^c}|\Im L_N(x)|\cdot \frac{\eta}{\eta^2+(E-x)^2}\frac{\rd x}{\pi}$. By rigidity of the eigenvalues we have $|\Im L_N(E)|\leq N^\e$ for all $E$ (this is a consequence of the implications (\ref{eqn:equivRig2}) below), with overwhelming probability, and this is enough to conclude.

\subsection{Proof of \Cref{cor:Beta}}. \label{sec:elementary}
From the definition of $\Im \log$ before \eqref{def:L_N}, we have \[
\Im L_N(E)=\pi \left( \sum_{i=1}^N \mathds{1}_{\{\lambda_i>E\}}- N\int_E^{\infty}\rho_V(s)\, \rd s \right).\]  
Let $k\in\llbracket 0,N-1\rrbracket$, 
$n\in\mathbb{N}$, 
and
$
M\in[n,n+1]
$ be parameters such that $n+2<k<N-(n+2)$. 
For $E\in[\gamma_k,\gamma_{k+1})$, 
the previous display, together with the definition \eqref{quantiles} of $\gamma_k$, yields the following implications:
\begin{align}
\Im L_N(E)>\pi M\ \Rightarrow
\   \, \;
 \lambda_{k-n+2}
> \gamma_{k}\; \; \ \ &\ \Rightarrow\ \ \lambda_{k-n+2}-\gamma_{k-n+2}
> \gamma_{k}-\gamma_{k-n+2}
,\notag\\
\Im L_N(E)<\pi M\ \Rightarrow
\ \ \lambda_{k-n-2}
< \gamma_{k+1} \ &\ \Rightarrow\ \ \lambda_{k-n-2}-\gamma_{k-n-2}
< \gamma_{k+1}-\gamma_{k-n-2}\label{eqn:equivRig}
.
\end{align}
For the upper bound on the eigenvalue deviations, we take
$M = \pi^{-1} \sqrt{\frac{2}{\beta} }(1 + \epsilon) \log N$. By 
Theorem \ref{thm:Beta}, we have for any $\delta>0$ that $\sup_{A + \delta < E < B - \delta} \Im L_N(E) \le \pi M$ with high probability. Then taking $n = \lfloor M \rfloor$ and $\delta$ sufficiently small (in a way that depends only on $\kappa$), \eqref{eqn:equivRig} implies that for every $j \in \llbracket \kappa N, (1-\kappa) N \rrbracket$, we have 
\be\label{intermediategap}
\lambda_j - \gamma_j < \gamma_{j + n +2  } - \gamma_j. 
\ee
Further, if we define the quantity $\eps_{j,m}$ implicitly by 
\[
\gamma_{j+m}-\gamma_j=\frac{m}{N\rho_{V}(\gamma_k)}(1+\eps_{j,m}),
\]
then
\eqref{quantiles} implies that
\be\label{intermediategapbound}
\sup_{j\in\llbracket \kappa N,(1-\kappa)N\rrbracket} \sup_{m\le (\kappa/2) \sqrt{N} } 
\eps_{j,m}  = \oo(1).
\ee
Combining \eqref{intermediategap} with \eqref{intermediategapbound} completes the proof of the upper bound. The proof for the lower bound is similar and hence omitted. 
\begin{remark}
The reasoning in this proof of \Cref{cor:Beta} could be reversed to show that \Cref{cor:Beta} implies the bound on the imaginary part in \Cref{thm:Beta}, which demonstrates that these statements are logically equivalent. The relevant implications are now (using the same notation)
\begin{align}
\lambda_{k-n-2}
-\gamma_{k-n-2}
> \gamma_{k+1}-\gamma_{k-n-2}
&\  \Rightarrow\ \ \Im L_N(E)>\pi M,\notag\\
\lambda_{k+n+2}-\gamma_{k+n+2}
< \gamma_{k}-\gamma_{k+n+2}
 & \ \Rightarrow\ \ \Im L_N(E)< \pi M.\label{eqn:equivRig2}
\end{align}
\end{remark}

\section{Relaxation}\label{sec:Relax}

This section proves convergence to equilibrium of $\sup \Im L_N$ (subsection \ref{s:relaxprelim}) and $\sup\Re L_N$ (subsection \ref{subsec:Re}) for
 the matrix Ornstein--Uhlenbeck dynamics (\ref{OU}) defined below.
Indeed we will prove that theorems \ref{thm:maxWigner}, \ref{thm:maxWignerDivisible}  and  \ref{thm:Wigner} (i) hold for matrices of type $H_t$,  for large enough $t$.  
We will finally prove relaxation of large moments of $ \Im L_N$,  which corresponds to a proof of Theorem \ref{thm:Wigner} (i) for such weakly Gaussian-divisible random matrices $H_t$ (subsection \ref{s:relaxLarge}).

As explained in the introduction,  this is an essential step in the proof for Wigner matrices,  which then proceeds by density of the weakly Gaussian-divisible ensemble in Wigner matrices (the moment matching from Section \ref{sec:MomMat}).

For local statistics in the bulk of the spectrum, relaxation was first proved in \cite{ErdSchYau2011} by a method based on entropy dissipation,  up to an averaging on the energy level which prevents from considering observables such as $\sup\Re L_N,\sup\Im L_N$.
Another method for relaxation was introduced in \cite{BouErdYauYin2016}, through a coupling of the spectrum of $H_t$ with a GOE. In this approach  relaxation follows from  homogenization of the Dyson Brownian motion: The difference between both spectra 
satisfies a deterministic, non-local parabolic equation at leading order,  locally and with probability $1-\oo(1)$.

While ergodicity of $\Im L_N$ is closely related to relaxation of local spectral statistics,  ergodicity of $\Re L_N$ requires convergence to equilibrium along the full spectrum.  Moreover part (ii) of Theorem \ref{thm:Wigner}  requires probability bounds stronger than $1-\oo(1)$.  
Fortunately,  the homogenization theory from \cite{BouErdYauYin2016} was greatly strengthened in \cite{LanSosYau2016} and in \cite{Bou2018},  as it holds with the probability bound $1-N^{-D}$ for arbitrary $D$.
For our paper,  the homogenization from \cite{Bou2018} is most pertinent as it holds for very large times,  a key fact for our observable $L_N$,  which reaches equilibrium only for $t= N^{-\oo(1)}$.  The methods from \cite{Bou2018} also directly cover relaxation of $\Re L_N$,  another decisive fact as the sum of the errors from the local homogenization in Proposition \ref{prop:homog} below exceeds the required $\oo(\log N)$ accuracy to catch the maximum of $\Re L_N$.

\subsection{The eigenvalues.}\  \label{s:relaxprelim} 
We first provide a quantitative relaxation of the eigenvalues (Proposition \ref{prop:homog}),  which is a variant of \cite[Theorem 3.1]{Bou2018} and relies on this work.

Let $H$ be a Wigner matrix. We first recall the definition of Dyson Brownian motion with initial data $H_0 = H$. As noted above in \Cref{s:ideas}, for concreteness we consider just the real symmetric case, as the complex Hermitian case is analogous.

Let $B \in \matn$ be such that the entries $\{ B_{ij} \}_{i < j }$ and $B_{ii}/\sqrt{2}$ are independent standard Brownian motions, and $B_{ij}=B_{ji}$. 
Consider the matrix Ornstein--Uhlenbeck process
\begin{equation}\label{OU}
\rd H_{t}=\frac{1}{\sqrt{N}} \rd B_t-\frac{1}{2}H_t\,\rd t.
\end{equation}
If the eigenvalues of $H_0$ are distinct, it is well known that the eigenvalues $\bla(t) = (\lambda_1(t), \lambda_2(t), \dots , \lambda_N(t))$ of $H_t$ are given by the strong solution
of the system of stochastic differential equations 
\begin{align}
\rd\la_k=\frac{\rd \beta_{k}}{\sqrt{N}}+\left(\frac{1}{N}\sum_{\ell\neq k}\frac{1}{\la_k-\la_\ell}-\frac{1}{2}\lambda_k\right)\rd t\label{eqn:eigenvaluesSymmetric},
\end{align}
where the  $\{ \beta_k \}_{k=1}^N$ are independent, standard Brownian motions. (See, for example, \cite[Lemma 4.3.3]{anderson2010introduction}.)

We now let $\bmu(t)=(\mu_1(t), \mu_2(t), \dots , \mu_N(t))$ be a strong solution of the same SDE \eqref{eqn:eigenvaluesSymmetric}
with initial condition $\bmu(0)=(\mu_1, \mu_2, \dots, \mu_N)$, where $\{\mu_k\}_{k=1}^N$ are the eigenvalues of a $\GOE_N$, denoted ${\rm GOE}$:
\[
\rd\mu_k=\frac{\rd \beta_{k}}{\sqrt{N}}+\left(\frac{1}{N}\sum_{\ell\neq k}\frac{1}{\mu_k-\mu_\ell}-\frac{1}{2}\mu_k\right)\rd t.
\]
For any $z\in \mathbb{H}$, we define
\begin{equation}\label{eqn:zt}
z_t=\frac{e^{t/2}(z+\sqrt{z^2-4})+e^{-t/2}(z-\sqrt{z^2-4})}{2},
\end{equation}
where $\sqrt{z^2-4}$ is defined using a branch cut in the segment $[-2,2]$, as with $\msc$ in \eqref{e:msc}. For $z\in\mathbb{R}$, we define $z_t=\lim_{\eta\to 0^+} (z+\ii\eta)_t$.
The following key estimate on the difference between $\bla(t)$ and $\bmu(t)$ follows from the main result in \cite{Bou2018}.  We recall the notations $\varphi$ from (\ref{eqn:phi2}) and $\gamma_k$ from \eqref{quantiles}, and let $L^H$ and $L^{\rm GOE}$ denote the observable \eqref{def:L_N} defined using the eigenvalues of $H$ and ${\rm GOE}$, respectively.

\begin{proposition}\label{prop:homog} Fix $\kappa,\e>0$. Then for any $D>0$ there exist $C(\eps, \kappa, D)>0$ such that for all 
$t \in (  \varphi^C/N, 1)$, $E \in [-2+\kappa,2-\kappa]$, and $k \in \unn{1}{N}$ such that $ \gamma_k \in[-2+\kappa,2-\kappa]$, we have
\begin{equation}\label{eqn:homogenization}
\mathbb{P}\Big(\Big|\la_k(t)-\mu_k(t)-\frac{\Im L^H_N(E_t)-\Im L^{\rm GOE}_N(E_t)}{N\im \msc(E_t)}\Big|>\frac{N^{1+\e}\max(|E-\gamma_k|,N^{-1})}{N^2 t}\Big)\leq C N^{-D}.
\end{equation}
\end{proposition}

\begin{proof}
The key to the proof is \cite[Theorem 3.1]{Bou2018}, which states that there exists $ C(D) >0$ such that
\begin{equation}\label{eqn:thm41}
\mathbb{P}\left( \Big|\big(\la_{k}(t)-\mu_k(t)\big)-\bar {\frak u}_k(t)\Big|>\frac{N^{\varepsilon}}{N^2 t}\right)\leq C N^{-D}
\end{equation}
for $t \in (  \varphi^C/N, 1)$, where we define
\begin{equation}
\bar {\frak u}_k(t)=\frac{1}{N\im \msc(\gamma_k^t)}\sum_{j=1}^N\im\left(\frac{1}{\gamma_j-\gamma_k^t}\right)\big(\la_j(0)-\mu_j(0)\big), \qquad \gamma_k^t = (\gamma_k)_t. \label{eqn:aver}
\end{equation}
Moreover, from \cite[Lemma 3.4]{Bou2018}, for all $\gamma_k,\gamma_\ell\in[-2+\kappa,2-\kappa]$ we have
\begin{equation}
\P \left( |\bar{\frak u}_k(t)-\bar{\frak u}_\ell(t)|\ge C\varphi \frac{|k-\ell|}{N^2t} \right) \le C N^{-D} \label{elem1}.
\end{equation}
Let $E \in [-2 +\kappa, 2-\kappa]$ be given, and fix some $\ell = \ell(E,N)$ such that  
$
|E-\gamma_\ell | =  \min_{j \in \unn{1}{N}} |E-\gamma_j|.
$
The definition of $\gamma_k$ in \eqref{quantiles} (with $\nu=\rho_{\rm sc}$) gives
\begin{equation}
| k - \ell| <   C N | \gamma _k - \gamma_\ell | \le 
CN \left( | \gamma _k - E | + | E - \gamma_\ell | \right)
\le 2 CN  | \gamma _k - E | ,
\end{equation}
for some constant $C>0$.
Then equations (\ref{eqn:thm41}) and (\ref{elem1}) together with the previous line imply that
\begin{equation}\label{eqn:intermold}
\mathbb{P}\left(\ \Big|\big(\la_{k}(t)-\mu_k(t)\big)-\bar {\frak u}_\ell(t)\Big|>  \frac{C N^{1+\varepsilon}\max(|E-\gamma_k|,N^{-1})}{N^2 t}\right)\leq CN^{-D},
\end{equation}
where we increased $C$ if necessary and used $N^\eps \ge \phi$ for sufficiently large $N$ (depending on $\eps$).
We  therefore just need to bound
$
\left|\frac{\Im L^H_N(E_t)-\Im L^{\rm GOE}_N(E_t)}{N\im \msc(E_t)}-\bar {\frak u}_\ell(t)\right|.
$
We write
\begin{equation}\label{320}
\frac{\Im L^H_N(E_t)-\Im L^{\rm GOE}_N(E_t)}{N\im \msc(E_t)}
= 
\frac{1}{{N\im \msc(E_t)}}
\Im \sum_{j=1}^N \log\left(1 + \frac{\lambda_j(0) -\mu_j(0)}{\mu_j(0) - E_t}  \right).
\end{equation}
On the rigidity event from (\ref{e:rigidity}),  a Taylor expansion of the logarithm gives, with overwhelming probability,
\begin{equation}\label{412}
\Im \sum_{j=1}^N \log\left(1 + \frac{\lambda_j(0) -\mu_j(0)}{\mu_j(0) - E_t}  \right) = \Im \sum_{j=1}^N  \frac{\lambda_j(0) -\mu_j(0)}{\mu_j(0) - E_t}   + \OO  \left( \sum_{j=1}^N \left| \frac{ \lambda_j(0) -\mu_j(0)}{\mu_j(0) - E_t} \right|^2 \right).
\end{equation}
For the error term,  on the rigidity event from (\ref{e:rigidity}) we can write
\begin{equation}
 \sum_{j=1}^N \left| \frac{ \lambda_j(0)-\mu_j(0) }{\mu_j(0) - E_t} \right|^2 \le \frac{C\phi^2}{N t}\cdot \frac{1}{N} 
\sum_{j=1}^N   \frac{\Im E_t}{|\mu_j(0) - E_t|^2}=
 \frac{C\phi^2}{N t}\ \im m_N(E_t)\leq  \frac{C\phi^2}{N t},
\end{equation}
where we used $c t \le \Im E_t\le Ct$ to bound $\im m_N(E_t)\le C$ using \eqref{e:sclaw}. This estimate on $\im m_N(E_t)$ also shows that the second term in \eqref{412} is negligible when inserted in \eqref{320}.

Finally, we need to bound 
\begin{align*}
\frac{1}{{N \Im \msc (E_t)}} \Im \sum_{j=1}^N  \frac{\lambda_j(0) -\mu_j(0)}{\mu_j(0) - E_t} -\bar {\frak u}_\ell(t)
&=\frac{1}{N}\sum_{j=1}^N (\lambda_j(0)-\mu_j(0))\left(
\frac{1}{{ \Im \msc (E_t)}}-\frac{1}{{\Im \msc (\gamma_\ell^t)}}\right)\im\frac{1}{\mu_j(0)-E_t} \\
+&
\frac{1}{N}\sum_{j=1}^N \frac{\lambda_j(0)-\mu_j(0)}{\Im \msc (\gamma_\ell^t)}\, \im\left(
  \frac{1}{\mu_j(0) -E_t}-\frac{1}{\gamma_j -\gamma_\ell^t}\right).
\end{align*}
For the first sum, from $|\im \msc(E_t)-\im \msc(\gamma_\ell^t)| \le C |E_t-\gamma_\ell^t|\le C N^{-1}$, 
$\im \msc(E_t) \ge c$, and $\im \msc(\gamma_\ell^t)\ge c$,   on the rigidity event from \eqref{e:rigidity} we obtain
\[
\frac{1}{N}\sum_{j=1}^N |\lambda_j(0)-\mu_j(0)|\left(
\frac{1}{{ \Im \msc (E_t)}}-\frac{1}{{\Im \msc (\gamma_\ell^t)}}\right)\im\frac{1}{\mu_j(0)-E_t}\leq \frac{C\varphi}{N^3}\cdot \sum\frac{t}{|\gamma_j-E_t|^2}\leq \frac{C\varphi}{N^2}.
\]
On the same rigidity event, the second sum is bounded by
\[
\frac{1}{N}\sum_{j=1}^N |\mu_j(0) - \lambda_j(0)| \left|\im\frac{E_t-\gamma_\ell^t}{(\mu_j(0) -E_t)(\gamma_j -\gamma_\ell^t)}\right|\leq \frac{C\varphi}{N^3}
\sum_j
\left(\frac{1}{|\mu_j(0) -E_t|^2} + \frac{1}{|\gamma_j -\gamma_\ell^t|^2}
\right)\leq \frac{C\varphi}{N^2 t}.
\]
We have thus obtained
\begin{equation}\label{eqn:MesoApprox}
\left|\frac{\Im L^H_N(E_t)-\Im L^{\rm GOE}_N(E_t)}{N\im \msc(E_t)}-\bar {\frak u}_\ell(t)\right|\leq C\frac{\varphi^2}{N^2t},
\end{equation}
which concludes the proof.
\end{proof}

\begin{remark}\label{rem:betterprobabilitybound}
A stronger result than (\ref{eqn:homogenization}) actually holds, in terms of the probability bound:
for any $\kappa,\e>0$ there exists $C,\delta,N_0>0$ (depending on $\kappa$ and $\e$) such that for any 
$t \in (  \varphi^C/N, 1)$, $E \in [-2+\kappa,2-\kappa]$,  $k \in \unn{1}{N}$ satisfying $ \gamma_k \in[-2+\kappa,2-\kappa]$  and $N\geq N_0$, we have
\[
\mathbb{P}\Big(\Big|\la_k(t)-\mu_k(t)-\frac{\Im L^H_N(E_t)-\Im L^{\rm GOE}_N(E_t)}{N\im \msc(E_t)}\Big|>\frac{N^{1+\e}\max(|E-\gamma_k|,N^{-1})}{N^2 t}\Big)\leq e^{-\delta\varphi^\delta}.
\]
Indeed the proof of (\ref{eqn:thm41}) from \cite{Bou2018} relies on the rigidity estimate (\ref{e:rigidity}), which holds with probability $e^{-c\varphi^c}$, so that the probability bound $N^{-D}$ in  \cite[Theorem 3.1]{Bou2018} can be strengthened to $e^{-\delta\varphi^\delta}$ for a fixed, small enough $\delta$, by elementary changes in the proof. 

This improved probability bound is not necessary for the proofs of Theorem \ref{thm:maxWigner} and Theorem \ref{thm:Wigner} (i). It will be used in the proof of Theorem \ref{thm:Wigner} (ii).
\end{remark}

\subsection{The characteristic polynomial.}\  \label{subsec:Re}
After the relaxation of individual eigenvalues in the previous subsection, we study the relaxation of $L_N$, an a priori more intricate problem as $\Re L_N$ depends on the full spectrum.
Our results are of two types: relaxation of the full characteristic polynomial in the rectangle $[-2+\kappa,2-\kappa]\times [N^{-1},1]$ and $t\in[\varphi^{-K},1]$,  up to an error $(\log N)^{1/2}$ (Proposition \ref{l:relax}), 
and relaxation of its maximum on $[-2+\kappa,2-\kappa]$ for $t=\Omega(1)$,  up to tightness (Proposition \ref{l:relax2})  .

\begin{proposition}\label{l:relax}
For all $K>10$, $\e,\kappa>0$, the following holds.   For $z=E+\ii\eta$,  uniformly in  $\eta \in (N^{-1}, 1)$ and  $t\in[\varphi^{-K},1]$, we have
\begin{equation}\label{eqn:relax}
\mathbb{P}\left(\max_{-2+\kappa<E<2+\kappa}
\left|\sum\log\big(z-\mu_k(t)\big)-\sum\log\big(z-\la_k(t)\big)\right| > (\log N)^{1/2+\e}\right)=\oo(1).
\end{equation}
\end{proposition}

\begin{proof}
If  $\varphi^{100}/N<\eta<1$, this follows from \cite{Bou2018} and \Cref{lem:maxSmoothField}. Indeed, by integrating over the parameter $\nu\in[0,1]$ in \cite[Proposition 2.11]{Bou2018}, we have
\begin{equation}\label{eqn:StochAdv}
\max_{-2+\kappa<E<2+\kappa}
\left|\sum_{k=1}^N \log \frac{ z-\mu_k(t)}{ z-\la_k(t)}
- \sum_{k=1}^N \log \frac{ z_t-\mu_k(0)}{z_t-\la_k(0)}\right|<1
\end{equation}
 with overwhelming probability. 
Together with \Cref{lem:maxSmoothField} (noting $ z_t \ge C^{-1} t$
for some $C(\kappa) >0$ by an explicit computation using \eqref{eqn:zt}), it implies that
\begin{equation}\label{eqn:abithigher}
\max_{-2+\kappa<E<2+\kappa}
\left|\sum\log\big( z-\mu_k(t)\big)-\sum\log\big( z-\la_k(t)\big)\right|<(\log N)^\e
\end{equation}
with probability $1-\oo(1)$.

We now consider the case where $N^{-1} <\im \eta <\varphi^{100}/N$, and denote $\tilde z= E +\ii\varphi^{100}/N$. Given that \eqref{eqn:abithigher} holds for $\tilde z$, we only need to prove that
\begin{equation}\label{eqn:enough}
\max_{-2+\kappa<E<2+\kappa}
\left| \sum_{k=1}^N  \log \frac{z-\mu_k(t)}{z-\la_k(t)} 
- \sum_{k=1}^N \log \frac{ \tilde z-\mu_k(t)}{\tilde z-\la_k(t)}
\right|<(\log N)^{1/2+\e}
\end{equation}
with probability $1-\oo(1)$ to complete the proof.

We divide the sum in \eqref{eqn:enough} into two parts, depend on whether $| \gamma_k  - E | > N^{-1/2}$ or not. For the terms such that $| \gamma_k  - E | > N^{-1/2}$,  by the rigidity bound \eqref{e:rigidity} and the Taylor expansion for $x \mapsto \log (1 +x)$ we have
\begin{multline*}
\sum_{|\gamma_k-E|>N^{-1/2}} \left|\sum\log(\tilde z-\mu_k(t))-\sum\log(z-\mu_k(t))\right|  
 =  \sum_{|\gamma_k-E|>N^{-1/2}} \left| \log \left( 1 + \frac{\widetilde z  - z + \mu_k(t) - \lambda_k(t)}{z - \lambda_k(t)}  \right) \right|
\\  \le C \sum_{|\gamma_k - E | > N^{-1/2} } \left| \frac{\tilde z - z }{z - \gamma_k} \right| + \left|  \frac{\mu_k(t) - \lambda_k(t)  }{z - \gamma_k}  \right|
\le \frac{C\varphi^{200}}{N} \frac{1}{|z - \gamma_k|} \le \frac{\varphi^{300}}{N^{1/2}} = \oo(1).
\end{multline*}
The same holds when replacing $\bmu$ with $\bla$. Hence  to prove (\ref{eqn:enough}) we only need to obtain the following bounds for the terms such that $| \gamma_k  - E | < N^{-1/2}$:
\begin{equation}\label{eqn:enough2}
\max_{-2+\kappa<E<2+\kappa}
\left|\sum_{|\gamma_k-E|<N^{-1/2}}\log(z-\mu_k(t))-\sum_{|\gamma_k-E|<N^{-1/2}}\log(z-\la_k(t)) \right|<(\log N)^{1/2+\e},
\end{equation}
and the same bound with $z$ replaced by $\widetilde z$. In the following we prove only \eqref{eqn:enough2}, as the proof for $\widetilde z$ is the same.
For this,  we now bound 
\begin{align}
\bigg|&\sum_{|\gamma_k-E|<N^{-1/2}}\log(z-\mu_k(t))-\sum_{|\gamma_k-E|<N^{-1/2}}\log(z-\la_k(t) ) \bigg|\\ 
&\leq 
\sum_{\substack{{|\gamma_k-E|<N^{-1/2}}\\|\mu_k(t)-E|\leq(\log N)^{\eps} N^{-1}}}\left|\log\frac{z-\mu_k(t)}{z-\lambda_k(t)}\right|
+ \sum_{\substack{{|\gamma_k-E|<N^{-1/2}}\\(\log N)^{\eps} N^{-1} \le |\mu_k(t)-E| \le \phi^3 N^{-1} }}\left|\log\frac{z-\mu_k(t)}{z-\lambda_k(t)}\right|
\label{eqn:dec}
\\
&+\label{lasttt}
\Bigg|\sum_{\substack{{|\gamma_k-E|<N^{-1/2}}\\ \phi^{100} N^{-1} \le |\mu_k(t)-E|}}\log\frac{z-\mu_k(t)}{z-\lambda_k(t)}\Bigg|.
\end{align}
We begin by considering the contribution from the first sum in \eqref{eqn:dec}. 
 First, suppose that $|z-\mu_k|\geq |z-\lambda_k|$. Using $|\log |1+z||<\log(1+|z|)$ for $|1+z|>1$ and Proposition \ref{prop:homog},
we obtain that with high probability
\begin{multline}\label{335}
\left|\log\frac{|z-\mu_k(t)|}{|z-\lambda_k(t)|}\right|\leq \log\left(1+\left|
\frac{\mu_k(t)-\lambda_k(t)}{z-\mu_k(t)}
\right|\right)\\
\leq 
\log\left(1+\left|
\frac{\Im L^H_N(E_t)-\Im L^{\rm GOE}_N(E_t)}{N\im \msc(E_t)\, (z-\mu_k(t))}+\OO\left(\frac{N^{1+\e}\max(|E-\gamma_k|,N^{-1})}{N^2 t|z-\mu_k(t)|}\right)\right|
\right).
\end{multline}
By  \Cref{lem:maxSmoothField}, we have for all $\eta' \in (0,1)$ that $\P( \mathcal G_{\eta'}) = 1 - \oo(1)$, where 
\begin{equation}
\mathcal G_{\eta'} = \left\{ \max_{-2+\kappa<E<2+\kappa}
\left|\sum\log\big( z'_t-\mu_k(0)\big)-\sum\log\big( z'_t-\la_k(0)\big)\right| \le (\log N)^\e \right\}, \quad z'= E + \ii \eta'
\end{equation}
and the implicit constant in the $\oo(1)$ term depends only on the choices of $K$, $\epsilon$, and $\kappa$. 
On $\mathcal G_\eta$, using $\im z\ge N^{-1}$ and $|N(z-\mu_k(t))|\ge1$ in \eqref{335}, we obtain 
\begin{equation}\label{e:neark}
\left|\log\frac{|z-\mu_k(t)|}{|z-\lambda_k(t)|}\right|\leq
\log\left(1+2 (\log N)^{\e}+\frac{N^{\e}\max(|E-\gamma_k|, N^{-1})}{ t}\right)
<C\e\log\log N.
\end{equation}
In \eqref{e:neark}, we used that $|\mu_k(t)-E|\le (\log N)^\e/N$ implies $|E-\gamma_k|<C \varphi/N$ when rigidity \eqref{e:rigidity} holds. 
The same bound as \eqref{e:neark} naturally holds for $|z-\mu_k|> |z-\lambda_k|$. Thus for the first sum in \eqref{eqn:dec} we conclude that
\begin{equation}\label{lastfactor}
\sum_{\substack{{|\gamma_k-E|<N^{-1/2}}\\|\mu_k(t)-E|\leq(\log N)^{\eps} N^{-1}}}
\left|\log\frac{|z-\mu_k(t)|}{|z-\lambda_k(t)|}\right|
\leq
C\e\log\log N\,
\left|\left\{|\mu_k-E|\leq \frac{(\log N)^\e}{N}\right\}\right|.
\end{equation}
Theorem \ref{thm:Hoi} applied with $k = ( \log N)^{1/2}$ gives
\begin{equation}\label{smallsum}
\P \left( \left|\left\{|\mu_k-E|\leq \frac{(\log N)^\e}{N}\right\}\right|  > (\log N)^{1/2} \right)  =\oo(1).
\end{equation}
Then \eqref{lastfactor} and \eqref{smallsum} together imply that, 
with probability $ 1 - \oo(1)$,
\begin{equation}\label{e:smallk}
\sum_{\substack{{|\gamma_k-E|<N^{-1/2}}\\|\mu_k(t)-E|\leq(\log N)^{\eps} N^{-1}}}
\left|\log\frac{z-\mu_k(t)}{z-\lambda_k(t)}\right| \le C ( \log N)^{1/2+\e}.
\end{equation}

We next consider 
the second sum on
the left side of \eqref{eqn:dec}, and work on the event $\mathcal G_\eta$. 
By Taylor expansion and \Cref{prop:homog},
\begin{align}\label{theexpansion}
\left|\log\frac{z-\mu_k(t)}{z-\lambda_k(t)}\right| &\le \left| \log \left(1+\frac{L^H(E,t)-L^{\rm GOE}(E,t)}{N(z-\mu_k(t))}+\OO\left(\frac{N^{1+\e}\max(|E-\gamma_k|,N^{-1})}{N^2 t|z-\mu_k(t)|}\right) \right)\right|
\\ &=
\label{343}
\OO\left( \frac{L^H(E,t)-L^{\rm GOE}(E,t)}{N(z-\mu_k(t))}\right)+\OO\left(\frac{N^{1+\e}\max(|E-\gamma_k|,N^{-1})}{N^2 t|z-\mu_k(t)|}\right)
\\
 &= \frac{C( \log N)^{\eps/2}}{N |z-\mu_k(t)|}
+  \frac{ CN^{-1 + \e} t^{-1}\varphi }{N |z-\mu_k(t)|}
\le   \frac{ C ( \log N)^{\eps/2}}{N |z-\mu_k(t)|}.
\label{334}
\end{align}
In the second term in \eqref{343}, 
we used rigidity and $|\mu_k(t)-E| < \varphi N^{-1}$ to show that $|E - \gamma_k | \le 2 \varphi N^{-1}$.

We now bound
\begin{equation}\label{e:mediumk}
\sum_{(\log N)^\e/N<|\mu_k-E|<\varphi^{100}/N} \frac{1}{\left|z-\mu_k(t)\right|}\leq
\sum_{1\le j\leq A(\log\log N)^2}\frac{1}{(2^j/N)}\left|\big\{| \mu_k - E|\in [2^j/N,2^{j+1}/N]\big\}\right|,
\end{equation}
where $A>0$ is a constant depending only on the constant $C_0$ used to define $\varphi$ in \eqref{eqn:phi2}.
Set $\tau = (\log N)^{-1/2}$ and define the event
\begin{equation}
\mathcal A_j^E=\left\{\left|\{|\mu_k-E|\in [2^j/N,2^{j+1}/N]\}\right|\leq \frac{2^j}{\tau}\right\}.
\end{equation}
From Theorem \ref{thm:Hoi} applied with $\gamma = 1/2$ and $k = \tau^{-1} 2^j$, we have
\begin{equation}\label{e:above1}
\mathbb{P}\left(\left(\mathcal A_j^E\right)^{\rm c}\right)\leq \exp\left(\frac{1}{4}\log(C\tau)(2^j/\tau)^2\right) +\exp\left(-N^c\right)
\end{equation}
for all $j \le A (\log \log N)^2$. Define
\begin{equation}
\mathcal A = \bigcap_{\substack{1\leq k\leq N\\0\leq j\leq A (\log\log N)^2}} \mathcal A_j^{\gamma_k}.
\end{equation}
From \eqref{e:above1}, we obtain using a union bound that
\begin{equation}\label{eqn:overcrowding}
\mathbb{P}(\mathcal A )=1-\oo(1).
\end{equation}
On the event where both $\mathcal A$ and the rigidity estimate \eqref{e:rigidity} hold, by \eqref{e:mediumk} we have
\begin{equation}\label{e:mediumsum}
\frac{1}{N}\sum_{(\log N)^\e/N<|\mu_k-E|<\varphi^{100}/N} \frac{1}{\left|z-\mu_k(t)\right|}
\le A (\log N)^{1/2} (\log \log N)^2 \le (\log N)^{1/2 + \eps/4}.
\end{equation}
Combining  \eqref{334}, and \eqref{e:mediumsum}, we find 
\begin{equation}\label{e:concludemediumk}
\sum_{\substack{{|\gamma_k-E|<N^{-1/2}}\\(\log N)^{\eps} N^{-1} \le |\mu_k(t)-E| \le \phi^{100} N^{-1} }}\left|\log\frac{z-\mu_k(t)}{z-\lambda_k(t)}\right|
 \le ( \log N)^{1/2 + \eps}.
\end{equation}

Finally, we consider the sum in \eqref{lasttt}. By Taylor expansion and rigidity \eqref{e:rigidity}, we have
\begin{multline}\label{may18}
\sum_{\substack{{|\gamma_k-E|<N^{-1/2}}\\ \phi^{100} N^{-1} \le |\mu_k(t)-E|}}\log\left(\frac{z-\mu_k(t)}{z-\lambda_k(t)}\right)\\
=\sum_{\substack{{|\gamma_k-E|<N^{-1/2}}\\ \phi^{100} N^{-1} \le |\mu_k(t)-E|}}
(\lambda_k(t)-\mu_k(t))\frac{1}{z-\lambda_k}+\OO\left(\sum_{\substack{{|\gamma_k-E|<N^{-1/2}}\\ \phi^{100} N^{-1} \le |\mu_k(t)-E|}}\frac{|\lambda_k(t)-\mu_k(t)|^2}{|z-\lambda_k(t)|^2}\right).
\end{multline}
The above error term is again bounded by rigidity,  as it is of order 
$\frac{\varphi^2}{N^2}\sum_{k\geq \varphi^{100}}\frac{N^2}{k^2}=\oo(1)$.
For the first term on the left side of \eqref{may18},  we use Proposition~\ref{prop:homog} to write it as
\begin{equation}\label{october11}
\frac{L^H(E,t)-L^{\rm GOE}(E,t)}{N}\sum_{\substack{{|\gamma_k-E|<N^{-1/2}}\\ \phi^{100} N^{-1} \le |\mu_k(t)-E|}}\frac{1}{z-\lambda_k}+\OO\left(\frac{N^\e}{N^{3/2}t}\right)\sum_{\substack{{|\gamma_k-E|<N^{-1/2}}\\ \phi^{100} N^{-1} \le |\mu_k(t)-E|}}\frac{1}{|z-\gamma_k|},
\end{equation}
with overwhelming probability.
The second error term above is again $\oo(1)$ by rigidity.
For the first term in \eqref{october11}, we work on $\mathcal G_\eta$ to obtain the high-probability bound
\begin{align}\label{may182}
\frac{(\log N)^\e}{N}\bigg|\sum_{\substack{{|\gamma_k-E|<N^{-1/2}}\\ \phi^{100} N^{-1} \le |\mu_k(t)-E|}}\frac{1}{z-\lambda_k}\bigg|
&\leq
\frac{(\log N)^\e}{N}\bigg|\sum_{\substack{{|\gamma_k-E|<N^{-1/2}}\\ \phi^{100} N^{-1} \le |\mu_k(t)-E|}}\frac{1}{z-\gamma_k}\bigg|\\ &+
\frac{(\log N)^\e}{N}\bigg|\sum_{\substack{{|\gamma_k-E|<N^{-1/2}}\\ \phi^{100} N^{-1} \le |\mu_k(t)-E|}}\frac{|\lambda_k-\gamma_k|}{|z-\gamma_k|^2}\bigg|.
\end{align}
The second term above is again negligible by rigidity.
Let $k_0$ be the index that minimizes $|E - \gamma_k|$ for $k \le N$. In the right side of \eqref{may182}, the contribution from $|k-k_0|>N^{1/4}$ is negligible by rigidity. In this term, we may then replace the summation over indices such that $|\gamma_k-E|<N^{-1/2}$ and $\phi^{100} N^{-1} \le |\mu_k(t)-E|$ with one over indices $k$ such that $N\varphi^2<|k-k_0|<N^{1/4}$. In this replacement, any indices with $|\mu_k(t)-E| \le \phi^{100} N^{-1}$ may be restored as necessary using the arguments leading to \eqref{e:concludemediumk}.

Note that 
$
\Im \frac{1}{z-\gamma_k}  = \frac{\eta}{|z-\gamma_k|^2}
$
and
\beq\label{kbdds}
 \frac{c | k - k_0|}{N} \le |E -\gamma_k| \le  \frac{c^{-1}| k - k_0|}{N}
\eeq  for $k \ge \phi^2 N$ by rigidity and $|\gamma_{k_0} - E| \le C N^{-1}$. Then
\beq\label{bigkimag}
\sum_{\substack{{|\gamma_k-E|<N^{-1/2}}\\ \phi^{100} N^{-1} \le |\mu_k(t)-E|}} \Im \frac{1}{z-\gamma_k}
\le 
\sum_{N\varphi^2<k-k_0<N^{1/4}}
\frac{C N^2 \eta }{| k - k_0|^2} 
\le  \frac{C N \eta }{ \phi^2}.
\eeq

Further,
\begin{multline}\label{may184}
\re\sum_{N\varphi^2<k-k_0<N^{1/4}}\left(\frac{1}{z-\gamma_k}+\frac{1}{z-\gamma_{2k_0-k}}\right)=\\
\sum_{\varphi^2 N <k-k_0<N^{1/4}}\left(\frac{2E-\gamma_k-\gamma_{2k_0-k}}{|z-\gamma_k|^2}+(E-\gamma_{2k_0-k})\left(\frac{1}{|z-\gamma_k|^2}-\frac{1}{|z-\gamma_{2k_0-k}|^2}\right)\right).
\end{multline}
By a direct computation, we have 
\beq\label{cancelE}
2E-\gamma_k-\gamma_{2k_0-k}=\mathrm{O}\big((k-k_0)^2/N^2\big).
\eeq
  Together with \eqref{kbdds} this gives 
\beq\label{part3a}
\sum_{\varphi^2 N <k-k_0<N^{1/4}}\frac{2E-\gamma_k-\gamma_{2k_0-k}}{|z-\gamma_k|^2} = \OO( N^{1/4} ).
\eeq
Also, using \eqref{cancelE} and \eqref{kbdds}, we compute
\begin{align}\label{part3b}
\frac{1}{|z-\gamma_k|^2}-\frac{1}{|z-\gamma_{2k_0-k}|^2}
&= \frac{(|z-\gamma_{2k_0-k}|
-|z-\gamma_k|
)(|z-\gamma_{2k_0-k}|
+|z-\gamma_k|)}{|z-\gamma_k|^2|z-\gamma_{2k_0-k}|^2}\\
&=\OO\left(  \frac{C N^{-3 } |k - k_0|^3 }{N^{-4} (k - k_0)^4}\right)
 = \OO\left(\frac{N}{|k - k_0|}\right).
\end{align}
Putting \eqref{part3a} and \eqref{part3b} into \eqref{may184}, and combining this with \eqref{bigkimag} and our previous bounds, we find
\beq\label{e:largek}
\Bigg|\sum_{\substack{{|\gamma_k-E|<N^{-1/2}}\\ \phi^{100} N^{-1} \le |\mu_k(t)-E|}}\log\left(\frac{z-\mu_k(t)}{z-\lambda_k(t)}\right)\Bigg| = \oo(1).
\eeq
We finish the proof by combining \eqref{e:smallk}, \eqref{e:concludemediumk},  and  \eqref{e:largek}. 
\end{proof}

The following Proposition directly implies Theorem \ref{thm:maxWignerDivisible}.  In the statement  and its proof, we abbreviate
\begin{equation}\label{eqn:LNx}
L_N^{\bx}(z)=\sum_{k=1}^N\log(z-x_k)-N\int\log(z-\lambda)\rho_{\rm sc}(\lambda)\rd\lambda.
\end{equation}
Note that the proposition below uses Theorem \ref{thm:maxWigner} in its proof,  but there is no circularity in the sense that the proposition is not used in proof of Theorem  \ref{thm:maxWigner}.
\begin{proposition}\label{l:relax2}
Let  $\kappa,t>0$ be fixed.  Then tightness holds for the sequence of random variables 
\[
\Big(\sup_{|E|<2-\kappa}\re L_N^{\bmu_t}(E)-\sup_{|E|<2-\kappa}\re L_N^{\bla_t}(E)\Big)_{N\geq 1}.
\]
\end{proposition}

\begin{proof}
For the proof we define
\[
m_N=\sup_{|E|<2-\kappa}\re L_N^{\bmu_t}(E).
\]
We will prove that for any $\e>0$ there exists $C>0$ such that for all $N$ we have
\[
\mathbb{P}\left(\sup_{|E|<2-\kappa}\re L_N^{\bla_t}(E) \geq m_N-C\right)\geq 1- \e,
\]
and the  same result holds when permuting $\bmu$ and $\bla$. The claimed tightness then follows directly.

In the following, we  denote
\[
\ell_E=\frac{\Im L_N^{\bla_0}(E_t)-\Im L_N^{\bmu_0}(E_t)}{N\im \msc(E_t)},\ \ \eta=\frac{1}{N},\ \ {\tilde\eta}=\frac{\varphi^{100}}{N}.
\]
Consider the following events depending on $N$ and sometimes additional parameters $M,\delta>0$:
\begin{align}
A&=A(M)=\bigcap_{|E|<2-\kappa}\left\{\big|
\re L_N^{\bmu_t}(E+\ii{\tilde\eta})-\re L_N^{\bla_t}(E+\ii{\tilde\eta})
\big|\leq M\right\},\\
B&=B(M)\\
&=\!\!\bigcap_{|E|<2-\kappa}\!\!\left\{
\big|(\re L_N^{\bmu_t}(E\!+\!\ii\eta)-\re L_N^{\bmu_t}(E\!+\!\ii{\tilde\eta}))
\!-\!
(\re L_N^{\bla_t}(E\!+\!\ii\eta\!+\!\ell_E)\!-\!\re L_N^{\bla_t}(E\!+\!\ii{\tilde\eta}\!+\!\ell_E))
\big|\leq M\right\},\\
C&=C(M)=\bigcap_{|E|<2-\kappa}\left\{
\big|\re L_N^{\bla_t}(E+\ii{\tilde\eta}+\ell_E)
-\re L_N^{\bla_t}(E+\ii{\tilde\eta})
\big|\leq M\right\},\\
D&=D(M)=\left\{\sup_{E\in[-2+\kappa,2-\kappa]}(|L_N^{\bla_0}(E_t)|+|L_N^{\bmu_0}(E_t)|)\leq M\right\},\\
E&=E(\delta)=\left\{\sup_{E\in[-2+\kappa-\delta,2-\kappa+\delta]}\re L_N^{\bla_t}(E+\ii\eta)\leq(1+\frac{10}{N\delta})| \sup_{E\in[-2+\kappa-2\delta,2-\kappa+2\delta]}\re L_N^{\bla_t}(E)|+3\right\},\\
F&=F(\delta)=\left\{\sup_{E\in[-2+\kappa-2\delta,-2+\kappa]\cup[2-\kappa,2-\kappa+2\delta]}\re L_N^{\bla_t}\leq \frac{\log N}{10}\right\},\\
G&=\left\{\sup_{E\in[2+\kappa,2-\kappa]}\re L_N^{\bla_t}\geq \frac{\log N}{2}\right\}\cap\left\{\sup_{E\in[2+\frac{\kappa}{2},2-\frac{\kappa}{2}]}\re L_N^{\bla_t}\leq 10\log N\right\}.
\end{align}
From (\ref{eqn:UBRe1}),  there is some fixed  $C_1>0$  such that,  for any $N$, there exists $|E_0|<2-\kappa$ such that
\[
\re L_N^{\bmu_t}(E_0+\ii\eta)\geq m_N-C_1.
\]
Therefore,  on $A$ we have 
\[
\re L_N^{\bmu_t}(E_0+\ii\eta)-\re L_N^{\bmu_t}(E_0+\ii{\tilde\eta})+\re L_N^{\bla_t}(E_0+\ii{\tilde\eta})\geq m_N-C_1-M,
\]
and on $A\cap B$ we can write
\[
\re L_N^{\bla_t}(E_0+\ii\eta+\ell_{E_0})-\re L_N^{\bla_t}(E_0+\ii{\tilde\eta}+\ell_{E_0})+\re L_N^{\bla_t}(E_0+\ii{\tilde\eta})\geq m_N-C_1-2M.
\]
On $A\cap B\cap C$ we therefore have
\[
\re L_N^{\bla_t}(E_0+\ii\eta+\ell_{E_0})\geq m_N-C_1-3M.
\]
Assuming our parameters satisfy $M/N=\oo(\delta)$ and $\frac{\log N}{N\delta}=\oo(1)$, on
$A\cap B\cap C\cap D\cap E\cap G$ this yields
\[
\sup_{|E|<2-\kappa+2\delta}\re L_N^{\bla_t}(E)\geq m_N-C_1-4M-10.
\]
Finally on $A\cap B\cap C\cap D\cap E\cap F\cap G$ we obtain
\[
\sup_{|E|<2-\kappa}\re L_N^{\bla_t}(E)\geq m_N-C_1-4M-10.
\]
The proof will therefore be complete if we obtain that,  for fixed $\e>0$,  each one of the ensembles $A,B,C,D,E,F,G$ has probability larger than $1-\e$ for large enough $N$.
For $A(M), C(M),D(M), E(\delta)$ and $F(\delta)$ this will be true for the choice $\delta=N^{-1+\theta}$,  $\theta\in(0,\frac{1}{10})$ arbitrary,  and $M$ fixed, large enough.

First,  Corollary \ref{lem:maxSmoothField} gives $\mathbb{P}(D(M))>1-\e$ for some fixed $M=M(\e,
\kappa)$ and large enough $N$.
Then from (\ref{eqn:StochAdv}) we have $\mathbb{P}(A(M+1))\geq \mathbb{P}(A(M+1)\cap D(M))\geq 1-\e$ for large enough $N$.

To bound $\mathbb{P}(C)$, note that $\inf_{|E|<2-\kappa}\im \msc(E_t)>c(\kappa)>0$ so that on $D(M)$ we have $\ell_{E}=\OO(1/N)$.  Therefore  on the rigidity event from \eqref{e:rigidity} and on $D(M)$,  
a Taylor expansion gives
\begin{multline}\left|\re L_N^{\bla_t}(E+\ii{\tilde\eta}+\ell_{E})
-\re L_N^{\bla_t}(E+\ii{\tilde\eta})\right|\leq\ell_{E}\left|\sum_i\frac{1}{E+\ii{\tilde\eta}-\lambda_i(t)}\right|+\OO\left(\ell_{E}^2\sum_i\frac{1}{|E+\ii{\tilde\eta}-\lambda_i(t)|^2}\right)\\
\leq C_2 N\ell_{E}.
\end{multline}
for some fixed $C_2=C_2(\kappa)$.
Hence there exists $\tilde M=\tilde M(\kappa,\e)$ such that $\mathbb{P}(C(\tilde M))>1-\e$ for large enough $N$.

Moreover,  Theorem \ref{thm:maxWigner} implies  $\mathbb{P}(G)\geq 1-\e$ for large enough $N$, and similarly (\ref{eqn:meso}) implies $\mathbb{P}(F)\geq 1-\e$ for $\theta<1/10$.

We now prove that for $\theta>0$ we have $\mathbb{P}(E)>1-\e$ for large enough $N$.  
First,  similarly to (\ref{eqn:Harmo3}),  with probability $ 1-\OO(N^{-20})$ we have, for any $|E|<2-\kappa+\delta$,  for our choice  $\delta=N^{-1+\theta}$, 
\begin{equation*}
\left|\Re L_N^{\bla_t}(E+\ii\eta)-\int_{[-2+\frac{\kappa}{2},2-\frac{\kappa}{2}]}\Re L_N^{\bla_t}(u)\cdot \frac{\eta}{\eta^2+(E-u)^2}\frac{\rd u}{\pi}\right|\leq 1.
\end{equation*}
Moreover,  from Theorem \ref{thm:maxWigner},  with probability $1-\oo(1)$ we have
\[
\int_{[-2+\frac{\kappa}{2},-2+\kappa-2\delta]\cup[2-\kappa+2\delta,2-\frac{\kappa}{2}]}\Re L_N^{\bla_t}(u)\cdot \frac{\eta}{\eta^2+(E-u)^2}\frac{\rd u}{\pi}\leq 10\log N\cdot \frac{\eta}{\delta}\leq 1
\]
for any $|E|<2-\kappa+\delta$.  With above two equations we obtain,  with probability $1-\oo(1)$, for any $|E|<2-\kappa+\delta$,
\begin{multline*}
\Re L_N^{\bla_t}(E+\ii\eta)\leq \int_{[-2+\kappa-2\delta,2-\kappa+2\delta]}\Re L_N^{\bla_t}(u)\cdot \frac{\eta}{\eta^2+(E-u)^2}\frac{\rd u}{\pi}+2\\
\leq
\sup_{|u|<2-\kappa+2\delta}\Re L_N^{\bla_t}(u)\cdot \int_{[-2+\kappa-\delta,2-\kappa+\delta]}\frac{\eta}{\eta^2+(E-u)^2}\frac{\rd u}{\pi}+2\leq(1+10\frac{\eta}{\delta})|\sup_{|u|<2-\kappa+2\delta}\Re L_N^{\bla_t}(u)|+2,
\end{multline*}
which concludes the proof that $\mathbb{P}(E)>1-\e$ for large enough $N$.

Finally, we prove that $\mathbb{P}(B)>1-\e$ for large enough $N$.
First,  we can easily ignore the contribution from eigenvalues close to the edge, because for any $|E|<2-\kappa$, we have
\begin{align*}
&  \sum_{j:||\gamma_j|-2|<\frac{\kappa}{10}}
\!\!\!\!\!
\!\!
\left|
(\log(E\!+\!\ii\eta\!-\!\mu_j(t))
\!-\!\log(E\!+\!\ii{\tilde\eta}\!-\!\mu_j(t)))
\!-\!
(\log(E\!+\!\ii\eta\!+\!\ell_{E}\!-\!\la_j(t))
\!-\!\log(E\!+\!\ii{\tilde\eta}\!+\!\ell_{E}\!-\!\la_j(t)))
\right|\\
&\leq 
\sum_{j:||\gamma_j|-2|<\frac{\kappa}{10}}{\tilde\eta}\left|
\frac{1}{E+\ii\eta-\mu_j(t)}-\frac{1}{E+\ii\eta+\ell_{E}-\la_j(t)}
\right|\\
&\qquad \qquad +C
\sum_{j:||\gamma_j|-2|<\frac{\kappa}{10}}{\tilde\eta}^2\left|
\frac{1}{|E+\ii\eta-\mu_j(t)|^2}+\frac{1}{|E+\ii\eta+\ell_{E}-\la_j(t)|^2}
\right| \\
&\leq N^{-1+\tilde \e},
\end{align*}
where the last inequality holds on $D$ and the rigidity event, for any fixed $\tilde\e>0$ and $N$ large enough.

Now fix $\alpha$ such that $\gamma_\alpha-2<\kappa/10$. From Proposition \ref{prop:homog}, with high probability we have, for any $k \in\llbracket\alpha N,(1-\alpha)N\rrbracket$ ,
\[
\Big|\la_k(t)-\mu_k(t)-\ell_{\gamma_k}\Big|<N^{-2+\tilde\e}.
\]
Thus choosing $k_0$ such that  $|E-\gamma_{k_0}|\leq \frac{C_3}{N}$,   $C_3=C_3(\kappa)$,  and $\ell_{\gamma_k} - \ell_E=\OO(\varphi^2(|k-k_0|+1)/N^2)$ (from (\ref{elem1}) and (\ref{eqn:MesoApprox})) we have
\[
\Big|\la_k(t)-\mu_k(t)-\ell_{E}\Big|<C_4\frac{N^{\tilde\e}+|k-k_0|}{N^2}
\]
for some $C_4=C_4(\kappa)$.
Note that for $k$ close to $k_0$, the above error term is much smaller than the regularization scale $\eta$, so that by Taylor expansion we obtain
\[
\sum_{\hat j\geq \alpha N}\left|
\log(E+\ii\eta-\mu_j(t))
-
\log(E+\ii\eta+\ell_{E}-\la_j(t))
\right|
\leq 
\frac{C_5}{N^2}\sum_{\hat j\geq \alpha N}\left|
\frac{N^{\tilde\e}+|j-k_0|}{\frac{|j-k_0|+1}{N}}
\right| 
\leq C_6
\]
for some constants  $C_5(\kappa,\e)$, $C_6(\kappa,\e)$.
This concludes the proof with the choice $M=C_6+1$.
\end{proof}

\subsection{Large moments.}\label{s:relaxLarge}\ 
We now prove quantitative relaxation for large moments of $\Im L_N$.  We denote $\overline{{\rm Tr}}h(M)={\rm Tr}h(M)-N\int h(x)\rho_{\rm sc}(x)\rd x$

\begin{lemma} \label{l:newcomparisonGD}
Let $H=H_0$ be a symmetric Wigner matrix and $\kappa,\e,A>0$.  Then there exists $N_0=N_0(\kappa,\e,A)$ such that for any $E\in[-2+\kappa,2-\kappa]$,  $1 \le p \le A\log N$ and  $t\in[\varphi^A,1]$ we have
\[
 \E \left[(\overline{{\rm Tr}}\mathds{1}_{[E,\infty)}(H_t))^{2p}\right]\leq \left(\frac{2}{e\pi^2}+\e\right)^p p^{p} (\log N)^{p}.
\]
The same bound holds for Hermitian Wigner matrices, with the prefactor $(\frac{1}{e\pi^2}+\e)^p$.
\end{lemma}
\begin{proof}
We first prove an equivalent concentration result for the Gaussian ensembles,  and then extend it to the Gaussian divisible ensemble thanks to our results on relaxation of the Dyson Brownian motion.\\

\noindent{\it First step: Concentration for GUE and GOE.} 
The proof first requires some concentration estimates for the GUE and GOE.
From \cite[Theorem 1.1]{ItsKra2008},  there is a $C_1$ such that uniformly in $\gamma$ in any compact set and $x\in[-2+\kappa,2-\kappa]$ we have
\begin{equation}\label{eqn:ItsKra}
\mathbb{E}_{{\rm GUE}}[e^{\gamma \pi \overline{{\rm Tr}}\mathds{1}_{[x,\infty)}(G)}]\leq C_1 e^{\frac{\gamma^2}{4}\log N}.
\end{equation}
With Markov's inequality, optimization in $\gamma$ gives
\begin{equation}\label{eqn:optimum}
\mathbb{P}_{\rm GUE}(|\overline{{\rm Tr}}\mathds{1}_{[E,\infty)}(G)|>x)\leq C_1e^{\frac{\gamma^2}{4}\log N-\gamma \pi x}\leq
C_1e^{-\frac{\pi^2x^2}{\log N}}\mathds{1}_{x\leq \frac{B\log N}{2}}+C_1e^{\frac{B^2\pi^2}{4}\log N-B\pi^2x}\mathds{1}_{x\geq \frac{B\log N}{2}}
\end{equation}
for any fixed, arbitrarily large $B$,  uniformly in $-2+\kappa\leq E\leq 2-\kappa$, with $C_1=C_1(B,\kappa)$.

For the GOE, we follow \cite[Proof of Lemma 23]{Oro2010} and write the equality in law
\begin{equation}\label{eqn:magic}
\overline{{\rm Tr}}\mathds{1}_{[x,\infty)}(G_2)=\frac{1}{2}[\overline{{\rm Tr}}\mathds{1}_{[x,\infty)}(G_1)+\overline{{\rm Tr}}\mathds{1}_{[x,\infty)}(G'_1)]+X
\end{equation}
where $G_2$ is a GUE,  $G_1$ and $G'_1$ are independent GOE, and $X$ is a random variable satisfying $|X|\leq 2$ almost surely.
Together with (\ref{eqn:ItsKra}) this implies 
\[
\left(\mathbb{E}_{{\rm GOE}}[e^{\gamma \pi\overline{{\rm Tr}}\mathds{1}_{[x,\infty)}(G)}]\right)^2\leq e^{4\gamma\pi}\mathbb{E}_{{\rm GUE}}[e^{2\gamma \overline{{\rm Tr}}\mathds{1}_{[x,\infty)}(G)}]\leq C_2e^{\gamma^2\log N}.
\]
Similarly to (\ref{eqn:optimum}), we conclude that 
\begin{equation}\label{eqn:optimum2}
\mathbb{P}_{\rm GOE}(|\overline{{\rm Tr}}\mathds{1}_{[E,\infty)}(G)|>x)\leq 
C_2e^{-\frac{\pi^2x^2}{2\log N}}\mathds{1}_{x\leq B\log N}+C_2e^{\frac{B^2\pi^2}{2}\log N-B\pi^2x}\mathds{1}_{x\geq B\log N}
\end{equation}
for any fixed, arbitrarily large $B$,  uniformly in $-2+\kappa\leq E\leq 2-\kappa$, with $C_2=C_2(B,\kappa)$.\\

\noindent {\it Second step: The weakly Gaussian-divisible ensemble.}\ We only consider the symmetric universality class (the proof for the Hermitian one is identical from now),  a weakly Gaussian divisible matrix $H_t$ with spectrum $\bla(t)$ coupled with the spectrum $\bnu(t)$ of $G$, a GOE matrix.  Define the good set
\begin{multline*}
\mathscr{G}=\bigcap_{\kappa N\leq k\leq (1-\kappa)N}\left\{\Big|\la_k(t)-\mu_k(t)-\frac{\Im L^H_N(\gamma_k^t)-\Im L^G_N(\gamma_k^t)}{N\im \msc(\gamma_k^t)}\Big|\leq\frac{N^{\e}}{N^2 t}\right\}\\
\bigcap_{1\leq j\leq N,s\in\{0,t\}}
\left\{
| \lambda_j(s)- \gamma_j|+|\mu_j(s)-\gamma_j|
 \le 2 \varphi\, \hat j^{-\frac{1}{3}} N^{-\frac{2}{3}} 
\right\}
.
\end{multline*}
From Remark \ref{rem:betterprobabilitybound} and (\ref{e:rigidity}) we have $\mathbb{P}(\mathscr{G})\geq 1-e^{-\delta\varphi^\delta}$ for a fixed $\delta>0$.
For $p=\OO(\log N)$ this implies
\[
\E \left[ |\overline{\rm Tr}\mathds{1}_{[x,\infty)}(H_t)|^{2p} \mathds{1}_{\mathscr{G}^{\rm c}}\right]\leq N^{2p}\cdot e^{-\delta \varphi^\delta/2}\ll1.
\]
We now bound
\[
\E[|\overline{\rm Tr}\mathds{1}_{[x,\infty)}(H_t)|^{2p}\mathds{1}_{\mathscr{G}}]=2p\int_0^{100\varphi}u^{2p-1}
\left(\mathbb{P}\left(\overline{\rm Tr}\mathds{1}_{[x,\infty)}(H_t)>u,\mathscr{G}\right)+\mathbb{P}\left(\overline{\rm Tr}\mathds{1}_{[x,\infty)}(H_t)<-u,\mathscr{G}\right)\right)\rd u.
\]
We  only consider $\mathbb{P}\left(\overline{\rm Tr}\mathds{1}_{[x,\infty)}(H_t)>u,\mathscr{G}\right)$,  as the same proof applies to the other term.  We define $\gamma_k$ the quantile closest to $x$,  $n=\lfloor  u/\pi\rfloor$ and
$j=k-n+2$. 
With (\ref{eqn:equivRig}) and the definition of $\mathscr{G}$, there is some $c=c(\kappa)>0$ such that for any $\theta\in[0,1/10]$ (eventually we will choose $\theta\to 0$) we have, 
\begin{multline*}
\mathbb{P}\left(\overline{\rm Tr}\mathds{1}_{[x,\infty)}(H_t)>u,\mathscr{G}\right)\leq
\mathbb{P}\left(\lambda_{j}(t)-\gamma_{j}>\gamma_k-\gamma_{j},\mathscr{G}\right)\\
\leq
\mathbb{P}_{\rm GOE}\left(\mu_{j}(t)-\gamma_{j}>(1-\theta)(\gamma_k-\gamma_{j})\right)
+
\mathbb{P}(|\Im L^H_N(\gamma_{j}^t)|>c\theta u)
+
\mathbb{P}(|\Im L^G_N(\gamma_{j}^t)|>c\theta u).
\end{multline*}
We first bound the above GOE probability for different ranges of $u$ (remember $j=j(u)$).  From (\ref{eqn:equivRig2}) and  (\ref{eqn:optimum2}) we have the following: For any $\e,\kappa,B>0$ there exists $C_3(\e,\kappa,B)>0$ such that for any $x\in[-2+\kappa,2-\kappa]$ and $u\in[0,100\varphi]$ we have 
\[
\mathbb{P}_{\rm GOE}\left(\mu_{j}(t)-\gamma_{j}>(1-\theta)(\gamma_k-\gamma_{j})\right)\leq 
C_3 e^{-(1-\e)(1-\theta)^2 \frac{u^2\pi^2}{2\log N}}\mathds{1}_{u<B\log N}+
C_3 e^{\frac{B^2\pi^2}{2}\log N-B(1-\theta)(1-\e)\pi^2u}\mathds{1}_{u\geq B\log N}.
\]
This implies
\begin{multline}
2p\int_0^{100\varphi}u^{2p-1}\mathbb{P}_{\rm GOE}\left(\mu_{j}(t)-\gamma_{j}>(1-\theta)(\gamma_k-\gamma_{j})\right)\rd u\\
\leq C_3 2p\int_0^\infty   u^{2p-1}e^{-(1-\e)(1-\theta)^2 \frac{u^2\pi^2}{2\log N}}\rd u
+
C_3 2p\int_{B\log N}^{\infty}u^{2p-1}e^{\frac{B^2\pi^2}{2}\log N-B(1-\theta)(1-\e)\pi^2u}\mathds{1}_{x\geq B\log N}\rd u.
\end{multline}
The first term is bounded with  $C_4(B,\kappa,\e)\left(\frac{2}{e\pi^2}+\alpha(\e,\theta)\right)^p p^{p} (\log N)^{p}$ by induction on $p$,
where $\alpha(\e,\theta)\to 0$ as $\e,\theta\to 0$.

For the second term, if $p<B(\log N)/10$ it is bounded with
\[C_32pe^{-\frac{B^2\pi^2(1-\theta)(1-\e)\log N}{2}}(B\log N)^{2p}\leq 2 C_3\Big(\frac{2}{e\pi^2}+\nu(\theta,\e)\Big)^p p^{p+1} (\log N)^{p},\]
where $\nu(\theta,\e)\to 0$ as $\theta,\e\to 0$,  and we have used $\sup_{x>0}x^{p}e^{-\frac{x\pi^2}{2}}=p^{p}(2/(e\pi^2))^{p}$. 
We have therefore proved that for any $\alpha>0$,  for $\e\leq \e_0(\kappa,\alpha)$ and $\theta\leq \theta_0(\kappa,\alpha)$,  $p<B(\log N)/10$ and $N\geq N_1(\alpha,\kappa,B)$ we have
\[
2p\int_0^{100\varphi}u^{2p-1}\mathbb{P}_{\rm GOE}\left(\mu_{j}(t)-\gamma_{j}>(1-\theta)(\gamma_k-\gamma_{j})\right)\rd u\leq \big(\frac{2}{e\pi^2}+\alpha\big)^pp^p(\log N)^p.
\]
We now consider
$
2p\int_0^{100\varphi}u^{2p-1}\mathbb{P}(|\Im L^{\rm GOE}_N(\gamma_{j}^t)|>c\theta u)\rd u.
$
Note that this could not be directly interpreted as a moment because $j=j(u)$.
A direct calculation based on (\ref{eqn:pointBD}) gives, for any $p<D(\log N)/10$,
\[
2p\int_0^{100\varphi}u^{2p-1}\mathbb{P}(|\Im L^{\rm GOE}_N(\gamma_{j}^t)|>c\theta u)\rd u\leq \frac{C_7^p}{\theta^{2p}}(\log\log N)^{2p}p^{\frac{3p}{2}},
\]
where $C_7=C_7(\kappa,D)$. The same estimate holds for $
2p\int_0^{100\varphi}u^{2p-1}\mathbb{P}(|\Im L^H_N(\gamma_{j}^t)|>c\theta u)\rd u.
$
We choose $\theta\to 0$ satisfying $\theta\geq (\log N)^{-1/100}$, so that for any $p<D(\log N)/10$ and $N\geq N_2(\kappa,\alpha,D)$ we have
\[
\frac{C_7^p}{\theta^{2p}}(\log\log N)^{2p}p^{\frac{3p}{2}}\leq \big(\frac{2}{e\pi^2}+\alpha\big)^pp^p(\log N)^p.
\]
This concludes the proof.
\end{proof}

\section{Moment Matching}\label{sec:MomMat}

This section contains moment matching lemmas that are used in the next section to establish our main results for Wigner matrices. \Cref{s:mmreal} provides a comparison result for the real part of the log-characteristic polynomial. \Cref{s:mmdeviations} establishes results for the deviations of the eigenvalues from their classical locations. 

\subsection{Real part of log-characteristic polynomial.}\label{s:mmreal}\ 
Given parameters $r>0$ and $x \in [N^{-1}, 1]$, we define the line segment
\begin{equation}
 \mathcal L_{r,x } = \{ z = E + \I \eta \in \mathbb{H} : |E| < 2 -r, \eta = x   \}.
\end{equation}
Given $M \in \matn$ with eigenvalues $\{ \lambda_i \}_{i=1}^N$, 
we will study the observable
\begin{equation}
	\label{e:maxa}
	\max_{i \in J} \left( \sum_{j} \log | z_i  - \lambda_j | - N \int_{\R} \log |z_i - \lambda|\, \rd\sc(\lambda) \right),
\end{equation} 
where $J$ is an index set satisfying $|J| \le CN$ for some constant $C>0$ and the points $\{z_i \}_{i\in J}$ satisfy $z_i \in \mathcal L_{r,x}$.
 We set
\begin{equation}\label{alphadef}
	\alpha_i   =  \sum_{j} \log | z_i  - \lambda_j | - N \int_{\R} \log |z_i - \lambda|\, d\sc(\lambda),\qquad 
	\bm \alpha = (\alpha_i)_{i \in J}.
\end{equation} 
 Using the fundamental theorem of calculus, we write \begin{equation}
\log | z_i  - \lambda_j |   = \log | \lambda_j - E - \I N^{100}| - \Im \int_{x}^{N^{100}} \frac{d\eta}{\lambda_j - E - \I\eta }.
\end{equation}
Then 
\begin{align}\label{may183}
\alpha_i  &= \sum_{j} \log | \lambda_j - E - \I N^{100}| - \Im \int_{x}^{N^{100}} N m_N(E + \I \eta) \, d\eta- N \int_{\R} \log |z_i - \lambda|\, d\sc(\lambda).
\end{align}
We also define a regularized version of $\alpha_i$: 
\begin{align*}
\tilde \alpha_i &= \sum_{j} \log | - E - \I N^{100}| - \Im \int_{x}^{N^{100}} N m_N(E + \I \eta) \, d\eta- N \int_{\R} \log |z_i - \lambda|\, d\sc(\lambda).
\end{align*}
As before, we write $\bm {\tilde \alpha} = (\tilde \alpha_i)_{i \in J}$, and suppress the dependence of $\bm {\tilde \alpha}$ on $x$ and $r$ in the notation. 
The following lemma shows that $\max_{i\in J}\tilde \alpha_i$ is a good substitute for $\max_{i \in J}  \alpha_i$ (that is, \eqref{e:maxa}).

\begin{lemma}\label{alphaapprox}
Let $H$ be a Wigner matrix and fix $r >0$. 
Then there exists $C(r)>0$ such that for all $x \in [N^{-1}, 1]$,
\begin{equation}\label{e:maxapprox}
\sup_{z \in \mathcal L_{r,x} } \P \left(\| \bm \alpha- \bm {\tilde \alpha} \|_\infty  > C N^{-10} \right) \le C N^{-D}.
\end{equation}
\end{lemma}
\begin{proof}
This follows from differentiating $y\mapsto \log |  y - E - \I N^{100}|$ in $y$, then using the eigenvalue rigidity estimate \eqref{e:rigidity} and the fundamental theorem of calculus.
\end{proof}

Given a vector $\bm w \in \R^{|J|}$ and parameters $\delta, \nu >0$, we introduce the regularized maximum observable denoted
\beq
F(\bm w) = F _{\delta,\nu}(\bm w) = \frac{1}{\delta} \log \left( \sum_{i \in J} \exp \left( \delta  \nu w_i \right)\right).
\eeq
The unusual notation $\delta$ for an inverse temperature aims at avoiding confusion with the $\beta$-ensembles.
The following lemma is elementary and its proof is omitted.
\bel\label{l:entropy} For any $\bm w \in \R^{|J|}$, we have 
\beq\label{e:entropy}
\left|  \sup_{i \in J } \nu w_i  -  F_{\delta}(\bm w) \right|  < \frac{2 \log N}{\delta}.\eeq
\eel

For the rest of this section, we fix $\delta = (\log N)^2$ and $\nu = 1/ \log N$, so that
\begin{equation}
	F (\bm {\tilde \alpha}) = \frac{1}{\delta} \left(  \sum_{i \in J} \exp(\delta \nu \tilde \alpha_i)  \right)
\end{equation} 
 approximates \eqref{e:maxa} with $O\left((\log N)^{-1}\right)$ error, with high probability, by \eqref{e:entropy} and \eqref{e:maxapprox}.

\begin{definition}
For any $w \in [0,1]$, $M = \left( m_{ij} \right)_{1\le i,j \le N} \in \matn$, and indices $a,b \in \unn{1}{N}$, we define $\Theta^{(a,b)}_w M  \in \matn$ as follows. 
If $(i,j) \notin \{ (a,b), (b,a) \}$, let the $(i,j)$ entry of $\Theta^{(a,b)}_w M$ {be equal to $m_{ij}$.}  If $(i,j) \in \{ (a,b), (b,a) \}$, then let the $(i,j)$ entry equal  $w m_{a,b} = w m_{b,a}$. 
We also set $\Theta^{(a,b)}_w  G(z)  = (\Theta^{(a,b)}_w M - z)^{-1}$.
\end{definition}
We recall that $\phi$ was defined in \eqref{eqn:phi2}.
\begin{lemma}
Let $H$ be a Wigner matrix and fix $D,r>0$.
There exists $C(D,r)>0$ such that 
\begin{equation}\label{e:Gbelow}
\sup_{x\in [N^{-1}, 1]} \sup_{z \in \mathcal L_{r,x}} \P \left(\sup_{w \in [0,1]} \sup_{a,b, i \in \unn{1}{N}} | \Theta^{(a,b)}_w G_{ii}(z)| >  C  \varphi^{10} \right) \le C N^{-D}.
\end{equation} 

\end{lemma}
\begin{proof}
For the unperturbed matrices, $w=1$, this is an immediate consequence of \eqref{e:gentries}. The statement for a general rank-one perturbations can be deduced from the unperturbed case using a resolvent expansion; see \cite[(4.54)]{LanLopMar2018} and the following material for details.
\end{proof}

\begin{lemma}
Let $H$ be a Wigner matrix and fix $D,r>0$. There exists $C(D,r) >0$ such that for all $x \in [N^{-1}, 1]$, 
\begin{equation}\label{e:dop}
 \P \left(\sup_{w \in [0,1]} \sup_{k \in \unn{1}{5}} \sup_{a,b, c,d \in \unn{1}{N}} \left|\partial^k_{ab} \tilde \alpha_i( \Theta^{(c,d)}_w H) \right| > C \varphi^{ 11 k }\right) \le C N ^{-D}, \end{equation}
 and
\begin{equation}\label{e:das}
\sup_{w \in [0,1]} \sup_{k \in \unn{1}{5}} \sup_{a,b, c,d \in \unn{1}{N}} \left|\partial^k_{ab} \tilde \alpha_i( \Theta^{(c,d)}_w H) \right| \le C N^{ C}
\end{equation}
almost surely for all $N$.

\end{lemma}

\begin{proof} The first and third terms in the definition of $\tilde \alpha_i$ are constants and have derivative zero.
For the second, we see using \eqref{e:gpartial} that
\begin{equation}
\partial_{ab} N m_N = \partial_{ab} \sum_i G_{ii}  =  - \sum_i G_{ia} G_{bi}.
\end{equation} 
Therefore
\begin{equation}
  N \left|\partial_{ab} m_N \right| \le \sum_i \left|G_{ia} G_{bi}\right| \le C \sum_{i} \left(|G_{ia}|^2 + |G_{bi}|^2\right) \le \frac{C}{\eta} \left(
|G_{aa}| + |G_{bb}| \right),
\end{equation}
where we used \eqref{e:ward} in the last inequality. Similarly, for the higher derivatives we have 
\begin{equation}
N \left|\partial^k_{ab} m_N \right| \le \frac{C}{\eta}
 \left( |G_{aa}| + |G_{bb}| + |G_{ab}|  \right)^k \le \frac{C\varphi^{k 10} }{\eta}
\end{equation}
by \eqref{e:Gbelow}.
Then
\begin{equation}
\left|\partial^k_{ab} \tilde \alpha_i \right| = \left|\partial^k_{ab} \Im \int_{N^{-1}}^{N^{100}} N m_N(E + \I \eta) \, d\eta \right| \le C \varphi^{k 10}  \int_{N^{-1}}^{N^{100}} \frac{1}{\eta}\, d \eta \le C \varphi^{1 + k 10},
\end{equation}
where we increased the value of $C$. The remaining claim is similar and uses the trivial bound $|G_{ij}| \le \eta^{-1}$ from \eqref{e:gtrivial}. This completes the proof.
\end{proof}

\begin{lemma}\label{Fderiv}
Let $H$ be a Wigner matrix and fix $D,r>0$. Then there exist $C(D,r) >0$ such that for all $x \in [N^{-1}, 1]$
\beq \label{e:TF}
\P \left(\sup_{ k \in \unn{1}{5}} \sup_{a,b,c,d \in \unn{1}{N} } \sup_{w \in [0,1]} \left|  \partial_{ab}^k F \left(  \bm {\tilde \alpha} (  \Theta^{(c,d)}_w H) \right) \right| > C \varphi^{12 k }\right) \le C N^{-D}.
\eeq
Also, we have almost surely that for all $x \in [N^{-1}, 1]$
\beq \label{e:as}
\sup_{ k \in \unn{1}{5}} \sup_{a,b,c,d \in \unn{1}{N} }\sup_{w \in [0,1]}  \left|  \partial_{ab}^k F \left(  \bm {\tilde \alpha} (  \Theta^{(c,d)}_w H) \right)\right| \leq C N^{C}. \eeq
\end{lemma}
\begin{proof}
First, we claim that the partial derivatives of $F_\delta (w)$ with respect to the entries of the vector $ w \in \R^N$ satisfy
\begin{equation} \label{e:bre4}
 \sum_{\underline{j}} \left|\frac{\partial^d F_\delta ( w)}{\partial_{j_1}\dots \partial_{j_d}}\right| \le C \delta^{d-1} ,
 \end{equation}
for any $d \in \N$. Here the sum runs over all multi-indices $\underline{j} = (j_1, \cdots, j_d )$ with values in $[1, N]^d$,  $\partial_j = \partial_{w_j}$, and $C = C(d)>0$ is a constant. This inequality follows by straightforward differentiation, and complete details are given in \cite[Lemma 3.4]{LanLopMar2018}.

Using the chain rule, \eqref{e:bre4} and \eqref{e:dop} imply \eqref{e:TF}. Similarly, \eqref{e:bre4} and \eqref{e:das} imply \eqref{e:as}.
\end{proof}
\begin{theorem}\label{t:main}
Fix $r >0$. Let $H$ and $M$ be Wigner matrices such that $\E [H_{11}^k] = \E [m_{11}^k]$ for $1 \le k \le 3$ and $\left| \E [H_{11}^4] - \E [m_{11}^4] \right| \le K_1 N^{-2} \varphi^{-K_2}$ for some $K_1, K_2 \ge 0$. Let $S \colon \R \rightarrow \R$ be a smooth function satisfying $\| S^{(k)} \|_\infty \le K_1$ for $k \in \unn{0}{5}$. Then there exists $C(K_1)>0$ such that if $K_2 > C$, then for all $x\in[N^{-1},1]$ we have 
\begin{equation}
\left| \E_{H} S\left( F_\delta(\bm{\tilde \alpha}) \right) -  \E_{M} S \left( F_\delta(\bm{\tilde \alpha}) \right) \right| \le  C \varphi^{- K_2/2}.
\end{equation}
\end{theorem}

\begin{proof}

We fix $z\in \mathcal L_{r,x} $ and omit it from the notation. Fix any bijection
\begin{equation}
\psi \colon \{ ( i, j) : 1 \le i \le j \le N\} \rightarrow \unn{1}{\gamma_N},
\end{equation}
where $\gamma_N = N ( N + 1) /2$, and define the matrices $H^1, H^2, \dots , H^{\gamma_N}$ by 
\begin{equation}
H_{ij}^\gamma = 
\begin{cases}
H_{ij} & \text{if } \psi(i,j) \leq \gamma
\\
m_{ij} & \text{if } \psi(i,j) > \gamma
\end{cases}
\end{equation}
for $i \le j$.

Fix some $\gamma \in \unn{1}{\gamma_N}$ and consider the indices $(i,j)$ such that $\psi(i,j) = \gamma$. Define $T\colon \matn \rightarrow \R$ by $T(X) = S\left( F_\delta(\bm {\tilde \alpha}(X)) \right)$. We Taylor expand $T \left(H^\gamma\right)$ in the $(i,j)$ entry and write $\partial = \partial_{ij}$ to find

\begin{align*}
 T \left( H^\gamma \right) - T \left(  \Theta^{(i,j)}_0 H^\gamma \right) &= \partial T \left(  \Theta^{(i,j)}_0 H^\gamma \right) H_{ij} + \frac{1}{2!} \partial^2 T \left(  \Theta^{(i,j)}_0 H^\gamma \right) H_{ij}^2 
+ \frac{1}{3!}\partial^3 T \left(  \Theta^{(i,j)}_0 H^\gamma \right) H_{ij}^3\\
& 
+ \frac{1}{4!}\partial^4 T \left(  \Theta^{(i,j)}_0 H^\gamma \right) H_{ij}^4 + \frac{1}{5!}\partial^5 T \left(  \Theta^{(i,j)}_{w_1(\gamma)} H^\gamma \right) H_{ij}^5,
\end{align*}
where $w_1(\gamma) \in [0,1]$ is a random variable depending on $H_{ij}$. Similarly, we expand $T \left( H^{\gamma -1} \right)$ in the $(i,j)$ entry to obtain
\begin{align}
\label{e:taylorb11} T \left( H^{\gamma -1} \right) - T \left(  \Theta^{(i,j)}_0 H^\gamma \right) 
&= \partial T \left(  \Theta^{(i,j)}_0 H^\gamma \right) m_{ij} + \frac{1}{2!} \partial^2 T \left(  \Theta^{(i,j)}_0 H^\gamma \right) m_{ij}^2 
+ \frac{1}{3!}\partial^3 T \left(  \Theta^{(i,j)}_0 H^\gamma \right) m_{ij}^3 \\
& \label{e:taylorb21} + \frac{1}{4!}\partial^4 T \left(  \Theta^{(i,j)}_0 H^\gamma \right) m_{ij}^4 + \frac{1}{5!}\partial^5 T \left(  \Theta^{(i,j)}_{w_2(\gamma)} H^\gamma \right) m_{ij}^5,
\end{align}
where $w_2(\gamma) \in [0,1]$ is a random variable depending on $m_{ij}$. Subtracting the previous two equations and taking expectation, we obtain
\begin{align}\label{e:4th}
  \E \left[ T \left( H^{\gamma } \right)  \right]- \E \left[ T \left(   H^{\gamma -1} \right) \right] &= \frac{1}{4!} \E \left[ \partial^4 T \left(  \Theta^{(i,j)}_0 H^\gamma \right) H_{ij}^4 \right] - \frac{1}{4!}\E \left[ \partial^4 T \left(  \Theta^{(i,j)}_0 H^\gamma \right) m_{ij}^4 \right] \\
  \label{e:fifthorder1old} &+ \frac{1}{5!} \E \left[ \partial^5 T \left(  \Theta^{(i,j)}_{w_1(\gamma)} H^\gamma \right) H_{ij}^5\right]
  - \frac{1}{5!} \E \left[ \partial^5 T \left(  \Theta^{(i,j)}_{w_2(\gamma)} H^\gamma \right) m_{ij}^5 \right],\end{align}
where we used that $\Theta^{(i,j)}_0 H^\gamma$ is independent from $H_{ij}$ and $m_{ij}$, and that $\E [ h^k_{ij} ] = \E [ m^k_{ij} ]$ for $k \in \unn{1}{3}$.

Because $H_{ij}$ and $m_{ij}$ are independent from $\Theta^{(i,j)}_0 H^\gamma$, we have
\begin{equation}
 \E \left[ \partial^4 T \left(  \Theta^{(i,j)}_0 H^\gamma \right) H_{ij}^4 \right] - \E \left[ \partial^4 T \left(  \Theta^{(i,j)}_0 H^\gamma \right) m_{ij}^4 \right]  =  \E \left[ \partial^4 T \left(  \Theta^{(i,j)}_0 H^\gamma \right) \right] \E\left[ H_{ij}^4 -m_{ij}^4 \right].
\end{equation}
By \eqref{e:TF}, \eqref{e:as}, and the assumptions on $S$, there exists $C(K_1)>0$ such that
\begin{equation}
\left| \E \left[ \partial^4 T \left(  \Theta^{(i,j)}_0 H^\gamma \right) \right] \right| \le C \varphi^{50}.
\end{equation}
We conclude using \eqref{e:4th} and the assumption on the fourth moments of $H_{ij}$ and $m_{ij}$ that
\begin{equation}
\left| \E \left[ \partial^4 T \left(  \Theta^{(i,j)}_0 H^\gamma \right) H_{ij}^4 \right] - \E \left[ \partial^4 T \left(  \Theta^{(i,j)}_0 H^\gamma \right) m_{ij}^4 \right] \right| \le CK_1 N^{-2} \varphi^{50 -K_2}.
\end{equation}
 The fifth order terms may be bounded similarly, and in fact are lower order since the fifth powers $h^5_{ij}$ and $m^5_{ij}$ contribute an additional factor of $N^{-1/2}$. Summing the Taylor expansions over all $O(N^2)$ indices $(i,j)$, we conclude that 
\begin{equation}
\left| \E \big[ T(H) \big] - \E \big[ T(M) \big] \right|  \le CK_1 \varphi^{50 -K_2}.
\end{equation}
 This completes the proof.
\end{proof}

\subsection{Maximal deviation from classical location.}\label{s:mmdeviations}\ 
 Using the rigidity and local law from Theorem \ref{l:lscl}, the proof of the following lemma is nearly identical to \cite[Lemma 3.2]{LanLopMar2018}.

\begin{lemma} \label{l:ev}
Fix $\kappa >0$. For all $i \in \unn{ \kappa N} {(1-\kappa)N}$, there exist smooth functions $\tilde \lambda_i \colon \matn \rightarrow \R$ such that the following holds. Suppose that $H$ is a real symmetric Wigner matrix.  There exist constants $C_1, C_2 >0$ such that, uniformly in $i$ and $k \in \unn{1}{5}$,
	\begin{equation} \label{e:evreg}
		\big|\tilde \lambda_i (H) - \lambda_i (H) \big| \leq \frac{ 1 }{ N \varphi^{C_2}}, \qquad \sup_{w \in [0,1]} \sup_{a,b,c,d \in \unn{1}{N}} \big| \partial^k_{ab} \tilde \lambda_i ( \Theta^{(c,d)}_w H ) \big| \leq \frac{ \varphi^{C_1  }}{N}
	\end{equation}
	 with  probability at least $1 - c^{-1} \exp( -c \phi)$. Here $\partial_{ab}=\partial_{X_{ab}}$ denotes the partial derivative with respect to the $(a,b)$-th matrix element.  

	Further, uniformly in $i$ and $k \in \unn{1}{5}$,  we have the deterministic bound
	 \begin{equation} \sup_{w \in [0,1]} \sup_{a,b,c,d \in \unn{1}{N}} \big| \partial^k_{ab} \tilde \lambda_i (\Theta^{(c,d)}_w H ) \big| \leq C_1 N^{C_1}.\end{equation}
\end{lemma}

We write $\bm \lambda = (\lambda_i)_{i \in \unn{1}{N}}$ and $\bm {\tilde \lambda} = (\lambda_i )_{i \in \unn{ \kappa N} {(1-\kappa)N}}$, using the notation of the previous lemma.  Set $J\subset \unn{ \kappa N }{ ( 1 - \kappa)N } $ and define the smoothed maximal deviation of a vector $\bm v \in \R^{|J|}$ from the classical eigenvalue locations $\gamma_i$ by
\begin{equation}
	\widehat F_{\delta}(\bm v)  = \frac{1}{\delta} \log \left( \sum_{i\in J} \exp\left(    \delta \nu_i (  v_i - \gamma_i ) \right)  +  \exp\left(    \delta \nu_i (  \gamma_i - v_i) \right)\right),
\end{equation}
where we set 
\beq
\nu_k =
\sqrt{\frac{\pi}{2}}\cdot \frac{k \rho_{\rm sc}(\gamma_k) }{\log N}
, \qquad\delta = (\log N)^2.
\eeq

We omit the proof of the following derivative bounds, since it is similar to the proof of Lemma~\ref{Fderiv}.

\begin{lemma}\label{Fhatderiv}
Let $H$ be a Wigner matrix and fix $D,r>0$. Then there exist $C(D,r) >0$ such that
\beq \label{e:TFnew}
\P \left(\sup_{ k \in \unn{1}{5}} \sup_{a,b,c,d \in \unn{1}{N} } \sup_{w \in [0,1]} \left|  \partial_{ab}^k \widehat  F \left(  \bm {\tilde \lambda} (  \Theta^{(c,d)}_w H) \right) \right| > C \varphi^{C j }\right) \le C N^{-D}.
\eeq
Also, we have almost surely that
\beq \label{e:asnew}
\sup_{ k \in \unn{1}{5}} \sup_{a,b,c,d \in \unn{1}{N} }\sup_{w \in [0,1]}  \left|  \partial_{ab}^k \widehat F \left(  \bm {\tilde \lambda} (  \Theta^{(c,d)}_w H) \right)\right| \leq C N^{Cj}. \eeq
\end{lemma}

Using the observable $F_{\delta}(\bm {\tilde \lambda})$ and Lemma~\ref{Fhatderiv}, 
the proof of the following comparison result is similar to Lemma~\ref{t:main}.

\begin{theorem}\label{t:main2}
Fix $\kappa >0$.
Let $H$ and $M$ be Wigner matrices such that $\E [H_{11}^k] = \E [m_{11}^k]$ for $1 \le k \le 3$ and $\big| \E [H_{11}^4] - \E [m_{11}^4] \big| \le K_1 N^{-2} \varphi^{-K_2}$ for some $K_1, K_2 \ge 0$. Let $S \colon \R \rightarrow \R$ be a smooth function satisfying $\| S^{(k)} \|_\infty \le K_1$ for $k \in \unn{0}{5}$. Then there exists $ C(K_1)$ such that, if $K_2 > C$, then 
\begin{equation}
\left| \E_{H} S\left( \widehat  F_{\delta}(\bm {\tilde \lambda}) \right) -  \E_{M} S \left( \widehat  F_{\delta}(\bm {\tilde \lambda}) \right) \right| \le
 C \varphi^{- K_2/2},
\end{equation}
and for any $i  \in \unn{ \kappa N} {(1-\kappa)N}$.
\end{theorem}

\subsection{Large moments.}\ 
For some parameter $\eta_1>0$ we define the function $f = f_E$ by
\beq\label{f}
f = 0 \text{ on } (-\infty, E] \cup [3, \infty), \qquad f = 1 \text{ on } [E + \eta_1, 2.5]
\eeq
\beq\label{fbounds}
\| f^{(k)} \|_{L^\infty(-\infty, 2)} 
\le 100 \cdot \eta_1^{-k}, \quad \| f^{(k)} \|_{L^\infty(2, \infty)} \le 100,
\eeq
for $k=1,2$.  
All results in this section hold for $\eta_1\in[1/N,c]$, but we now fix
\[
\eta_1=\frac{\sqrt{\log N}}{N},
\]
which will be enough to prove  part (ii) of Theorem \ref{thm:Wigner}, in Subsection \ref{subsec:rigidity}.

Moreover,  given $M\in \matn$ we define 
\begin{equation}\label{eqn:Xp}
 X(M) = \sum_{i=1}^N f_E (\tilde\lambda_i) - N \int_\R f_E(x) \, \rd \sc(x), \ \ \ X_p(M)=X(M)^{2p}.
\end{equation}

\begin{lemma}\label{Xderiv}
Let $H$ be a Wigner matrix, and fix $A,\kappa>0$. Then there exists $C(A,\kappa)>0$ such that
\beq \label{e:TFX}
\P \left(\sup_{ k \in \unn{1}{5}} \sup_{a,b,c,d \in \unn{1}{N} } \sup_{w \in [0,1]} \left|  \partial_{ab}^k X    \big(  \Theta^{(c,d)}_w H\big)  \right| > C \varphi^{C k }\right) \le C \exp( - C^{-1} \phi). 
\eeq
Also, we have 
\beq \label{e:asX}
\sup_{ k \in \unn{1}{5}} \sup_{a,b,c,d \in \unn{1}{N} }\sup_{w \in [0,1]}  \left|  \partial_{ab}^k X   \big(  \Theta^{(c,d)}_w H\big) \right| \leq C N^{C}. \eeq
\end{lemma}
\begin{proof} Since the proof is similar to previous estimates, we give details only in the $j=2$ case to illustrate the general principles involved.  We have 
\beq
\partial^2_{ab} f(\tilde \lambda) = f''( \tilde \lambda) (\partial_{ab} \tilde \lambda)^2 + 
f'(\tilde \lambda )\partial^2_{ab} \tilde \lambda .
\eeq
The $f''$ term is the most dangerous. As in the proof of the previous lemma, there are at most $N \eta_1 (\log N)^C$ eigenvalues in the interval $I_A$ where $f'$ is nonzero. By Lemma~\ref{l:ev} we have
\beq
f''_0( \tilde \lambda) (\partial_{ab} \tilde \lambda)^2 \le \eta_1^{-2} N^{-2} \varphi^{2C} \le C \phi^{4C} .
\eeq
where we used $\log N \le \phi$ and increased the constant $C$ if necessary. The claim then follows after redefining $C$. This shows \eqref{e:TFX} for $j=2$; the other $j$ are similar. The bound \eqref{e:asX} follows from the second inequality in \eqref{fbounds}.
\end{proof}
\bel Let $H$ be a Wigner matrix, and fix $A,\kappa >0$. There exists $c(A,\kappa)>0$ such that for all $E \in [ -2 + \kappa, 2 -\kappa]$, 
\beq\label{absX}
\sup_{c,d \in \unn{1}{N}} \sup_{w \in [0,1]}
\P\left(
 \left| X \big(\Theta^{(c,d)}_w H \big) 
 \right|  \ge c^{-1} \phi \right) 
\le c^{-1} \exp ( - c (\log N)^{c \log \log N})
\eeq
\eel
\begin{proof}
We give the details only for $w=0$, since the case of general $w$ follows by a straightforward perturbation argument. By Equation \ref{lemma:regLin}, it suffices to bound $\sum_{i=1}^N f (\lambda_i) - N \int_\R f \, \rd \sc$, and this is immediate from \eqref{e:rigidity}. 
\end{proof}

\bel Let $H$ be a Wigner matrix, and fix $A, \kappa>0$. There exists $c(\kappa) >0$ such that, for all $E\in [-2 + \kappa, 2-\kappa]$, 
\beq\label{lemma:regLin}
\P\left( \left|\sum_{i=1}^N f_E(\tilde \lambda_i) - \sum_{i=1}^N f_E( \lambda_i) \right| > (\log N)^{1/2} \right) 
\le c^{-1} \exp( - c \phi^c).
\eeq
\eel
\begin{proof}
Note that $f'$ is supported on the interval $I_A = [E, E + \eta_1]$. 
For $\lambda_i$ outside the support of $f'$, it is straightforward to replace $\lambda_i$ with $\tilde \lambda_i$.  There are at most $N$ such eigenvalues, and $f_0(\lambda_i) - f_0(\tilde \lambda_i) = O(N^{-1}\varphi^{-c})$ by Lemma~\ref{l:ev}, so the overall error from all such $\lambda_i$ is $O(\varphi^{-c})$. 

For the other eigenvalues, we know that $f'$ can be as large as $\eta_1^{-1} = N (A \log N)^{-1}$. In the interval $I_A$ there are at most $N \eta_1 (\log N)^C$ eigenvalues with probability at least $1-c^{-1} \exp( - c \phi^c)$ by the rigidity estimate  \eqref{e:rigidity}. We have $| \tilde \lambda - \lambda| \le N^{-1}\varphi^{-c}$. Then the accumulated error is 
\beq
N^{-1} \varphi^{-c} \eta_1^{-1} N \eta_1 (\log N)^C = \varphi^{-c} (\log N)^C = o(1).
\eeq
which is acceptable. 
\end{proof}

The next lemma considers a moment-matching argument for diverging moments. Such estimates for large moments appeared first in random matrix theory in
\cite{KnoYin2013,TaoVu2013}.  For the application to the accurate Gaussian decay exponent in Theorem \ref{thm:Wigner} (ii),  optimal upper bounds sharper  than in \cite{KnoYin2013,TaoVu2013} are required, which correspond to the best possible $\theta$ below.
A similar result in this direction was obtained in the context of eigenvectors in \cite{BenLop2022}.

\bel \label{l:newcomparison}
There exists $M_0>0$ such that the following holds.   
Let $H$ and  $R$ be two Wigner matrices satisfying $\E[H_{ij}^k] = \E[R_{ij}^k]$ for $k\in\unn{1}{3}$ and $\left| \E[H_{ij}^4 ]  -  \E[R_{ij}^4 ] \right| \le N^{-2} s$ where $s<\varphi^{-M_0}$.  Assume that there is $N_0(A,\kappa)$ and $\theta(A,\kappa)$ such that
\[
\E\big[ X_{p}(R)\big]  \le \theta^p (\log N)^{p}p^p
\]
for all $E \in [ -2 + \kappa, 2 -\kappa]$, $ p\leq A\log N$ and $N\geq N_0$. Then
there is $N_1(A,\kappa)$ such that for all $N\geq N_1$ we have 
\begin{equation}
  \label{eq-ElAl1}
\E\big[ X_{p}(H)\big]  \le {(1+\varphi^{-5})^p\theta^p (\log N)^{p}p^p
\leq 3\theta^p (\log N)^{p}p^p.}
\end{equation}
\eel

\begin{proof}
Fix any bijection
\beq
\phi \colon \{ ( i, j) : 1 \le i \le j \le N\} \rightarrow \unn{1}{\gamma_N},
\eeq
where $\gamma_N = N ( N + 1) /2$, and define the matrices $H^1, H^2, \dots , H^{\gamma_N}$ by 
\beq
H_{ij}^\gamma = 
\begin{cases}
H_{ij} & \text{if } \phi(i,j) \leq \gamma
\\
R_{ij} & \text{if } \phi(i,j) > \gamma
\end{cases}
\eeq
for $i \le j$. We also fix $z$ throughout the argument.

Fix some $\gamma \in \unn{1}{\gamma_N}$ and consider the indices $(i,j)$ such that $\phi(i,j) = \gamma$. For any $m \ge 1$, we may Taylor expand $X_m \left(H^\gamma\right)$ in the $(i,j)$ entry, write $\partial = \partial_{ij}$, and find

\begin{align}
\label{e:taylora11} X_m \left( H^\gamma \right) - X_m \left(  \Theta^{(i,j)}_0 H^\gamma \right) &= \partial X_m \left(  \Theta^{(i,j)}_0 H^\gamma \right) H_{ij} + \frac{1}{2!} \partial^2 X_m \left(  \Theta^{(i,j)}_0 H^\gamma \right) H_{ij}^2 
+ \frac{1}{3!}\partial^3 X_m \left(  \Theta^{(i,j)}_0 H^\gamma \right) H_{ij}^3 \\
& \label{e:taylora21} + \frac{1}{4!}\partial^4 X_m \left(  \Theta^{(i,j)}_0 H^\gamma \right) H_{ij}^4 + \frac{1}{5!}\partial^5 X_m \left(  \Theta^{(i,j)}_{w_1(\gamma)} H^\gamma \right) H_{ij}^5,
\end{align}
where $w_1(\gamma) \in [0,1]$ is a random variable depending on $H_{ij}$. Similarly, we may expand $X_m \left( H^{\gamma -1} \right)$ in the $(i,j)$ entry to obtain
a similar expansion with $w_2(\gamma) \in [0,1]$, a random variable depending on $R_{ij}$. Subtracting the expansion of $X_m \left(H^\gamma\right)$ from \eqref{e:taylora11} and \eqref{e:taylora21}, and taking expectation, we find
\begin{align}
 \label{e:fourthorder1} \E \left[ X_m \left( H^{\gamma } \right)  \right]- \E \left[ X_m \left(   H^{\gamma -1} \right) \right] &= \frac{1}{4!} \E \left[ \partial^4 X_m \left(  \Theta^{(i,j)}_0 H^\gamma \right) H_{ij}^4 \right] - \frac{1}{4!}\E \left[ \partial^4 X_m \left(  \Theta^{(i,j)}_0 H^\gamma \right) R_{ij}^4 \right] \\
  \label{e:fifthorder1} &+ \frac{1}{5!} \E \left[ \partial^5 X_m \left(  \Theta^{(i,j)}_{w_1(\gamma)} H^\gamma \right) H_{ij}^5\right]
  - \frac{1}{5!} \E \left[ \partial^5 X_m \left(  \Theta^{(i,j)}_{w_2(\gamma)} H^\gamma \right) R_{ij}^5 \right],\end{align}
where we used that $\Theta^{(i,j)}_0 H^\gamma$ is independent from $H_{ij}$ and $R_{ij}$, and that $\E [ h^k_{ij} ] = \E [ r^k_{ij} ]$ for $k \in \unn{1}{3}$.

We now proceed by induction, with the induction hypothesis at step $m \in \N$ being that 
\beq\label{e:inductionhypo1}
\E X_{n} \left(  \Theta^{(a,b)}_{w } H^\gamma \right)  \le {(1+\varphi^{-5})^n} \theta^n (\log N)^{n}n^n
\eeq
holds for all $0 \le n \le m$ and choices of  $w \in [ 0 ,1 ]$ and $(a,b) \in \unn{1}{N}^2$.

 The base case $m = 0$ is trivial. Assuming the induction hypothesis holds for $m -1$, we will show it holds for $m$. Using the independence of $H_{ij}$ and $R_{ij}$ from $ \Theta^{(i,j)}_0 H^\gamma$, we may rewrite the first two terms terms on the right side of \eqref{e:fourthorder1} as 
\beq\
\E \left[ \partial^4 X_m\left(  \Theta^{(i,j)}_0 H^\gamma \right) H_{ij}^4 \right] - \E \left[ \partial^4 X_m\left(  \Theta^{(i,j)}_0 H^\gamma \right) R_{ij}^4 \right] =  \E \left[ \partial^4 X_m\left(  \Theta^{(i,j)}_0 H^\gamma \right)  \right]  \E \left[  H_{ij}^4 - R_{ij}^4  \right].
\eeq
For the second factor, we recall that $\left|  \E \left[ H_{ij}^4\right] - \E \left[ R_{ij}^4 \right]  \right|  \le N^{-2}s =  N^{-2 } \varphi^{-M}$.  For the first, we {abbreviate $X_m=X_m\left(  \Theta^{(i,j)}_0 H^\gamma \right)$, 
write $X_m^{(\ell)}$ for the $\ell$th derivative of $X_m$ with respect to the
$(i,j)$ coordinate, and} compute 
\begin{align}\label{e:4thcompute1}
  \partial^4 X_m = \partial^4 \left( X^{2m} \right) &= 2m X^{2m-1} X^{(4)}  
  + 3 (2m) (2m - 1 ) X^{2m-2} (X^{(2)})^2\\ &+ 2m (2m-1) (2m-2) (2m-3) X^{2m-4} (X^{{(1)}})^4 \nonumber\\
  & + 4 (2m) (2m-1) X^{2m-2} X^{(1)} X^{(3)} + 6 (2m) (2m-1) (2m-2) X^{2m-3} (X^{{(1)}})^2 X^{(2)}.\nonumber
\end{align}
The terms with even powers of $X$ may be bounded using the induction hypothesis \eqref{e:inductionhypo1} for $ n \le m -1$ and Lemma~\ref{Xderiv}. The bound the odd powers, we additionally use \eqref{absX} to show 
\begin{equation}
\E | X^{2p -1 } | \le  \varphi\E X^{2p -2 }  + (N)^{2p-1} c^{-1} \exp ( - c (\log N)^{c \log \log N}),
\end{equation}
where we observe that the second term is $\oo(1)$ for $p\le A\log N$. Let $\mathcal B$ be the set where \eqref{e:TFX} holds.
We find\footnote{We note that the constants in the probability bound given by Lemma~\ref{Xderiv} do not depend on $\gamma$, since the matrices $H^\gamma$ verify \Cref{def:wig} simultaneously for the appropriate choice of constants. Therefore, the $C$ in \eqref{e:4analogytrivial1} is uniform in $\gamma$.}

\beq\label{e:4analogy1}
 \E \left| \one_{\mathcal B} \partial^4 X_m\left(  \Theta^{(i,j)}_0 H^\gamma \right)  \right|  \le   C  \varphi^{C } \theta^m (\log N)^{m}m ^m.
\eeq
Here $C>0$ is a constant that  is independent of $m$.
We also have by Lemma~\ref{Xderiv} that\beq\label{e:4analogytrivial1}
\left|\E \left[  \one_{\mathcal B^c}  \partial^4 X_m\left(  \Theta^{(i,j)}_0 H^\gamma \right)  \right]  \right| \le \tilde C N^{-100}.
\eeq
for some $\tilde C$ which does not depend on $m\leq A\log N$,  $i,j,N$.

 It follows from \eqref{e:4analogy1}, \eqref{e:4analogytrivial1}, and $m \le \varphi$ that, if $M_0$ is chosen large enough relative to $C$, then
\begin{equation}\label{e:4final2}
  \left|\frac{1}{4!} \E \left[ \partial^4 X_m\left(  \Theta^{(i,j)}_0 H^\gamma \right) H_{ij}^4 \right] -\frac{1}{4!} \E \left[ \partial^4 X_m\left(  \Theta^{(i,j)}_0 H^\gamma \right) R_{ij}^4 \right] \right| \le  {\varphi^{-10}}
  N^{-2} \theta^m (\log N)^{m}m ^m
\end{equation}
holds uniformly in $N\geq N_0$ and $m\leq A\log N$,  where $N_0$ does not depend on $m$.

Let $\mathcal D$ be the event where $\sup_{i,j} | R_{ij} | + | H_{ij} |\le  N^{-1/2 + \delta_1}$ holds. Since the variables $R_{ij}$ and $H_{ij}$ are subexponential, we have
\beq\label{e:Csubexp}
\P\left( \mathcal D^c \right) \le D_1 \exp\left({-d_1(\log N)^{d_1\log\log N}}\right), 
\eeq
for some constants $D_1(\delta_1), d_1(\delta_1)>0$.

For the terms in \eqref{e:fifthorder1}, we compute
\begin{align}
\left| \E \left[ \partial^5 X_m\left(  \Theta^{(i,j)}_{w_1(\gamma)} H^\gamma \right) H_{ij}^5\right] \right| 
&\le \left| \E  \left[  \one_{\mathcal D} \partial^5 X_m\left(  \Theta^{(i,j)}_{w_1(\gamma)} H^\gamma \right) H_{ij}^5\right] \right| + \left| \E \left[ \one_{\mathcal D^c} \partial^5 X_m\left(  \Theta^{(i,j)}_{w_1(\gamma)} H^\gamma \right) H_{ij}^5\right] \right|\\
&\le C N^{-5/2 + 5 \delta_1} \left( \E \left[ \left| \partial^5 X_m\left(  \Theta^{(i,j)}_{w_1(\gamma)} H^\gamma \right) \right| \right]  + 1\right),
\end{align}
where in the last line we used \eqref{e:Csubexp} and the constant $C$ comes from Lemma~\ref{Xderiv}.
Then repeating the previous argument for the fourth order term given in \eqref{e:4analogy1} and \eqref{e:4analogytrivial1}, we find that there exists $N_1(A)$ such that, for $\delta_1<1/100$,  $m\leq A\log N$ and $N\geq N_1$ we have
\begin{multline}\label{e:5final2a}
\Bigg|\frac{1}{5!}
\E \left[ \partial^5 X_m\left(  \Theta^{(i,j)}_{w_1(\gamma)} H^\gamma \right) H_{ij}^5\right]
  \Bigg|
\le  C m^5  N^{-5/2 + 5\delta_1} \varphi^{C} \theta^m (\log N)^{m}m ^m\\
\le C  N^{-2 - 1/4} \theta^m (\log N)^{m}m ^m\leq 
{N^{-2-1/8}} \theta^m (\log N)^{m}m ^m.
\end{multline}
Combining \eqref{e:4final2} and \eqref{e:5final2a} yields
\beq\label{e:gammainc}
\left| \E \left[ X_m\left( H^{\gamma } \right)  \right]- \E \left[ X_m\left(   H^{\gamma -1} \right) \right] \right|  \le {\varphi^{-5}}
N^{-2} \theta^m (\log N)^{m}m ^m,
\eeq
and summing \eqref{e:gammainc} over all $\gamma_N$ pairs $(i,j)$, we find
\beq\label{e:4momentconclude1}
\left| \E \left[ X_m\left( R \right)  \right]- \E \left[ X_m\left( H^\gamma \right) \right] \right| \le {\varphi^{-5}}
\theta^m (\log N)^{m}m ^m
\eeq
for any $\gamma$.  Together with our hypothesis on $\E\big[ X_{p}(R)\big]$ this gives
\beq
\E\left[ X_m\left( H^\gamma \right) \right] \le {(1+\varphi^{-5})^m}
\theta^m (\log N)^{m}m ^m.
\eeq
 This verifies the induction hypothesis \eqref{e:inductionhypo1} when $w = 1$.  

To address other values of $w$, we consider the following expansion:
\begin{align}
\label{e:taylorc1a}X_m\left( H^\gamma \right) - X_m\left(  \Theta^{(a,b)}_w H^\gamma \right)=&\sum_{\ell=1}^4\frac{1}{\ell!}\partial^\ell X_m\left(  \Theta^{(a,b)}_0 H^\gamma \right) H_{ij}^\ell(1-w^\ell)\\&
 + \frac{1}{5!}\partial^5 X_m\left(  \Theta^{(a,b)}_{\tau_1} H^\gamma \right) H_{ij}^5-
  \frac{1}{5!}\partial^5 X_m\left(  \Theta^{(a,b)}_{\tau_w} H^\gamma \right) H_{ij}^5\label{eqn:secondB}w^5.
\end{align}
Here $\tau_1,\tau_w \in [0, 1]$ are random variables. The same argument that gave the bound \eqref{e:4final2} shows that the expectation of the right side of \eqref{e:taylorc1a} is bounded with $\frac{1}{2}\theta^m (\log N)^{m}m ^m$. 
Note this bound holds because of the additional factors of $N^{-1/2}$ coming from moments of $H_{ij}$, which are enough even for $\ell=1,2,3$ as we don't sum over $N^2$ terms.
The expectation of  \eqref{eqn:secondB} is also bounded by 
${(1+\varphi^{-5})^m}
\theta^m (\log N)^{m}m ^m$ by the reasoning leading to \eqref{e:5final2a}. 

This proves
\beq
\sup_{w \in [0,1]} \sup_{a,b \in \unn{1}{n}} \E \left[ X_m\left(  \Theta^{(a,b)}_{w } H^\gamma \right)\right]  \le {(1+\varphi^{-5})}^m \theta^m (\log N)^{m}m ^m.
\eeq
and completes the induction.
{The second inequality in \eqref{eq-ElAl1} follows because $p\leq A\log N$.}
\end{proof}

\section{Maximum for Wigner Matrices}\label{s:wigmax}

This section proves Theorem \ref{thm:maxWigner} and Theorem \ref{thm:Wigner}  by combining the dynamics from Section 
\ref{sec:Relax} and the moment matching results from Section \ref{sec:MomMat}.  It also relies heavily on Section \ref{sec:maxgauss},  both its results (as the GOE serves as the base point of our comparison), and for methods used there to smooth the corresponding fields.

As we proceed by comparison,  we will need to specify the matrix ensembles related to the characteristic polynomials: We will write
$L_N^H$ for the quantity  (\ref{def:L_N}), when considering Wigner matrices as in Definition \ref{def:wig},  and we will write $L_N^{\GOE}$ for the same quantity when the eigenvalues of $H$ are replaced by those of $\GOE_N$.

We wil first prove Theorem \ref{thm:maxWigner}  for the real part, and then part (i) of Theorem \ref{thm:Wigner} on the deviations of $\lambda_i-\gamma_i$, which is equivalent to Theorem \ref{thm:maxWigner}  for the imaginary part (see Section \ref{sec:elementary}).
Indeed,  while the proof for $\Re L_N$ will go through a regularization similarly to the proof of Theorem \ref{thm:Beta}  in  Section \ref{sec:maxgauss}, we cannot directly follow the same path for $\Im L_N^H$: for the upper bound,  a priori smoothing $\Im L_N(E)$ into $\Im L_N(E+\frac{\ii}{N})$ as in 
\eqref{eqn:UBIm1bis}
is not possible because a local law allowing (\ref{eqn:microBound}) is not known in the case of Wigner matrices.

In all the following proofs, we will need an intermediate weakly Gaussian-divisible random matrix ensemble as in the following result, 
which is an immediate consequence of \cite[Lemma 16.2]{erdos2017dynamical}.

\begin{lemma}\label{l:momentmatch}
Let $H$ be a Wigner matrix. Then there exist constants $C,c>0$ such that the following holds for any $t \in (0,c)$. There exists a Wigner matrix $\wt H$
such that 
\begin{equation}\label{e:wt1}
\wt H(t) = \sqrt{1 - t} \wt H + \sqrt{t} W
\end{equation}
is a Wigner matrix satisfying 
\begin{equation}
\E \left[ \wt H_{ij}(t)^k \right] = \E \left[ H_{ij}^k \right] ,\qquad  \left| \E \left[ \wt H_{ij}(t)^4 \right]  - \E \left[ H^4_{ij} \right] \right| \le C t N^{-2}
\end{equation}
for $k \in \unn{1}{3}$.
Here $W \in \matn$ is a a Wigner matrix and each $W_{ij}$ is a mean zero  Gaussian random variable (independent from $\wt H$). 
\end{lemma}

\subsection{Upper bound for the real part in Theorem \ref{thm:maxWigner}.}\ 
We start with the deterministic bound (\ref{eqn:UBRe1}) in the particular case $\nu=\rho_{\rm sc}(x)\rd x$, so that 
\begin{equation}\label{eqn:UBReW1}
\sup_{E\in [A+\kappa,B-\kappa]}\Re L^H_N(E)\leq \sup_{E\in J}\Re L^H_N\left(E+\frac{\ii}{N}\right)+C_1
\end{equation}
where $J\subset[A+\kappa,B-\kappa]$ has cardinality at most $C_2N$, and $C_1,C_2$ are absolute constants.

Set $t=\varphi^{-K}$, where $K>0$ is a parameter. Let $\wt H(t)$ be the matrix \eqref{e:wt1} given by Lemma~\ref{l:momentmatch}.  For any $\e>0$, let $f_\e$ be a smooth function such that $0 \le f_\e(x) \le 1$ for $x\in \RR$, $f_\e(x) = 0$ for $ x \le \sqrt{2} + \eps/2$, and $f_\e(x) =1$ for $x \ge \sqrt{2}+\e$. 
By Theorem \ref{t:main},  Lemma \ref{alphaapprox} and  Lemma \ref{l:entropy}, if $K$ is chosen large enough, we have
\begin{equation}\label{eqn:upper3}
\E\left[ f_\e\left((\log N)^{-1}\sup_{z\in J+\frac{\ii}{N}} \Re L_N^{\wt H(t)}(z)\right) \right] -  \E \left[f_\e \left( (\log N)^{-1}\sup_{z\in J+\frac{\ii}{N}} \Re L_N^{H}(z) \right) \right]=\OO\left( \frac{1}{\log N} \right).
\end{equation}
By Proposition \ref{l:relax}, there exists a coupling of $\wt H(t)$ and $\GOE_N$ such that 
\begin{equation}\label{eqn:upper2}
\mathbb{P}\left(\sup_{z\in J+\frac{\ii}{N}} \Big|L_N^{\wt H(t)}(z)- L_N^{\GOE}(z)\Big|>(\log N)^{\frac{1}{2}+\e}\right)= \oo(1).
\end{equation}
Finally, from (\ref{eqn:UBRe2}) and (\ref{eqn:UBRe3}), and recalling that the GOE is a $\beta$-ensemble with $\beta=1$ and a quadratic potential (see, e.g., \cite[(4.4)]{erdos2017dynamical}), we have
\begin{equation}\label{eqn:upper4}
\mathbb{P}\left((\log N)^{-1}\sup_{z\in J+\frac{\ii}{N}} \Re L_N^{\GOE}(z)>\sqrt{2}+\frac{\e}{2}\right) = \oo(1).
\end{equation}
To conclude, we observe that from \eqref{eqn:UBReW1},  \eqref{eqn:upper3},  \eqref{eqn:upper2}, \eqref{eqn:upper4} we have
\[
\mathbb{P}\left((\log N)^{-1}\sup_{E\in [A+\kappa,B-\kappa]}\Re L^H_N(E)>\sqrt{2}+2\e\right) = \oo(1).
\]

\subsection{Lower bound for the real part in Theorem \ref{thm:maxWigner}.}\
Let $I=[-2+2\kappa,2-2\kappa]\cap N^{-1}\mathbb{Z}$.
We start with a direct analogue of (\ref{eqn:LBstep1}), with identical proof:
\begin{equation}\label{eqn:LBWstep1}
\mathbb{P}\Big(\sup_{z\in I+\ii\eta_0}\Re L^H_N(z)\leq \sup_{E\in[-2+\kappa,2-2\kappa]} \Re L^H_N(E)+1\Big)=1-\oo(1),
\end{equation}
where $\eta_0$ is defined in (\ref{eqn:etao1}).
We take $t=\varphi^{-K}$, where $K>0$ is a parameter,  and $\wt H(t)$ as in Lemma~\ref{l:momentmatch}. 
Let $f_\e$ be a smooth function such that $0 \le f_\e(x) \le 1$ for $x\in \RR$, $f_\e(x) = 0$ for $ x \ge 1 - \e/2$, and $f_\e(x) = 1$ for $x \le 1 - \e$. 
By Theorem \ref{t:main} and Lemma \ref{l:entropy}, if $K$ is chosen large enough, we have
\begin{equation}\label{eqn:lower3}
\E \left[f_\e \left((\log N)^{-1}\sup_{z\in I+\ii\eta_0} \Re L_N^{\wt H(t)}(z)\right)\right] -  \E \left[f_\e \left( (\log N)^{-1}\sup_{z\in I+\ii\eta_0} \Re L_N^{H}(z) \right)\right]=\OO\left( \frac{1}{\log N} \right).
\end{equation}
By Proposition~\ref{l:relax}, there exists a coupling of $\wt H(t)$ and $\GOE_N$ such that 
\begin{equation}\label{eqn:lower2}
\mathbb{P}\left(\sup_{z\in I+\ii\eta_0}\Big| L_N^{\tilde H(t)}(z)- L_N^{\GOE}(z)\Big|>(\log N)^{\frac{1}{2}+\e}\right) = \oo(1).
\end{equation}
Assuming that 
\begin{equation}\label{eqn:lower4}
\mathbb{P}\left((\log N)^{-1}\sup_{z\in I+\ii\eta_0} L_N^{\GOE}(z)<\sqrt{2}-\frac{\e}{2}\right) = \oo(1),
\end{equation}
the desired lower bound  follows from  \eqref{eqn:LBWstep1}, \eqref{eqn:lower3},  \eqref{eqn:lower2}, \eqref{eqn:lower4}:
\[
\mathbb{P}\left((\log N)^{-1}\sup_{E\in [2+\kappa,2-\kappa]}\Re L^H_N(E)<\sqrt{2}-2\e\right)= \oo(1).
\]
We now prove \eqref{eqn:lower4}.   Equation (\ref{eqn:LBstep2}) yields
\begin{equation}\label{eqn:LBstepW2}
\mathbb{P}\left((\log N)^{-1}\max_{z\in[-2+2\kappa,2-2\kappa]+\ii\eta_0}\Re  L_N^{\rm GOE}(z)\leq \sqrt{2}-\frac{\e}{4}\right)= \oo(1).
\end{equation}
Moreover,
\begin{multline}\label{eqn:LBstepW3}
\mathbb{P}\left(\exists z,w\in [-2+2\kappa,2-2\kappa]+\ii\eta_0: |z-w|<\frac{1}{N}, |L_N^{\rm GOE}(z)-L_N^{\rm GOE}(w)|> (\log N)^{9/10}\right)\\
\leq \mathbb{P}\left(\exists z\in [-2+2\kappa,2-2\kappa]+\ii\eta_0: |s(z)|> (\log N)^{7/10}\right)=\oo(1),
\end{multline}
where the last inequality follows from a union bound, Theorem \ref{th:local_law_bulk}, Markov's inequality, and a straightforward mesh argument (similar to the one before \eqref{e:int0}).
Equations (\ref{eqn:LBstepW2}) and (\ref{eqn:LBstepW3}) give \eqref{eqn:lower4} and conclude the proof.

\subsection{Extremal deviation with optimal constant}.  We now prove part (i) of Theorem \ref{thm:Wigner}.

As before,  set $t=\varphi^{-K}$, where $K>0$ is a large parameter and $\wt H(t)$ be the matrix \eqref{e:wt1} given by Lemma~\ref{l:momentmatch}.  Let $f_\e$ be a smooth function such that $0 \le f_\e(x) \le 1$ for $x\in \RR$, $f_\e(x) = 0$ for $|x| \in [\sqrt{2}-\frac{\e}{2},\sqrt{2} + \eps/2]$, and $f_\e(x) =1$ for $x \in[0,\sqrt{2}-\e]\cup [\sqrt{2}+\e,\infty)$. 
By Theorem \ref{t:main2}, Lemma \ref{l:ev} and Lemma \ref{l:entropy},  we have
\begin{equation}\label{eqn:upper3C}
\Big(\E^{\wt H(t)}-\E^{H}\Big) \left[ f_\e\left(\frac{\pi N}{\log N}\max_{k\in\llbracket \kappa N,(1-\kappa)N\rrbracket} \rho_{\rm sc}(\gamma_k)|\lambda_k-\gamma_k| \right) \right] =\OO\left( \frac{1}{\log N} \right).
\end{equation}
By Proposition \ref{prop:homog} and \Cref{lem:maxSmoothField}, there exists a coupling of $\wt H(t)$ (with eigenvalues $\bla$) and $\GOE_N$ (with eigenvalues $\bmu$) such that 
\begin{equation}\label{eqn:upper2C}
\mathbb{P}\left(\max_{k\in\llbracket\kappa N,(1-\kappa)N\rrbracket}|\lambda_k-\mu_k|>\frac{(\log N)^{\frac{1}{2}+\e}}{N}\right)= \oo(1).
\end{equation}
Finally, from Corollary \ref{cor:Beta} we have
\begin{equation}\label{eqn:upper4C}
\mathbb{P}^{\rm GOE}\left(\frac{\pi N}{\log N}\max_{k\in\llbracket \kappa N,(1-\kappa)N\rrbracket} \rho_{\rm sc}(\gamma_k)|\lambda_k-\gamma_k|\not\in\left[\sqrt{2}-\frac{\e}{2},\sqrt{2}+\frac{\e}{2}\right]\right)= \oo(1).
\end{equation}
From \eqref{eqn:upper3C},  \eqref{eqn:upper2C},  and \eqref{eqn:upper4C} we have
\[
\mathbb{P}^H\left(\frac{\pi N}{\log N}\max_{k\in\llbracket \kappa N,(1-\kappa)N\rrbracket} \rho_{\rm sc}(\gamma_k)|\lambda_k-\gamma_k|\not\in[\sqrt{2}-\e,\sqrt{2}+\e]\right)= \oo(1).
\]

\subsection{Rigidity with optimal order}.\label{subsec:rigidity}
 We finally prove part (ii) of Theorem \ref{thm:Wigner}, building on the key relaxation and moment matching results, namely lemmas \ref{l:newcomparisonGD} and \ref{l:newcomparison}. \\

\noindent{\it First step: smoothed indicator for Gaussian divisible ensemble.} We specify $f=f_E$ from (\ref{f}) to be of type  $f=\int\eta_1^{-1}h((x-E)/\eta_1)\mathds{1}_{[x,\infty)}\rd x$ where $h$ is positive,  smooth, compactly supported on $[0,1]$ and $\int h=1$.  Then  $f$ satisfies the bounds (\ref{fbounds}).
Moreover, for $t$ and $H$ as in  Lemma \ref{l:newcomparisonGD},  we have 
\begin{equation}\E[|\overline{{\rm Tr}}f(H_t)|^{2p}]\leq
\int\eta_1^{-1}h((x-E)/\eta_1)\E[|\overline{{\rm Tr}}\mathds{1}_{[x,\infty)}(H_t)|^{2p}]\rd x\leq \left(\frac{2}{e\pi^2}+\e\right)^p p^{p} (\log N)^{p},\label{eqn:smoothed1}
\end{equation}
where we have used convexity of $x\mapsto x^{2p}$ in the first inequality and the result of Lemma \ref{l:newcomparisonGD} in the second inequality.\\

\noindent{\it Second step: smoothed spectrum for Gaussian divisible ensemble.}\ We now prove that the actual spectrum $\bla$ can be replaced by the smoothed one $\tilde\bla$ in
(\ref{eqn:smoothed1}).

Let $A=\{\left|\sum_{i=1}^N f(\tilde \lambda_i) - \sum_{i=1}^N f( \lambda_i) \right| > \sqrt{\log N} \}$, and remember the notation (\ref{eqn:Xp}).
From Cauchy Schwarz and Lemma \ref{lemma:regLin} we have
\[
\E\left[X_p(H_t)\mathds{1}_{A}\right]\ll N^{2\log N} \exp( - c \phi/2)\to 0.
\]
Moreover,  by definition of $A$,  for any $\e>0$ and $\lambda>0$ to be chosen we bound
\begin{multline*}
\E\left[X_p(H_t)\mathds{1}_{A^c}\right]\leq \E\left[(|\overline{{\rm Tr}}f(H_t)|+\sqrt{\log N})^{2p}\right]\\\leq 
 \E\left[(|\overline{{\rm Tr}}f(H_t)|+\frac{|\overline{{\rm Tr}}f(H_t)|}{\lambda})^{2p}\mathds{1}_{|\overline{{\rm Tr}}f(H_t)|>\lambda\sqrt{\log N}}\right]
+
 \E\left[(\lambda\sqrt{\log N}+\sqrt{\log N})^{2p}\mathds{1}_{|\overline{{\rm Tr}}f(H_t)|\leq \lambda\sqrt{\log N}]}\right]\\
\leq(1+\lambda^{-1})^{2p} \left(\frac{2}{e\pi^2}+\e\right)^p p^{p} (\log N)^{p}+(\lambda+1)^{2p}(\log N)^p\leq  \left(\frac{2}{e\pi^2}+2\e\right)^p p^{p} (\log N)^{p}
\end{multline*}
for $N\geq N_0(\e,\kappa,A)$. We have used the definition of $A$ in the first inequality,  (\ref{eqn:smoothed1}) in the third one and the choice $\lambda=p^{1/10}$ in the third one.

We have therefore proved that for any $\e,\kappa,A>0$ there is a  $N_1(\e,\kappa,A)$ such that
\begin{equation}\label{eqn:interm11}
\E\left[X_p(H_t)\right]\leq  \left(\frac{2}{e\pi^2}+\e\right)^p p^{p} (\log N)^{p}
\end{equation}
for any $p\leq A\log N$,  $E\in[-2+\kappa,2-\kappa]$, and $t>\exp(-(\log N)^{1/10})$.\\

\noindent{\it Third step: moment matching.}\
Let $H$ be the Wigner matrix of interest in  part (ii) of Theorem \ref{thm:Wigner}.  Considering the dynamics  \eqref{OU},
the moment matching lemma \cite[Lemma 16.2]{erdos2017dynamical} gives existence of
a Wigner matrix $H_0$ such that the matrix $R= H_{t}$ satisfies $\E[H_{ij}^k] = \E[R_{ij}^k]$ for $k\in\unn{1}{3}$ and $\left| \E[H_{ij}^4 ]  -  \E[R_{ij}^4 ] \right| \le C N^{-2} t$, where $C>0$ depends only on the constants from Definition~\ref{def:wig}. 

We choose $t=\varphi^{-M}$ for a fixed $M>M_0$, with $M_0$ the constant from Lemma \ref{l:newcomparison}.  Note that $t>\exp(-(\log N)^{1/10})$ so that (\ref{eqn:interm11}) holds.  We can therefore
apply Lemma \ref{l:newcomparison} with $R=H_t$ and obtain that there is a $N_2(\e,\kappa,A)$ such that
\begin{equation}\label{eqn:interm12old}
\E\left[X_p(H)\right]\leq  \left(\frac{2}{e\pi^2}+\e\right)^p p^{p} (\log N)^{p}
\end{equation}
for any $p\leq A\log N$,  $E\in[-2+\kappa,2-\kappa]$ and $N\geq N_2$.
As in the previous step but in the reverse direction now,  with  Lemma \ref{lemma:regLin} we obtain that the same property holds for the actual eigenvalues: for any choice of the parameters,
there is a $N_3(\e,\kappa,A)$ such that
\begin{equation}\label{eqn:interm12}
\E\left[|\overline{{\rm Tr}}f(H)|^{2p}\right]\leq  \left(\frac{2}{e\pi^2}+\e\right)^p p^{p} (\log N)^{p}
\end{equation}
for any $p\leq A\log N$,  $E\in[-2+\kappa,2-\kappa]$ and $N\geq N_3$. 

To conclude,  note that for any fixed $\e$ there is a $N_4(\e,\kappa)$ such that for any $u>1$, the inequality
$\lambda_k-\gamma_k>u\cdot\frac{\sqrt{2}}{\pi\rho_{\rm sc}(\gamma_k)}\cdot\frac{\sqrt{\log N}}{N}$ implies
that for $E=\gamma_k+u\cdot(\frac{\sqrt{2}}{\pi\rho_{\rm sc}(\gamma_k)}-1)\cdot\frac{\sqrt{\log N}}{N}$ we have
$\overline{{\rm Tr}}f_E(H)>(1-\e) \frac{u\sqrt{2}}{\pi}\sqrt{\log N}$.
With (\ref{eqn:interm12}) this gives
\[
\mathbb{P}\left(\lambda_k-\gamma_k>u\cdot\frac{\sqrt{2}}{\pi\rho_{\rm sc}(\gamma_k)}\cdot\frac{\sqrt{\log N}}{N}\right)\leq 
 \left(\frac{2}{e\pi^2}+\e\right)^p p^{p} (\log N)^{p}\cdot ((1-\e) \frac{u\sqrt{2}}{\pi}\sqrt{\log N})^{-2p}\leq (\frac{p}{e}+10\e)^{p}u^{-2p}.
\]
Optimization in $p$ then concludes the proof.

\subsection{Gaussian divisible ensemble: universality up to tightness.}\  Theorem \ref{thm:maxWignerDivisible} follows immediately from Proposition \ref{l:relax2}.

\appendix

\section{High Moments of Linear Statistics for Wigner Matrices}
\label{sec:moments}

The main goal of this appendix is to prove \Cref{p:moments}, which provides estimates on large moments (growing in $N$) of the semicircle law. This proposition is used in the proof of \Cref{l:relax}, via its  \Cref{lem:maxSmoothField}, and \Cref{l:newcomparisonGD}. While a weaker result would suffice for the proof of \Cref{l:relax}, for example a bound on a fixed but large moment, it is indeed necessary to control growing moments for the application in \Cref{l:newcomparisonGD}.

\begin{proposition}\label{p:moments}
Let $H$ be a real symmetric Wigner matrix. Fix $K, \kappa,A >0$. 
For every $E \in [-2 + \kappa, 2 - \kappa]$, $\eta \in \left[ \phi^{-K}, \phi^K \right]$, and $p\in \N$ with $p \le A\log N$, there exists a  constant $C(K, \kappa,A) >0$ such that \beq\label{eqn:momSt}
\E   \left[ \big| m_N(z) - \msc (z) \big|^{p} \right]
\le \left(\frac{C}{N\eta} \right)^p p^{3p/4}.
\eeq
\end{proposition}

\begin{remark}
A natural approach  to bounds such as (\ref{eqn:momSt}) relies on concentration  for random matrices.
However,  even on close-to-macroscopic scales and for bounded or log-concave matrix entries,   this method would not give accurate enough bounds for
\Cref{l:relax} and \Cref{l:newcomparisonGD}.
Indeed the concentration from \cite{GuiZei2000} yields
$
\mathbb{P}\left(N|m_N(z)-\E[m_{N}(z)]|>\lambda\right)\leq c^{-1}e^{-c\lambda^2/\eta^4}$ ($\eta={\rm Im} z$),
so that 
\[
\E   \left[ \big| m_N(z) -\E[m_{N}(z)] \big|^{p} \right]
\le \left(\frac{C}{N\eta^2} \right)^p p^{p/2}.
\]
Integration of  $|m_N-\E[m_{N}(z)]|\lesssim C/(N\eta^2)$ gives $|L_N|\lesssim \varphi^{C}$ for $\eta\asymp \varphi^{-C}$,  an error bigger than the order of magnitude
 $\max_{|E|<2-\e}L_N\asymp \log N$ that we aim at.
\end{remark}

\begin{remark}
For our application to \Cref{l:newcomparisonGD}, it is critical that the exponent $3/4$ in (\ref{eqn:momSt}) is smaller than $1$, i.e.  one could not afford the exponential tail error  $\left(\frac{C}{N\eta} \right)^p p^{p}$.
\end{remark}
We defer the proof of \Cref{p:moments} to \Cref{s:momentproof}, after establishing various preliminary results in the following subsections. Throughout, we use the notations defined in \Cref{s:relaxprelim}.

\subsection{High-probability bound on the log-characteristic polynomial.}\ 
We begin with an application of \Cref{p:moments} that provides estimates on the maximum of the log-characteristic polynomial smoothed at almost-macroscopic scale.

\begin{corollary}\label{lem:maxSmoothField} Let  $\bla$ be the spectrum of a $N\times N$ Wigner matrix and remember the notation (\ref{eqn:LNx}).  Let $\kappa>0$ and  denote $z=E+\ii\eta$. 

\begin{enumerate}[(i)]
\item
Fix $C_1>10$.
There exists a constant $c(C_1, \kappa)>0$ such that for every $N\geq 1$ and $u\in[1,c\, (\log N)^{3/4}]$, we have
\beq\label{sufficeclm}
\mathbb{P}\left( \max_{|E|<2-\kappa,\eta \in [\phi^{-C_1},1]} \left|
L_N^{\bla}(z)\right|>u\right)\leq c^{-1} e^{c^{-1}(\log\log N)^2- c(\frac{u}{\log_2 N})^{4/3}}.
\eeq
Moreover, for any $u\leq C_1\varphi$, denoting $q=C_1\log n$ we have
\beq\label{eqn:pointBD}
\mathbb{P}\left( \left|
L_N^{\bla}(z)\right|>u\right)\leq c^{-1} e^{- c(\frac{u}{\log_2 N})^{4/3}}\mathds{1}_{u<c^{-1} (\log N)^{3/4}}+c^{-q}q^{\frac{3}{2}q}u^{-2q}\mathds{1}_{u>c^{-1} (\log N)^{3/4}}
\eeq
\item Let $0<\eta_1<\eta_2<0$ be fixed and $\mathscr{C}$ be a fixed smooth path in $[-2+\kappa,2-\kappa]\times [\eta_1,\eta_2]$, of finite length.  Then for every $\e>0$ there exists $M>0$ such that for any $N\geq 1$,
\begin{equation}
\mathbb{P}\left( \max_{\mathscr{C}} \left|
L_N^{\bla}(z)\right|>M\right)\leq 1-\e.
\end{equation}
\end{enumerate}
\end{corollary}
\begin{proof} We start with (i).  
By the local semicircle law \eqref{e:sclaw} we have
\begin{equation}\label{e:fieldregularity}
\partial_z \left( \sum_{k=1}^N\log(z - \la_k )
\right)=  \sum_{k=1}^N \frac{1}{z - \la_k}=  -  N m_N(z) =  
\OO(\varphi^{C_1+4}),
\end{equation}
with probability $1 - \OO(\exp( - \varphi^c ))$,  uniformly in $|E|<2-\kappa$ and $\eta\in[\varphi^{-C_1},1]$.

Let $\delta  = \varphi^{- C_1 -4}$, and let $\mathcal M  = \mathcal M_\delta = \{ m_i \}_{i=1}^{ \lfloor 4 \varphi^4\delta^{-1} \rfloor}$ be a collection of points in $[ -2 + \kappa, 2 - \kappa]\times [\varphi^{-C_1},1]$
such that any $z\in[ -2 + \kappa, 2 - \kappa]\times [\varphi^{-C_1},1]$ there exists $m_i$ such that $|z-m_i|\leq 10\delta$.
 By \eqref{e:fieldregularity} and the definition of $\mathcal M$, we have
\begin{equation}
\max_{E \in[ -2 + \kappa, 2 - \kappa] }|L_N^{\bla}(z)|=  \max_{E \in \mathcal M }|L_N^{\bla}(z)|+ \OO(\delta \varphi^{C_1+4}).
\end{equation}
Since $\delta \varphi^{C_1+4} = \OO(1)$, to establish \eqref{sufficeclm}, it therefore suffices to show that,  for any $z\in\mathcal{M}$,
\begin{equation}\label{e:a2conclusion}
\P\left(\left| L_N^{\bla}(z) \right| > u \right) \le {c}^{-1}e^{- c(\frac{u}{\log_3 N})^{4/3}},
\end{equation}
and then use a union bound on $|\mathcal{M}|=\OO(\varphi^{C_1+8})$ points.
To prove (\ref{e:a2conclusion}), for some parameter $\eta_1>\eta$ we first write
\begin{equation*}
L_N^{\bla}(z) =   N\int_{s = \eta }^{\eta_1}  (m_N(E + \iu s) - \msc(E + \iu s)) \rd s
+ L_N^{\bla}(E+\ii \eta_1).
\end{equation*}
To bound the above terms, we compute using the rigidity bound \eqref{e:rigidity} and Taylor expansion of $\log$
that 
\begin{multline}
\left|\sum_{i =1}^N\log(E+\ii \eta_1 -\lambda_i)-\sum_{i =1}^N\log(E+\ii \eta_1 -\gamma_i)\right| =\left|\sum_{i =1}^N\log\left(1+\frac{\lambda_i-\gamma_i}{E+\ii \eta_1 -\gamma_i}\right)\right|
\leq \sum_{i =1}^N C |\lambda_i-\gamma_i|/\eta_1\\
\leq C \varphi^{1-C_2}\sum_{i=1}^N \min(i,N+1-i)^{-1/3}N^{-2/3}\leq C  \varphi^{1-C_2},\label{eqn:far}
\end{multline}
where we choose $\eta_1=\varphi^{C_2}$ for some $C_2>0$.
Similar reasoning shows that 
\begin{equation*}
\left|N \int \log(E  + \iu \eta_1 - u) \sc(u) \rd u-\sum_{i =1}^N\log(E+\ii \eta_1 -\gamma_i)\right| \le C \phi^{1 - C_2},
\end{equation*}
where $C=C(C_1, \kappa)>0$ may change from line to line below.
The previous three equations give (for fixed, large enough $C_2$), for arbitrary $D>0$ and $N\geq N_0(D)$,
\[
\mathbb{P}\left(|L^{\bla}_N(z)- \int_{\eta}^{\eta_1} N\big(m_N(E + \iu s) - \msc(E + \iu s)\big) \, \rd s|>1\right)\leq N^{-D}.
\]
We denote $\Delta_N(z)=N(m_N(z) - \msc(z))$. By Markov's inequality, for any $p\geq 1$,
\begin{multline}\label{eqn:close}
\mathbb{P}\left(\left| \int_{\eta}^{\eta_1} \Delta_N(E+\ii s) \rd s \right|>u\right)
\le u^{-2p} \E\left[ \left|\int_{\eta}^{\eta_1} \Delta_N(E+\ii s)  \rd s  \right|^{2p}\right]\\
\leq
u^{-2p}
\left|\int_{\eta}^{\eta_1} \left(  \E 
\left| \Delta_N(E+\ii s)\right|^{2p}\right)^{\frac{1}{2p}} \rd s
\right|^{2p}
\leq u^{-2p}C^{2p} p^{\frac{3p}{2}} \left( \int_{\eta}^{\eta_1} \frac{\rd\eta}{\eta} \right)^{2p},
\end{multline}
where the second inequality is obtained by expansion and H\"{o}lder's inequality,  and the third inequality relies on Proposition \ref{p:moments} for $p=\OO(\log N)$.
We now recall that $\log \varphi = \mathrm{O}\big( (\log \log N)^2 \big)$,  so that the above probability is also bounded with
\[
u^{-2p}C^{2p} p^{\frac{3p}{2}} \left( C\log\log N\right)^{2p}
\]
The choice $p=e(u/\log_2N)^{4/3}$  proves \eqref{e:a2conclusion} and concludes the proof of  (\ref{sufficeclm}). The proof of (\ref{eqn:pointBD}) is the same, with no need of discretization.

The proof of (ii) is simpler as it does not need any discretization and only requires finite moment estimates.  Indeed it follows form the following the following two facts. First,  $(L_N(z_0))_N$ is tight, where 
$z_0\in[-2+\kappa,2-\kappa]\times [\eta_1,\eta_2]$ is fixed.  This follows from convergence in distribution of this series (see e.g.\ \cite{LytPas2009}). Then
$\max_{\mathscr{C}}|L_N(z)-L_N(z_0)|$ is also tight because it dominated with $\int_{\mathscr{C}}|\Delta_N(w)|\cdot|\rd w|$, which is tight by Proposition \ref{p:moments} with $p=2$.
\end{proof}

\subsection{Preliminaries.}\
We first list some preliminary results necessary for the proof of  
\Cref{p:moments}. 
We begin with a power counting lemma for resolvent entries. Given a parameter $A>0$, we set
\begin{equation}\label{DA}
\mathcal D_{A} = \left\{ z =E + \iu \eta \in \mathbb{H} :  |E| < 2 - A^{-1}, \eta \ge \phi^{-A}   \right\}.
\end{equation}
Throughout, we let $H$ be a Wigner matrix and let $G$ denote its resolvent. 
\bel\label{l:pc}
Fix $A>0$. There exists $C(A)>0$ and $N_0(A)> 0$ such that the following holds for all $z \in \mathcal D_{A}$  and $N \ge N_0$.
For any $i,k \in \unn{1}{N}$ and random variable $F(z)$ such that $|F| \le (A  \eta^{-1} )^{A \log N}$ almost surely,
\begin{equation}\label{smallpowercount}
\E \left[ \left| F \right| \cdot \frac{1}{N} \sum_{j=1}^N | G_{ij} | \right] \le \left( \frac{ C }{N\eta} \right)^{1/2} \E  \big[ |F | \big ] 
+ N^{- A \log N}
,\quad 
\E \left[  \left| F \right| \cdot \frac{1}{N} \sum_{j=1}^N | G_{ij} | |G_{j k}| \right] \le \frac{ C }{N\eta}  \E  \big[ |F|  \big] + N^{- A \log N},
\end{equation}
where we set $G=G(z)$, $F=F(z)$, and $z = E + \ii \eta$. 
More generally, for any $n \le A \log N$,
\begin{equation}\label{powercount}
\E \left[  \left| F \right| \cdot  \frac{1}{N^n} \sum_{j_1,\dots,j_n = 1}^N \left| G_{ij_1} \cdots G_{j_n k} \right| \right] \le \left( \frac{ C }{N\eta} \right)^{n/2} \E  \left|F \right| 
+ N^{- A \log N}
.
\end{equation}
\eel
\begin{proof}
We give only the proof of \eqref{powercount}, as the proofs of the remaining statements are similar.
Suppose that $n$ is odd. We apply the elementary inequality 
\begin{equation}
\sum_{j=1}^N | G_{aj} G_{jb} | \le \frac{1}{2} \sum_{j=1}^N | G_{aj}|^2 + |G_{jb} |^2
\end{equation}
for $j=j_1,j_3,j_5\dots$ and the Ward identity \eqref{e:ward} to show that the left side of \eqref{powercount} is bounded by 
\begin{equation}
\E\left[\left| F \right| \cdot 
\frac{1}{N^n} \sum_{j_2,j_4,j_6,\dots} \frac{1}{2\eta}(\im G_{ii}+\im G_{j_2j_2})
\cdots \frac{1}{2\eta}(\im G_{j_{n-1} j_{n-1}}+\im G_{kk})\right],
\end{equation}
where there are $(n+1)/2$ factors in the sum. 

Let $\mathcal A$ be the high-probability set from \eqref{e:gentries}.
For $\eta>N^{-1/2}$, we have $\im G_{jj}<C$ on $\mathcal A$, which implies
\begin{align}\label{onA}
&\E\left[\mathds{1}_{\mathcal A}\left| F \right| \cdot  \frac{1}{N^n} \sum_{j_2,j_4,j_6,\dots} \frac{1}{2\eta}(\im G_{kk}+\im G_{j_2j_2})\cdot \frac{1}{2\eta}(\im G_{j_2j_2}+\im G_{j_4j_4})\dots \right]\\
&\le \frac{N^{(n-1)/2}}{N^n}\left( \frac{ C }{\eta} \right)^{(n+1)/2}
\E  \big[|F |\big]  
 = \left(\frac{C}{N\eta }\right)^{(n+1)/2}\E  \big[|F |\big] .
\end{align}
using $C (N\eta)^{-1} \le 1$.  On $\mathcal{A}^c$ we use the trivial bound $\im G_{ii}<\eta^{-1}$ and the strong probability estimate on $\P(\mathcal{A}^c)$ from \eqref{e:gentries}.
This gives
\begin{align}\label{offA}
&\E\left[\left| F \right| \cdot  \mathds{1}_{A^c}\frac{1}{N^n} \sum_{j_2,j_4,j_6,\dots} \frac{1}{2\eta}(\im G_{kk}+\im G_{j_2j_2})\cdot \frac{1}{2\eta}(\im G_{j_2j_2}+\im G_{j_4j_4})\dots \right]\\
& \le\frac{N^{(n+1)/2}}{N^n\eta^{n+1}}  \cdot (A  \eta^{-1} )^{A \log N}\cdot c^{-1} \exp ( - c (\log N)^{C_0 \log \log N})
\le N^{- A \log N},
\end{align}
by the assumptions on $\eta$ and $n$ (recall $z\in \mathcal D_A$), for sufficiently large $N$. The claim follows by combining \eqref{onA} and \eqref{offA}.

 The proof for even $n$ is similar, using  $|G_{j_n k}| \le C$ on $\mathcal A$ to bound the left side of \eqref{powercount} by 
\begin{equation}
\E\left[ \left| F \right| \cdot \frac{1}{N^n} \sum_{j_2,j_4,j_6,\dots} \frac{1}{2\eta}(\im G_{ii}+\im G_{j_2j_2})
\cdots \frac{1}{2\eta}(\im G_{j_{n-2} j_{n-2}}+\im G_{j_n j_n}) |G_{j_n k}| \right].
\end{equation}
\end{proof}

Given a random variable $X$, we let $\kappa^{(j)}(X)$ denote the $j$-th cumulant of $X$. The following lemma is  \cite[Lemma 3.2]{lee2018local}.
Known as a cumulant expansion, it provides an extension of the well-known Gaussian integration by parts formula to non-Gaussian random variables. (Observe that \eqref{cumulantexpansion} reduces to a single term when $Y$ is Gaussian, as all higher cumulants of $Y$ vanish in this case.)
\bel\label{cumulant}
Fix $\ell \in \N$, $Q>0$, and  $F \in C^{\ell+1}(\R; \C^+)$. 
Let $Y$ be a random variable such that $\E[Y]=0$ with finite moments to order $\ell+2$. Then
\begin{equation}\label{cumulantexpansion}
\E \big[ YF(Y) \big] = \sum_{r=1}^\ell  \frac{\kappa^{(r+1)}(Y)}{r!}
 \E \left[ F^{(r)}(Y) \right] 
+
 \E \Big[ \Omega_\ell \big( YF(Y) \big) \Big],
\end{equation}
where $\Omega_\ell(YF(Y))$ is an error term that satisfies
\begin{equation}
\left|  \E\Big[ \Omega_\ell \big(Y F(Y) \big) \Big]   \right| \le 
C_\ell \E \left[  | Y |^{\ell+2} \right]   \sup_{|t| \le Q } \left|  F^{(\ell+1)} (t) \right|
+ 
C_\ell \E\left[ | Y|^{\ell+2} \one(|Y| > Q) \right] 
\sup_{t \in \R} \left|  F^{(\ell+1)} (t) \right|.
\end{equation}
The constant $C_\ell$ satisfies $C_\ell \le (C \ell)^\ell / \ell !$ for some $C>0$ that does not depend on $Q$, $F$, or $\ell$. 
\eel
\bel\label{l:moments}
Let $H$ be a Wigner matrix in the sense of Definition~\ref{def:wig}.
Then there exists a constant $C >0$ such that
\begin{equation}\label{cgrowth}
\big| \kappa^{(k)}(H_{ij}) \big| \le \frac{ (C k)^{Ck} } { N^{k/2}}, \quad k \ge 3
\end{equation}
for all $i,j \in \unn{1}{N}$.
\eel
\begin{proof}
The claim follows by expressing the cumulants in terms of moments, and using the moment bound 
\begin{equation}\label{momentgrowth}
\E \big[ | H_{ij}|^k  \big] \le \frac{ (C k)^{Ck} } { N^{k/2}}, \quad k \ge 3.
\end{equation}
The bound \eqref{momentgrowth} follows from the subexponential decay hypothesis
\eqref{eqn:tail}.
\end{proof}

\subsection{Main Calculation.}\
Let $H=H(N)$ be a $N\times N$ Wigner matrix. We introduce the shorthand
$m = m_N(z)$, where $m_N$ denotes the Stieltjes transform of $H$.
The proof of \Cref{p:moments} proceeds by bounding the moments 
\begin{equation}\label{mmoments}
\E\big[ | 1 + zm + m^2 |^{2D} \big], \qquad D \in \mathbb{N}.
\end{equation}
To explain this strategy, observe that by the explicit form of $\msc$ in \eqref{e:msc}, we have  
$1 + z\msc + \msc^2 =0$.
Then, since  $m \approx \msc$ for large $N$ (by \eqref{e:gentries}), we have $1 + zm + m^2 \approx 0$. We will see later that the reverse implication also holds, so that sufficiently strong bounds on the moments in \eqref{mmoments} imply the bounds on the moments of $|\msc - m|$ claimed in \Cref{p:moments}. Therefore, we focus for now on \eqref{mmoments}. 

Let $G = (H -z)^{-1}$ be the resolvent of $H$. To bound the moments \eqref{mmoments}, we use the definition of $m$ in \eqref{e:st} to write 
\begin{equation}
\E\big[ | 1 + zm + m^2 |^{2D} \big]
= \E \left[ \left(\frac{1}{N} \sum_{i=1}^N (1 + z G_{ii}) + m^2  \right) (1+zm + m^2)^{D-1} ( \overline{1+zm+m^2})^D  \right],
\end{equation}
which holds for any $D \in \mathbb{N}$.
Set 
\begin{equation}
  \label{eq-of2}
   P=P(m)=1+zm + m^2.
\end{equation}
By the definition of the resolvent, we have $1 + z G_{ii} = (HG)_{ii}$, which implies
\begin{equation}\label{twoterms}
\E \big[ | P |^{2D} \big]
= \E \left[ \left(\frac{1}{N} \sum_{i,k=1}^N H_{ik} G_{ki} \right)   P^{D-1} \overline{P}^D  \right]
+
\E \left[  m^2   P^{D-1} \overline{P}^D  \right].
\end{equation}
Let $\ell \in N$ be a parameter, which will be fixed later.
By a cumulant expansion using Lemma~\ref{cumulant} (setting $Y=H_{ik}$), we find that 
\begin{equation}\label{exp1}
\E \left[ \left(\frac{1}{N} \sum_{i,k} H_{ik} G_{ki} \right)   P^{D-1} \overline{P}^D  \right] = 
\frac{1}{N} \sum_{r=1}^\ell \frac{\kappa_{r+1} }{r!} 
\E \left[\sum_{i,k} (1 + \delta_{ik}\Delta_{r+1}) \partial^r_{ik} \left(G_{ki} P^{D-1} \overline{P}^D  \right)  \right] + \E [\Omega_{\ell}].
\end{equation}
Here $\Omega_\ell=\Omega_\ell(z)$ is an error term that we will examine later, $\partial_{ik}$ is the partial derivative in the matrix entry $H_{ik}$, $\kappa_{r+1}$ is the $(r+1)$-th cumulant of $H_{ij}$ for $i\neq j$, and $\Delta_r$ is equal to $(\kappa^{(r)}(H_{11}) - \kappa_r )/\kappa_r$.

We introduce the notation $I_{r,s}$ to denote the component of the $r$-th term of \eqref{exp1} where $r-s$ derivatives fall on $G_{ik}$ and $s$ derivatives fall on $P^{D-1} \overline{P}^D$:
\beq\label{Irs}
I_{r,s} =  \frac{\kappa_{r+1} }{N}  \E \left[\sum_{i,k}(1 + \delta_{ik}\Delta_{r+1}) \left( \partial^{r-s}_{ik} G_{ki} \right) \partial^s_{ik}\left( P^{D-1} \overline{P}^D  \right)  \right] .
\eeq
Then \eqref{exp1} becomes 
\beq
\sum_{r=1}^{\ell}
\sum_{s=0}^r w_{r,s} \E [I_{r,s}]
+ \E [\Omega_{\ell}], \qquad w_{r,s} = \frac{1}{r!}\binom{r}{s}.
\eeq

We begin by bounding $\E [\Omega_{\ell}]$, then proceed to the $I_{r,s}$ terms.

\subsubsection{Truncation.}\
Let $E^{[ik]}$ denote the $N\times N$ matrix with entries 
\beq
(E^{[ik]})_{a,b} = \delta_{ia}\delta_{kb} + \delta_{ib}\delta_{ka}\text{ if }i \neq k, \quad (E^{[ik]})_{a,b} = \delta_{ia}\delta_{ib} \text{ if } i = k.
\eeq
For $i,k \in \unn{1}{N}$, we define $H^{(ik)} =   H - H_{ik} E^{[ik]}$, 
which sets the $(i,k)$ and $(k,i)$ entries of $H$ to zero.

\bel\label{l:Fperturb}
Suppose $i,k\in \unn{1}{N}$, $D \le \log N$, $A>0$, and $z \in \mathcal D_A$. Define the function $F\colon \matn \rightarrow \C$ by 
\beq
F_{ki}(M) = R_{ki}(M) P^{D-1} \overline{P}^D,
\eeq
where here $R(M)$  denotes the resolvent of $M$. 
Choose $\ell \in \N$ such that $\ell < A \log N$. 
Then there exist constants $C, C_1(A) >0$ such that 
\beq\label{Fperturb}
\E \left[  \sup_{x \in \R, |x| \le N^{-1/4}} \left| \partial^\ell_{ik} F_{ki} \left(
H^{(ik)} + x E^{[ik]}
 \right)  \right| \right] \le
 (4D + 2 \ell)^{\ell} C^{4D + \ell} + C_1 N^{-2 \log N}.
\eeq
\eel
\begin{proof}
Fix index pairs $(a,b)$ and $(i,k)$. By resolvent expansion \eqref{resolventexpansion}, we have
\beq\label{rr1}
\tilde G_{ab} = G_{ab} + xH_{ik} \tilde G_{ai} G_{kb} + 
x H_{ik} \tilde G_{ak} G_{ib},
\eeq
where $\tilde G$ is the resolvent of $H^{(ik)} + x E^{[ik]}$.
By \eqref{eqn:tail}, we have $|H_{ik}| \le 1$ with probability at least $1 - c^{-1} \exp(-  N^{c/2})$. Combining this bound with \eqref{e:gentries} and \eqref{rr1}, we obtain uniformly in $z\in \mathcal D_A$, with high probability, that
\beq\label{perturbationbound}
\sup_{|x| \le N^{-1/4} } \max_{a,b} \left| \tilde G_{ab} \right| \le C.
\eeq

Since $P$ is a quadratic polynomial in $m$, $F_{ki}$ is a polynomial of degree $4D-1$ in $G_{ki}$ and $m$, with at most $2^{2D-1}$ terms. When $\partial_{jk}$ acts on $P$ or $\overline{P}$ it generates a new factor 
\beq\label{newf}
2N^{-1} \sum_{i_l} G_{j i_l} G_{i_l k},
\eeq
with a new summation index $i_{l}$.
Then $\partial^\ell_{ik} F_{ki}$ has degree
$4D-1 + \ell$ when considered as a polynomial in Green's function entries $G_{ii}$, $G_{kk}$, $G_{ik}$, $m$, and terms of the form \eqref{newf}. The number of such terms in $\partial^\ell_{ik} F_{ki}$ is bounded by $2^{2D-1} \times (4D-1 + 2\ell)^{\ell}$. 
Using \eqref{perturbationbound}, the contribution to the left side of \eqref{Fperturb} from the expectation on the set where \eqref{e:gentries} holds is bounded by the first term of \eqref{Fperturb} (after increasing $C$). On the low probability set where \eqref{perturbationbound} does not hold, we use the trivial bound $|G_{ij}| \le \eta^{-1}$ (from \eqref{e:gtrivial}) and the assumed lower bound $\eta > \phi^{-A}$, which produces the second term of $\eqref{Fperturb}$.
\end{proof}

\bel\label{l:truncate}
Let $A >0$, $4 \log N \le \ell \le A \log N$, $D \le \log N$, and $z \in \mathcal D_{A}$. There exists a constant $ C(A) >0$ such that the term $\Omega_\ell$ from \eqref{exp1} satisfies 
$ \sup_{z \in \mathcal D_{A} } \big| \E[ \Omega_{\ell}] \big| \le C N^{-\ell /4}$ .
\eel
\begin{proof}
By Lemma~\ref{cumulant} with $Q = N^{-1/4}$, 
\begin{align}
\Big|  \E \big[ \Omega_\ell (H_{ik} F_{ki})  \big]   \Big| &\le 
C_\ell \E \left[  | H_{ik} |^{\ell+2} \right]  \E\left[ \sup_{|x| \le N^{-1/4} } \left|  F^{(\ell+1)} (H^{(ik)} + x E^{[ik]}) \right|\right]\\
&+ 
C_\ell \E \left[ | H_{ik} |^{\ell+2} \one\big(|H_{ik} | >  N^{-1/4}\big) \right] 
\E \left[ \sup_{x \in \R} \left|  F_{ik}^{(\ell+1)} (H^{(ik)} + x E^{[ik]}) \right|\right].
\end{align}
The first term may be bounded using Lemma~\ref{l:Fperturb}, Lemma~\ref{cumulant}, and the moment bound \eqref{momentgrowth}:
\beq
C_\ell \E \left[  | H_{ik} |^{\ell+2} \right] \E\left[  \sup_{|x| \le N^{-1/4} } \left|  F^{(\ell+1)} (H^{(ik)} + x E^{[ik]}) \right|\right]
\le \frac{ (C \ell)^{C\ell } } { \ell! N^{\ell/2}}\left( (4D + 2\ell)^{\ell} C^{4D + \ell} + C N^{-2 \log N}\right).
\eeq
For the second term, we use the trivial bound $|G_{ij}| \le \eta^{-1}$ (from \eqref{e:gtrivial}) to obtain the deterministic bound
\beq
\sup_{x \in \R} \left|  F_{ik}^{(\ell+1)} (H^{(ik)} + x E^{[ik]}) \right| \le 
2^{2D} (4D + 2\ell)^{\ell} \left( \frac{C}{\eta}\right)^{4D + \ell},
\eeq
and using H\"older's inequality and \eqref{eqn:tail} we have 
\beq
\E \left[ | H_{ik} |^{\ell+2} \one(|H_{ik} | >  N^{-1/4}) \right] 
\le
\frac{ (C \ell)^{C\ell } } { N^{\ell/2}}\exp( -c N^{c/4} ).
\eeq
Combining our estimates for the first and second terms of $\Omega_\ell$ yields the conclusion after using $\ell \ge 4\log N$.
\end{proof}
\subsubsection{Main Terms.}\ We will need to analyze the  terms $I_{r,s}$ from \eqref{Irs} explicitly for $s\le 1$. For the others, we can proceed on general combinatorial grounds. The following lemma collects our estimates on these terms. We set $P'= P'(m) = 2m + z$.
\bel\label{mainterms}
Fix $A>0$ and suppose $D \le \log N$. For all $z \in \mathcal D_{A}$, we have 
\begin{equation}\label{I1}
\E [I_{1,0}] + \E [I_{1,1}] = - \E[m^2 P^{D-1} \overline{P}^{D} ]+ \Omega,
\end{equation}
where $\Omega = \Omega(z) \ge 0$ is an error term satisfying
\begin{equation}
\Omega \le 
 \left( \frac{C}{N \eta} \right) \E \left[ \left| P \right|^{2D-1}   \right] 
+
D \left(\frac{C}{N \eta}\right)^2\E \left[  \left| P\right|^{2D-2}   \right] + C N^{-2 \log N}.
\end{equation}
Further, we have
\begin{equation}\label{I2}
\big| \E [I_{2,0}]  \big|  \le  C(\log N) N^{-1/2}
\sum_{a=1}^{2D} \left( \frac{1}{N\eta} \right)^a \E\left[ |P|^{2D-a}  \right]  + N^{-2 \log N}
\end{equation}
and
\begin{equation}\label{l25}
\big| \E [I_{2,1}]  \big|  \le C(\log N) N^{-1/2}
\sum_{a=1}^{2D} \left( \frac{1}{N\eta} \right)^a \E\left[ |P|^{2D-a}  \right]  + N^{-2 \log N}.
\end{equation}
For $r \ge 3$ we have 
\beq\label{I3}
\big| \E [I_{r,0}] \big|  + \big| \E [I_{r,1}] \big| \le  \frac{ (Cr)^{Cr} } { N^{(r-1)/2}}  \E \left[  |P|^{2D-1}   \right] +  D C^r \frac{ (Cr)^{Cr} } { N^{(r-1)/2}}  
 \frac{C}{N\eta} \E \left[  |P|^{2D-2} \right]  + C N^{-2 \log N}.
\eeq
Finally, for $r\ge 2 $, 
\beq\label{I4a}
\big| \E [I_{r,r}]  \big|
\le
\frac{ (Cr)^{Cr} } { N^{(r-1)/2}}
\sum_{s_0 = 2}^{r+1}   \left(\frac{C}{N\eta} \right)^{1/2}
  \left(\frac{CD}{N\eta}  \right)^{s_0 - 1}  \E \left[  |P |^{2D - s_0} \right] + C N^{-2\log N},
\eeq
and for pairs $(r,s)$ such that $r>s \ge 2$, 
\beq\label{I4}
\big| \E [I_{r,s}] \big|
\le
\frac{ (Cr)^{Cr} } { N^{(r-1)/2 }}
\sum_{a = 2}^{r+1} \left(\frac{CD}{N\eta}  \right)^{a - 1} \E \left[ 
 | P |^{2D - a} \right] + C N^{-2\log N} .
\eeq
\eel
For clarity we prove the claims \eqref{I1}, \eqref{I2}, \eqref{l25}, \eqref{I3}, \eqref{I4a}, and \eqref{I4} separately.
\begin{proof}[Proof of \eqref{I1}]
We write, using $\kappa_2 = N^{-1}$,
\beq
\E [I_{1,0}]  = \frac{1 }{N^2}  \E \left[\sum_{i,k}(1 + \delta_{ik}\Delta_{r+1}) \left( \partial_{ik} G_{ki} \right) P^{D-1} \overline{P}^D   \right].
\eeq
We first bound the terms with $i=k$, which are sub-leading. Using Theorem \ref{l:lscl}, we have
\beq\label{diagonalremainder}
 \frac{1 }{N^2} \left|  \E \left[\sum_{i}(1 + \Delta_{r+1}) \left( \partial_{ii} G_{ii} \right) P^{D-1} \overline{P}^D   \right] \right| \le C N^{-1} \E \left[ |P|^{2D-1} \right] + CN^{-2 \log N}.
\eeq
For $i\neq k$, we have
\begin{align}\label{I1terms}
\left( \partial_{ik} G_{ki} \right) P^{D-1} \overline{P}^D   = -G_{ii} G_{kk} P^{D-1} \overline{P}^{D} 
- G_{ki} G_{ki} P^{D-1} \overline{P}^{D}.
\end{align}
Considering the first term, we have 
\beq
\frac{1}{N^2}\sum_{i,k}  - G_{ii} G_{kk} P^{D-1} \overline{P}^{D} = - m^2 P^{D-1} \overline{P}^{D},
\eeq
which matches the first term of \eqref{I1}. 
Considering the second term of \eqref{I1terms} and using \eqref{smallpowercount}, we have 
\begin{equation}
\frac{1}{N^2}
\left|\sum_{i,k} G_{ki} G_{ki} P^{D-1} \overline{P}^{D} \right| \le  \frac{C}{N \eta}  \E \left[ \left| P \right|^{2D-1}   \right] + N^{-2 \log N} \eer.
\end{equation}
This completes the analysis of $I_{1,0}$. 
Next, we have
\beq
\E [I_{0,1}]  = \frac{1 }{N^2}  \E \left[\sum_{i,k}(1 + \delta_{ik}\Delta_{r+1}) G_{ki} \cdot \partial_{ik} \left( P^{D-1} \overline{P}^D \right)  \right],
\eeq
and 
\begin{align}\label{apple}
\partial_{ik} \left( P^{D-1} \overline{P}^D \right) = -&\frac{2(D-1)}{N} G_{ki} P' \sum_{j=1} G_{jk} G_{ij} P^{D-2}  \overline{P}^{D}
-\frac{2D}{N} G_{ki} \overline{P}' \sum_{j=1}^N 
\overline{G_{jk} G_{ij}} P^{D-1} \overline{P}^{D-1}.
\end{align}
As in \eqref{diagonalremainder}, we can remove the $i=k$ terms with negligible error.
For the first term of \eqref{apple}, we use \eqref{e:gentries} and \eqref{powercount} to get
\begin{equation}
\frac{1}{N^3}\left| \E \left[ \sum_{i \neq k} G_{ki} P' \sum_{j=1} G_{jk} G_{ij} P^{D-2}  \overline{P}^{D} \right]\right|
\le \left(\frac{C}{N \eta}\right)^2\E \left[  | P' | | P|^{2D-2}   \right] + N^{-2 \log N}.
\end{equation}
The bound on the second term of \eqref{apple} is similar.
\end{proof}

\begin{proof}[Proof of \eqref{I2}]
We write
 $I_{2,0} =  I^{(1)}_{2,0} + I^{(3)}_{2,0} +R_1$,
where $I^{(1)}_{2,0}$ contains all terms with exactly one off-diagonal resolvent entry, $I^{(1)}_{2,0}$ contains all terms with three off-diagonal resolvent entries, and $R_1$ contains all other terms. Reasoning as in \eqref{diagonalremainder}, we see that $R_1$ is negligible:
\beq
\big| \E [R_1] \big| \le N^{-1} \E \left[ |P|^{2D-1} \right] + C N^{-2 \log N}.
\eeq
By Theorem \ref{l:lscl}, we get
\begin{equation}\label{i2b}
\left| \E [I^{(3)}_{2,0}] \right| 
=
\frac{\kappa_{3} }{N} \left| \E \left[ \sum_{i \neq k} (G_{ik})^3 P^{D-1} \overline{P}^{D}  \right] \right|
\le N^{-1/2} \left( \frac{C}{N\eta}\right)^{3/2} \E \left[ |P|^{2D-1} \right] + C N^{-2 \log N}.
\end{equation}
To study $ I^{(1)}_{2,0}$, we perform another cumulant expansion using \eqref{cumulantexpansion}:
\beq\label{i2c}
z \E [I^{(1)}_{2,0}] = \sum_{r=1}^{\ell'} \sum_{s=0}^r w_{r,s} \E[ \hat I_{r,s}] + \hat \Omega_{\ell'},
\eeq
where
\beq
\E [\hat I_{r,s}] = N \kappa_{r+1} N \kappa_{3} \E \left[ \frac{1}{N^3} \sum_{ i \neq k, j} \left( \partial^{r-s}_{kj}
 \left( G_{ji} G_{kk} G_{ii} \right)\right)
\left( \partial^s_{kj} \left( P^{D-1} \overline{P}^{D} \right) \right)\right].
\eeq
We take $\ell' = 20 \log N$ and see that 
\beq\label{i2d}
| \hat \Omega_{\ell'} | \le C N^{- 2 \log N}
\eeq
by a straightforward modification of the proof of Lemma~\ref{l:truncate}.

We begin with the terms $\hat I_{r,0}$ for $r\ge 1$. When $r=1$, we define terms $\hat I^{(i)}_{1,0}$ for $i = 1,2,3$ by the decomposition
\begin{align}
\E [\hat I_{1,0}]  = 
& - N \kappa_3 \E \left[ N^{-3} \sum_{i_1\neq i_2 , i_3} G_{i_2 i_1} G_{i_3 i_3} G_{i_2 i_2} G_{i_1 i_1} P^{D-1} \overline{P}^{D} \right]\\
& -3 N \kappa_3\E \left[  N^{-3}\sum_{i_1\neq i_2 , i_3} 
G_{i_2 i_3} G_{i_3 i_1} G_{i_2 i_2} G_{i_1 i_1} P^{D-1} \overline{P}^{D} 
 \right]\\
& -2 N \kappa_3 \E \left[ N^{-3} \sum_{i_1\neq i_2 , i_3}
G_{i_1 i_2} G_{i_2 i_3} G_{i_3 i_1} G_{i_2 i_2} P^{D-1} \overline{P}^{D} 
  \right] \\
& = \E [\hat I^{(1)}_{1,0}] + 3\E [\hat I^{(2)}_{1,0}] + 2 \E [\hat I^{(3)}_{1,0} ].
\end{align}
where we have decomposed the sum according to the number of off-diagonal resolvent entries
in each product, 
Using Lemma~\ref{l:pc} on the off-diagonal resolvent entries, we find
\beq\label{i2f}
\big| \E [\hat I^{(2)}_{1,0}]\big| + | \E [\hat I^{(3)}_{1,0}]| \le N^{-1/2}\left( \frac{C}{N\eta} \right) \E \left|  P\right|^{2D-1} + N^{-2 \log N}.
\eeq
Further, by using \eqref{e:gentries} and incurring a negligible error, we may replace $\E [\hat I^{(1)}_{1,0}]$ by 
\beq\label{i2e}
- \msc \E \left[ N^{-2} \sum_{i_1 \neq i_2} G_{i_2 i_1} G_{i_2 i_2} G_{i_1 i_1} P^{D-1} \overline{P}^{D} \right] = - \msc \E [I_{2,0}^{(1)}].
\eeq
We now turn to terms $\hat I_{r,0}$ with $r>1$. 
We observe that Lemma~\ref{l:moments} implies
\beq\label{r13}
\left( N \kappa_{r+1}\right) \left( N \kappa_3  \right) \le (Cr)^{Cr} N^{ - r/2}
\eeq
and recall that $|w_{r,0}| \le 1$. Since every product in $\hat I_{r,0}$ has at least one off-diagonal entry, by \eqref{e:gentries} we have the bound 
\beq
w_{r,0} \left| \E [\hat I_{r,0}]\right| \le (Cr)^{Cr} N^{ - r/2}
\left(  \frac{1}{N\eta} \right)^{1/2} \E \left| P \right|^{2D-1} + N^{-2 \log N}.
\eeq
We next consider the terms $\hat I_{r,s}$ with $s=1$. In this case, we first note that the order $r-1$ derivative of the product of resolvent entries in $\hat I_{r,1}$ contributes at least one off-diagonal resolvent entry. Next, we see that  each $\partial_{ij} \left( P^{D-1} \overline{P}^D \right)$ contains either two factors of \eqref{newf}, or the derivative of \eqref{newf}. 
Using \eqref{r13}, this leads to the bound 
\beq
\left| \E [\hat I_{r,1}] \right| \le  D (Cr)^{Cr} N^{ - r/2} \left(\frac{1}{N\eta}\right)^{3/2} \E \left[ |P'| |P|^{2D-2} \right].
\eeq
We absorbed the combinatorial factor corresponding to the number of terms coming from the derivatives, which is bounded by $C^r$, into the prefactor.

Now we consider the case of $\hat I_{r,s}$ with $2 \le s \le r$. We begin by noting that \eqref{e:gentries} gives
\begin{align}
\left|  \E [\hat I_{r,s}] \right| &\le 
(Cr)^{Cr} N^{ - r/2} \left|  \E \left[\frac{1}{N^3} \sum_{ i \neq k, j} \left( \partial^{r-s}_{jk}
 \left( G_{ji} G_{kk} G_{ii} \right)\right)
\left( \partial^s_{jk} \left( P^{D-1} \overline{P}^{D} \right) \right) \right] \right|\\
& \le (Cr)^{Cr} N^{ - r/2} \left(\frac{1}{N\eta}\right)^{1/2}  \E \left[\frac{1}{N^2} \sum_{ k, j} 
\left| \partial^s_{jk} \left( P^{D-1} \overline{P}^{D} \right) \right| \right].
\end{align}
Here we used the fact that $ \partial^{r-s}_{jk}
 \left( G_{ji} G_{kk} G_{ii} \right)$ always contains at least one off-diagonal resolvent entry, and again absorbed the combinatorial factor corresponding to the number of terms in this derivative into the prefactor.

For $s \ge 2$, the derivative $\partial^s_{jk} \left( P^{D-1} \overline{P}^{D} \right)$ is a sum of terms that may contain factors of $P$ and $P'$, and their conjugates. Any such term is a constant times a product of the form $
P^{D-s_1} \overline{P}^{D-s_2} (P')^{s_3}\left( \overline{P'}\right)^{s_4}$, 
with $s_i \ge 0$ for $i \in \unn{1}{4}$. 
A term with such a product came from $\partial_{jk}$ acting $s_1 - 1$ times on $P$ and $s_2$ times on $\overline{P}$, so we must have $s_1 - 1 \ge s_3$, $s_2 \ge s_4$, and $s_1 - 1 + s_2 \le s$. We further see that $\partial_{jk}$ acted $s_1 - 1 - s_3$ times on $P'$ and $s_2  - s_4$ times on $\overline{P'}$. 

When $\partial_{jk}$ acts on a power of $P$, $\overline{P}$, or their derivatives, it generates a new factor of
$\partial_{jk}m = 2N^{-1} \sum_{i_l} G_{j i_l} G_{i_l k}$,
where $i_l$ is a new summation index (not appearing elsewhere), and a constant prefactor no greater than $D$ (by the chain rule applied to $P^k$ and analogous terms for $k \le D$). The number of new summation indices is then
\be\label{octindex}
s_1 -1 + s_2 + (s_1 -1 - s_3) + (s_2 - s_4) = 2s_1 + 2s_2 - s_3 - s_4  -2.
\ee
Further, this number does not decrease when $\partial_{jk}$ acts on resolvent entries instead of $P$, $\overline{P}$, or their derivatives. We introduce $a= s_1 + s_2$ and $b = s_3 + s_4$. Then, using $a \le s + 1$, \eqref{octindex} yields
\begin{align}
\left|  \E [\hat I_{r,s}] \right| & 
\le (Cr)^{Cr} D^r N^{ - r/2}  \sum_{a=2}^{s+1}
\sum_{b = 0}^{a-2}
 \left(\frac{1}{N\eta}\right)^{1/2 + 2a - b - 2}
 \E \left[   |P'|^{b} |P|^{2D-a} \right]\\
 & + (Cr)^{Cr} D^r N^{ - r/2} \sum_{a = 2}^{s+1}
  \left(\frac{1}{N\eta}\right)^{1/2 + 2a - 1}
 \E \left[   |P'|^{a-1} |P|^{2D-a} \right],
\end{align}
where the second term comes from terms such that $b = a-1$.
After increasing $C$, we obtain
\begin{align}
\left|  \E [\hat I_{r,s}] \right|  \le (Cr)^{Cr} D^r N^{ - r/2} 
 \sum_{a=2}^{r+1} \left( \frac{1}{N\eta} \right)^a \E\left[ |P|^{2D-a}  \right] + C N^{-5\log N}.
\end{align}
Combining the definition of $I_{2,0}$, \eqref{i2b}, \eqref{i2c}, \eqref{i2d}, and the estimates on the $\hat I_{r,s}$ terms, we obtain 
\begin{align}
\left| (z +  \msc) \E [I_{2,0}^{(1)}]\right| \le&  \sum_{r=2}^{8 \log N}  (Cr)^{Cr} D^r N^{ - r/2} 
\sum_{a=1}^r \left( \frac{1}{N\eta} \right)^a \E\left[ |P|^{2D-a}  \right] \\
&+  \sum_{r=2}^{8 \log N} (Cr)^{Cr} D N^{ - r/2} \left(\frac{1}{N\eta}\right)^{3/2} \E \left[|P|^{2D-2} \right]\\
&+  \sum_{r=2}^{8 \log N}  (Cr)^{Cr} N^{ - r/2}
\left(  \frac{1}{N\eta} \right)^{1/2} \E \left| P \right|^{2D-1}\\
&+ N^{-1/2}\left( \frac{C}{N\eta} \right) \E \left|  P\right|^{2D-1} + N^{-2 \log N}.
\end{align}
After some simplification and increasing the value of $C$, this implies
\begin{align}
\left| (z +  \msc) \E [I_{2,0}^{(1)}]\right| &\le  C(\log N) N^{-1/2}
\sum_{a=1}^{2D} \left( \frac{1}{N\eta} \right)^a \E\left[ |P|^{2D-a}  \right]  + N^{-2 \log N}.
\end{align}
Using $|z + \msc(z) | > c(A)>0$ on $\mathcal D_A$ (see \eqref{e:msc}), we obtain the conclusion.
\end{proof}
\begin{proof}[Proof of \eqref{l25}]
We have 
\begin{align}
 \E [I_{2,1 }] &=   N \kappa_{3}  
\E \left[ N^{-2}\sum_{i,k}(1 + \delta_{ik}\Delta_{3}) \left( \partial_{ik} G_{ki} \right)   \partial_{ik} \left(P^{D-1} \overline{P}^D \right)    \right].
\end{align}
Using the logic of the previous proof, it is straightforward to see that the leading-order contribution is given by 
\begin{align}
J =& N \kappa_3 \E\left[ \frac{2(D-1)}{N^2} \sum_{i_1 \neq i_2 } 
G_{i_1 i_1} G_{i_2 i_2} \left(\frac{1}{N} \sum_{i_3 =1}^N G_{i_2 i_3} G_{i_3 i_1}   \right) P' P^{D-2}
\overline{P}^D
 \right]\\
 & + N \kappa_3 \E\left[ \frac{2D}{N^2} \sum_{i_1 \neq i_2 } 
G_{i_1 i_1} G_{i_2 i_2} \left(\frac{1}{N} \sum_{i_3 =1}^N G_{i_2 i_3} G_{i_3 i_1}   \right) P' P^{D-1}
\overline{P}^{D-1}
\right].
\end{align}
For the first term, 
we use the resolvent expansion to write it as
 \beq
 2(D-1) N \kappa_3 \E\left[ \frac{1}{N^3} \sum_{i_1 \neq i_2 , i_3  , i_4} 
H_{i_2 i_4} G_{i_4 i_3} G_{i_1 i_1} G_{i_2 i_2} G_{i_3 i_1}  P' P^{D-2}
\overline{P}^D \right].
 \eeq
 As in the previous proof, we now use the cumulant expansion \eqref{cumulantexpansion} to calculate this term (expanding each term in the sum in the variable $H_{i_2 i_4}$). The leading term in the expansion is 
 \beq
  2(D-1) N \kappa_3 \E\left[ \frac{1}{N^3} \sum_{i_1 \neq i_2, i_3 } 
m  G_{i_2 i_3} G_{i_2 i_2} G_{i_3 i_1}  P' P^{D-2}
\overline{P}^D \right],
\eeq
and we obtain
\begin{multline}\label{109}
\left| z + \msc(z) \right|  \left| 2(D-1) N  \kappa_3  \right|\E\left[ \frac{1}{N^3} \sum_{i_1 \neq i_2, i_3 } 
G_{i_2 i_3} G_{i_2 i_2} G_{i_3 i_1}  P' P^{D-2}
\overline{P}^D \right]\\ \le C(\log N) N^{-1/2}
\sum_{a=1}^{2D} \left( \frac{1}{N\eta} \right)^a \E\left[ |P|^{2D-a}  \right]  + N^{-2 \log N}.
\end{multline}
This controls the first term of $J$.
By nearly identical reasoning, a similar bound holds for the second term of $J$. Using $|z + \msc(z) | \ge c$ for $ z \in \mathcal D_A$ in \eqref{109}, we obtain the result.
\end{proof}
\begin{proof}[Proof of \eqref{I3}]
By Theorem \ref{l:lscl} we have
\begin{align}
| \E [I_{r,0}] |&= \left| N \kappa_{r+1}  
\E \left[ N^{-2}\sum_{i,k}(1 + \delta_{ik}\Delta_{r+1}) \left( \partial^{r}_{ik} G_{ki} \right)  P^{D-1} \overline{P}^D    \right] 
 \right|\\
 &\le \frac{ (Cr)^{Cr} } { N^{(r-1)/2}}  \E \left[  |P|^{2D-1}   \right] + C N^{-2 \log N} .
\end{align}
We absorbed the combinatorial factor $4^r$ representing the number of different terms coming from the derivatives of $G_{ki}$ into the constant.

For $I_{r,1}$, we have 
\beq
\left| \E [I_{r,1}] \right| \le \frac{ (Cr)^{Cr} } { N^{(r-1)/2}} \left|  \E \left[ N^{-2}\sum_{i,k}(1 + \delta_{ik}\Delta_{r+1}) \left( \partial^{r-1}_{ik} G_{ki} \right) \partial_{ik}\left( P^{D-1} \overline{P}^D  \right)  \right] \right|.
\eeq
From the derivative on $P^{D-1} \overline{P}^D$ we get $|P'| |P|^{2D-2}$, a factor of $2 N^{-1} \sum_{a} G_{ia} G_{ak}$, a factor of $D$, and some constant that is bounded uniformly in $r$. For the terms with at least $3$ off-diagonal $G_{ab}$, we can use Theorem \ref{l:lscl} and Lemma~\ref{l:pc} to get the bound

\begin{align}
DC^r \frac{ (Cr)^{Cr} } { N^{(r-1)/2}}  
& \E \left[ N^{-3}\sum_{i,k,a }  |G_{ki} G_{ia}G_{ak}| |P'| |P|^{2D-2}  \right] \\
&\le 
DC^r \frac{ (Cr)^{Cr} } { N^{(r-1)/2}}  
\left( \frac{C}{N\eta} \right)^{3/2} \E \left[  |P'| |P|^{2D-2}  \right]  + CN^{-5\log N} \\
&\le 
DC^r \frac{ (Cr)^{Cr} } { N^{(r-2)/2}}  
\left( \frac{C}{N\eta} \right)^{2} \E \left[  |P|^{2D-2}  \right] + CN^{-5\log N}.
\end{align}
For terms with only two off-diagonal entries, we have
\begin{align}
DC^r \frac{ (Cr)^{Cr} } { N^{(r-1)/2}}  
&\E \left[ N^{-2}\sum_{i,k} G^{r/2}_{ii} G^{r/2}_{kk} \left( \frac{1}{N} \sum G_{ia} G_{ak} \right) P' P^{D-2} \overline{P}^D\right] \\
&\le D C^r \frac{ (Cr)^{Cr} } { N^{(r-1)/2}}  
 \frac{C}{N\eta} \E \left[  |P|^{2D-2} \right]  + C N^{-2 \log N}
\end{align}
and the same bound for the term where the derivative falls on $ \overline{P}^D$. This completes the proof.
\end{proof}

\begin{proof}[Proof of \eqref{I4a} and \eqref{I4}] 
We treat all terms $I_{r,s}$ with $2 \le s \le r$  the same way. From the order $s$ derivatives in $I_{r,s}$, we get a sum of monomials of the form
\beq\label{Pmonomial}
P^{D-s_1} \overline{P^{D-s_2}} \left( P'  \right)^{s_3} \overline{\left( P'  \right)}^{s_4}
\prod_{d=1}^n Q_d,
\eeq
with $1 \le s_1 \le D$ and $0 \le s_2 \le D$ depending on how the derivatives fall, and we have omitted the constant prefactor. Each $Q_d$ represents a ``fresh summation index'' $i_d$ in the following way. Any derivative on $P$, $\overline{P}$ generates a new summation index $a$ with a factor $\partial_{ik} m = 2 N^{-1} \sum_{a} G_{ia} G_{ak}$. Similarly a derivative on $P'$ or $\overline{P}'$ gives $4 N^{-1} \sum_{a} G_{ia} G_{ak}$. Each $Q_d$ represents a sum corresponding to one of these new indices, potentially differentiated further. For example, $Q_1$ could be 
$
2 N^{-1} \sum_{a} G_{ia} G_{ak}$, 
or
(applying a derivative $\partial_{ik}$)
\beq\label{above1}
2 N^{-1} \sum_{a} \left( G_{ii} G_{ak} G_{ak}  + G_{ik} G_{ia} G_{ak}  + G_{ia}G_{ia} G_{kk} + G_{ia} G_{ak} G_{ik} \right),
\eeq
or any higher derivative. We consider any constant factors that are produced when a new index is generated, or when a $Q_d$ term is differentiated, as part of the corresponding $Q_d$ term. For example, in \eqref{above1}, we consider the factor $2$ as part of $Q_d$.

We see that the monomial \eqref{Pmonomial} came from $P^{D-1} \overline{P^D}$ because $s_1-1$  derivatives $\partial_{ik}$ derivatives fell on $P$ and $s_2$ on $\overline{P}$. This implies $s_3 \le s_1 - 1$ and $s_4 \le s_2$. The number of derivatives on $P'$ was $s_1 -1 - s_3$, and on $\overline{P'}$ it was $s_2 - s_4$. Then the total number of new indices is 
\beq
n= 2s_1 + 2s_2 - s_3 - s_4 - 2.
\eeq
 Let the number of derivatives that fall on some on some $Q$-type term be $s_5$. Note that $n+ s_5 = s$. 

We next consider the constant factor associated to \eqref{Pmonomial}. There are two contributions to this: the number of times such a monomial appears through differentiation, and a factor from the derivatives of powers of $P$, $P'$, and their conjugates. We bound the first contribution by the total number of monomials produced, which is crudely bounded by $(s+4)^s$, because a derivative of a term of the form \eqref{Pmonomial} produces $n+4 \le s +4 $ new monomials of the same form, one for each choice of factor to differentiate. For the second, we see that the $P$-type terms appear with power at most $D$, and there are $n$ total derivatives applied to them, so this contribution is bounded by $D^n$. We therefore see that the constant factor is no larger than $(Cr)^{Cr} D^n$.

Now consider bounding each monomial. We will bound the $P$ and $P'$ terms (and their conjugates) by their absolute values. For the $Q$ terms, we will use Theorem \ref{l:lscl} to bound the $G_{ab}$ terms with no new index, and then invoke \eqref{powercount} in the $m=2$ case. We must further track the constant pre-factors coming from derivatives of $Q$ terms. Each $Q$ term starts as $
2 N^{-1} \sum_{a} G_{ia} G_{ak}
$
or
$
4 N^{-1} \sum_{a} G_{ia} G_{ak},
$
and each successive derivative multiplies the number of terms by $4$. Recall there are $s_5$ such derivatives. Combining the bound on the number of new terms and the semicircle law bound, we get a $C^{s_5}$ factor, which we bound by $C^r$. 

We first consider the case $r=s$. Set $s_0 = s_1 + s_2$ and $s' = s_3 + s_4$. We recall there are $n=2s_0 - s'  -2$ new indices. Using power counting \eqref{powercount} and noting the isolated off-diagonal $G_{ik}$ term, which is not differentiated, we get 
\begin{align}\label{Irr}
\left|  \E I_{r,r} \right|
&\le  \frac{ (Cr)^{Cr} } { N^{(r-1)/2}}
\sum_{s_0 = 2}^{r+1} \sum_{s' =0}^{s_0-1}  \left(\frac{C}{N\eta} \right)^{1/2}
\left(\frac{CD}{N\eta}  \right)^{2s_0 - s' - 2}  \E \left[ 
| P ' |^{s'} | P |^{2D - s_0} \right]\\
& \le \label{Irrold}
  \frac{ (Cr)^{Cr} } { N^{(r-1)/2}}
\sum_{s_0 = 2}^{r+1}   \left(\frac{C}{N\eta} \right)^{1/2}
  \left(\frac{CD}{N\eta}  \right)^{s_0 - 1}  \E \left[  P |^{2D - s_0} \right] + C N^{-2\log N},
\end{align}
where we increased $C$ in the second line.

For $s \neq r$ we get 

\begin{align}
\left|  \E I_{r,s} \right|
&\le \frac{ (Cr)^{Cr} } { N^{(r-1)/2 }}
\sum_{s_0 = 2}^{r+1} \sum_{s' =1}^{s_0-2} 
\left(\frac{CD}{N\eta}  \right)^{2s_0 - s' - 2}
\E \left[ 
| P ' |^{s'} | P |^{2D - s_0} \right]\\
&+
\frac{ (Cr)^{Cr} } { N^{(r-1)/2 }}
\sum_{s_0 = 2}^{r+1} \left(\frac{CD}{N\eta}  \right)^{s_0 - 1} \E \left[ 
| P ' |^{s_0-1} | P |^{2D - s_0} \right]
\end{align}
The second term bounds the terms with $s_0 - 1 = s'$ that come from when $\partial_{ik}$ lands $s_1 - 1$  times on $P$ and $s_2$ times on $\overline{P}$, and their derivatives are not hit. By Theorem \ref{l:lscl}, we obtain
the desired bound
\begin{align}
\left|  \E I_{r,s} \right|
\le
\frac{ (Cr)^{Cr} } { N^{(r-1)/2 }}
\sum_{s_0 = 2}^{r+1} \left(\frac{CD}{N\eta}  \right)^{s_0 - 1} \E \left[ 
 | P |^{2D - s_0} \right] + C N^{-5\log N}.
\end{align}
\end{proof}

\subsection{Proof of Moment Bound.}\ 
\label{s:momentproof}
\begin{proof}[Proof of \Cref{p:moments}]
We proceed by induction to bound the powers $\E\big[ | P |^{p}\big]$, with $P$ as in \eqref{eq-of2}.
 The base case $p=0$ is trivial. 
For the induction step, suppose $D \le (\log N)/2$, and that there exists $C_1(K,\kappa) > 0$ such that 
\begin{equation}\label{inductionhypo}
\E\left[ | P |^{p}\right] \leq \left(\frac{C_1}{N\eta} \right)^p p^{3p/4}
\end{equation}
for all $p \le 2D-2$. We will show that if $C_1$ is chosen large enough, in a way that does not depend on $D$, then \eqref{inductionhypo} also holds for $p=2D$ and $p={2D-1}$. 

 Set $\ell = 20 \log N$. Combining \eqref{twoterms}, \eqref{exp1}, Lemma~\ref{l:truncate}, and Lemma~\ref{mainterms}, we have
 \begin{align}
 \E\left[\left|  P \right|^{2D} \right] \le&  \left( \frac{C}{N \eta} \right) \E \left[ \left| P \right|^{2D-1}   \right] 
+
D \left(\frac{C}{N \eta}\right)^2\E \left[  | P|^{2D-2}   \right]\\
&+C(\log N) N^{-1/2}
\sum_{a=1}^{2D} \left( \frac{1}{N\eta} \right)^a \E\left[ |P|^{2D-a}  \right]  \\
&+ \sum_{r=3}^{8 \log N}  \frac{ (Cr)^{Cr} } { N^{(r-1)/2}}  \E \left[  |P|^{2D-1}   \right] +  \sum_{r=3}^{8 \log N} D C^r \frac{ (Cr)^{Cr} } { N^{(r-1)/2}}  
 \frac{C}{N\eta} \E \left[  |P|^{2D-2} \right] \\
&+ \sum_{r=2}^{8 \log N} \frac{ (Cr)^{Cr} } { N^{(r-1)/2}}
\sum_{a = 2}^{r+1}   \left(\frac{C}{N\eta} \right)^{1/2}
  \left(\frac{CD}{N\eta}  \right)^{a - 1}  \E \left[  P |^{2D - a} \right]\\
 &+ \sum_{r=3}^{8 \log N} \frac{ (Cr)^{Cr} } { N^{(r-1)/2 }}
\sum_{a = 2}^{r+1} \left(\frac{CD}{N\eta}  \right)^{a - 1} \E \left[ 
 | P |^{2D - a} \right] + C N^{-2\log N}.
 \end{align}
 We now use \eqref{inductionhypo} in the above inequality. After increasing $C$, 
 and choosing $C_1>C$ in a way that only depends on $C$, $K$, and $\kappa$, we obtain
 \begin{equation}\label{inductionhyponew}
\E\left[ | P |^{2D}\right] \leq \left(\frac{C_1}{N\eta} \right)^{2D} (2D)^{3D/2}.
\end{equation}
Further, from H\"older's inequality, we obtain as desired that 
\beq
\E\left[ | P |^{2D-1}\right] \le \E\left[ | P |^{2D}\right]^{\frac{2D-1}{2D}} = 
\left(\frac{C_1}{N\eta} \right)^{2D-1} (2D)^{3(2D-1)/4},
\eeq
which completes the induction step.

We now recall the stability of the defining equation $u^2 = zu + 1 =0$ for $\msc$ (see \cite[Lemma 5.5]{benaych2016lectures}). Set $\mathcal A_z = \{ | P(z) | \le 1 \}$. Then $m$ satisfies $
\one_{\mathcal A_z} \big| m(z) - \msc(z) \big| \le C \big|P(z)\big|$ 
for some $C(\kappa) >0$. The claimed result follows from \eqref{inductionhypo} and the bound $ | m(z) - \msc(z) | \le 2 \eta^{-1}$ on the set $\mathcal A^c_z$, which has negligible probability by Markov's inequality and 
\eqref{inductionhypo}.
\end{proof}

\section[Mesoscopic Fluctuations for  $\beta$-ensembles]
{Mesoscopic Fluctuations for $\beta$-ensembles}
\label{sec:gauss}

This appendix considers $\beta$-ensembles as defined in (\ref{eq:beta_ensembles_density}).
We follow Johansson's loop equations method from \cite{Joh1998} to establish Gaussian fluctuations of the characteristic polynomial at any mesoscopic scale larger than $(\log N)^C/N$. The main result is Theorem~\ref{thm:GaussFluct} below.  It is used in Section~\ref{sec:maxgauss}.

Compared to \cite{Joh1998}, our work presents two novelties. First, \cite{Joh1998} considered macroscopic scales, while we prove a result for all mesoscopic scales. Second,  \cite{Joh1998} considers the Laplace transform, while we give asymptotics of the mixed Fourier--Laplace transform.  We note that for the proof on the leading order of the maximum in  Section~\ref{sec:maxgauss}, only the Laplace transform is needed,  but the Fourier transform may be of independent interest (for example for future finer  estimates).

We face the following difficulties in proving these generalizations.
\begin{enumerate}
\item For the Laplace transform, to prove rigidity of the measures (\ref{eq:beta_ensembles_density}) perturbed on mesoscopic scales,  we 
need precise a priori bounds.  Our main tool is the local law with Gaussian tail  from \cite{BouModPai2020},  as stated in Theorem~\ref{th:local_law_bulk}.
\item For the Fourier transform, the loop equation method requires handling complex measures and the partition function may vanish. Despite this difficulty, asymptotics of characteristic functions were obtained in \cite[Appendix A]{BouErdYauYin2016}. We follow the argument developed there,  which proceeds through a  Gronwall lemma.
\end{enumerate}

\subsection{Preliminary facts and notations.}\ 
We consider the  probability density (\ref{eq:beta_ensembles_density}), with $V$ satisfying the assumptions of Section \ref{subsec:Results}, i.e. (A1), (A2) (i), (A3) and (A4) (see Subsection \ref{s:otherpotentials} regarding the Assumption (A2) (ii)). 
In this section we abbreviate the corresponding probability measure by $\mu=\mu_N$.  We recall that the equilibrium density \eqref{eqn:r} is assumed to lie on a single interval $[A,B]$  and defines a function $r(E)$.
We will also need the notations 
\begin{equation}\label{e:taunotations}
\tau(s)=\sqrt{(s-A)(B-s)},\ \ 
	b(z) =\sqrt{z-A}\sqrt{z-B},
\end{equation}
where we use the principal branch of the square root, extended to negative real numbers by
$\sqrt{-x}=\ii\sqrt{x}$ for $x>0$.
We will use the formula
\begin{equation}\label{eqn:linear}
\int_A^B \frac{\tau(s)}{s-z} \diff s = \pi \left(\frac{A+B}{2} - z + b(z)\right),
\end{equation}
which is just the usual formula for the Stieltjes transform of the semicircle law from \eqref{e:msc}, up to an affine change of variables. Then the Stieltjes transform $m_V$ from (\ref{e:szdef}) satisfies the equation 
\begin{equation}\label{eqn:2mplusV}
2 m_V(z) + V'(z) = 2r(z)b(z)
\end{equation}
for any $z\not\in[A,B]$, where we recall that $r$ from (\ref{eqn:r}) is assumed to admit an analytic extension to $\mathbb{C}$.

Given a function  $g\colon\RR\to\RR$, we consider the linear statistics
\[
S_N(g):=\sum_{k=1}^N g(\lambda_k)-N\int g\, \rd\mu_V.
\]
The functions $g$ considered in this appendix are 
\begin{equation}\label{eqn:logz2}
\rlog_z(\lambda)={\rm Re}\log(z-\lambda),\ \  
\ilog_z(\lambda)={\rm Im}\log(z-\lambda),
\end{equation}
where $\Re\log$ and $\Im\log$ are defined in (\ref{eqn:logz}). The limiting covariance for these test functions will be written in terms of 
\[v(z)=\frac{1}{2}\bigg(\frac{A+B}{2}-z+b(z)\bigg),\ \ 
\gamma=\frac{(A-B)^2}{16},\ \ 
c(z,w)=\log\bigg(1-\frac{v(z)v(w)}{\gamma}\bigg),\]
where $\log$ is the usual complex logarithm.  

The mesoscopic central limit theorem proved in this section will hold on any scale greater than a 
parameter 
\begin{equation}\label{eqn:eta0}
\eta_0\in\Big[\frac{(\log N)^{1000}}{N},\tilde \eta\Big],
\end{equation}
where $\tilde \eta$  is given by Theorem \ref{th:local_law_bulk}.

\subsection{Mixed Fourier-Laplace transform.}\ 
Given  points $\bz=(z_i)_{i=1}^p$ in $\mathbb{C}\setminus \mathbb{R}$, we define the following quadratic form in complex vectors $\bzeta=(\zeta_i)_{i=1}^p$, $\boxi=(\xi_j)_{j=1}^p$, which will represent an asymptotic covariance:
\begin{multline}\label{eqn:sigma}
\sigma(\bzeta,\boxi,\bz)=-\frac{1}{2\beta}\sum_{i,j=1}^p\big[(\zeta_j-\ii\xi_j)(\zeta_i-\ii\xi_i)c(z_j,z_i)+(\zeta_j-\ii\xi_j)(\zeta_i+\ii\xi_i)c(z_j,\bar z_i)\\
+(\zeta_j+\ii\xi_j)(\zeta_i-\ii\xi_i)c(\bar z_j,z_i)
+(\zeta_j+\ii\xi_j)(\zeta_i+\ii\xi_i)c(\bar z_j,\bar z_i)\big].
\end{multline}
We also define the following function, which will represent an asymptotic shift:
\begin{multline}\label{eqn:shift}
\mu(\bzeta,\boxi,\bz)=\sum_{j=1}^p(\zeta_j-\ii\xi_j)\int_{z_j}^{z_j+\ii\infty} 
	 \pa{\frac{1}{4}-\frac{1}{2\beta}} 
	\left( (b'(z)-1) + \int_A^B \frac{r'(s) \tau(s)}{r(s)(s-z)}\frac{ \diff s}{\pi} \right)\frac{\rd z}{b(z)}\\
-\sum_{j=1}^p(\zeta_j+\ii\xi_j) \int_{z_j}^{z_j+\ii\infty}
	 \pa{\frac{1}{4}-\frac{1}{2\beta}} 
	\left(  (b'(\bar z)-1) + \int_A^B \frac{r'(s) \tau(s)}{r(s)(s-\bar z)}\frac{\diff s}{\pi}\right) \frac{\rd z}{b(\bar z)}.
\end{multline}
Note that $\mu$ depends on the external potential $V$ through $r$, while $\sigma$ is independent of $V$.

\begin{theorem}\label{thm:GaussFluct} With the notation (\ref{eqn:logz2}), let
\begin{equation}\label{eqn:fchoice}
h=\sum_{i=1}^p(\zeta_i \rlog_{z_i}+\xi_i \ilog_{z_i}),
\end{equation}
where $p\geq 1$ is fixed.  Let $\kappa,M>0$.
Then, uniformly in $\Re(\bzeta,\boxi)\in[-M,M]^{2p}$,  $\Im(\bzeta,\boxi)\in\sqrt{\beta}\cdot[-\frac{1}{10p},\frac{1}{10p}]^{2p}$, and $\bz\in([A+\kappa,B-\kappa]\times [\eta_0,\tilde \eta])^p$,  we have
\[
\E_{\mu}\left[e^{S_N(h)}\right]
=
e^{\frac{\sigma(\bzeta,\boxi,\bz)}{2}+\mu(\bzeta,\boxi,\bz)}\cdot\left(1+\OO_{\kappa,M,p}\left(\frac{1}{\sqrt{N\eta_0}}\right)\right).
\]
\end{theorem}

We now state an elementary lemma about the size of the variance and shift terms occurring in the above central limit theorem.

\begin{lemma}\label{lem:varshift} Fix $M,\kappa>0$ and  $p\in \mathbb{N}$.  Then
uniformly in $|\bzeta|,|\boxi|\in[-M,M]^{p}$,  $\bz\in([A+\kappa,B-\kappa]\times [0,1])^p$ we have
$
\mu(\bzeta,\boxi,\bz)=\OO(1).
$

Moreover,  uniformly in $z,w\in[A+\kappa,B-\kappa]\times [0,1]$, we have $c(z,w)=\OO_\kappa(1)$, while for 
$z\in[A+\kappa,B-\kappa]\times [0,1]$, and $w\in[A+\kappa,B-\kappa]\times [-1,0]$, we have
\[
c(z,w)=\log(z-w)+f(z,w),
\]
where $f$ is a continuous function that satisfies $f(z,w) = \OO_\kappa(1)$ uniformly for such $z,w$. 
\end{lemma}

\begin{proof}
We start with the asymptotics for (\ref{eqn:shift}), which is a linear combination of terms of type
\begin{align}
&\int_{w}^{w+\ii\infty}\frac{b'(z)-1}{b(z)}\rd z,\label{eqn:1stshift}\\
&\int_{w}^{w+\ii\infty}\bigg(\int_{A}^B\frac{r'(s)}{r(s)(s-z)}\bigg)\cdot \frac{1}{b(z)}\rd z\label{eqn:2ndshift},
\end{align}
with bounded coefficients.
Since $\Re w\in[A+\kappa,B-\kappa]$, we have 
$
\int_{w}^{w+\ii}\frac{b'(z)-1}{b(z)}\rd z=\OO(1).
$
Moreover $b'(z)-1=\OO(1/|z|)$ and $b(z)\sim|z|$ as $\im z\to\infty$,  so  
$
\int_{w+\ii}^{w+\ii\infty}\frac{b'(z)-1}{b(z)}\rd z=\OO(1)$.
We have proved that (\ref{eqn:1stshift}) is $\OO(1)$ as expected.

For the term (\ref{eqn:2ndshift}),  from our non-vanishing assumption on $r$ it is of order at most
\[
\int_{w}^{w+\ii\infty}\Big(\int_{A}^B\frac{1}{|s-z|}\Big)\cdot \frac{1}{|b(z)|}\rd |z|\le 
 C \int_{w}^{w+\ii}\frac{\log\eta_z}{|b(z)|}\rd |z|+C \int_{w_j+\ii}^{w_j+\ii\infty}\frac{1}{|z|\cdot|b(z)|}\rd |z|.
\]
The first term is $\OO(1)$,  since we have $|b(z)|\ge c$ for some $c>0$ by $\Re w_j\in[A+\kappa,B-\kappa]$.
The second term is $\OO(1)$ because of the quadratic decay at $\ii\infty$.

For the asymptotics of $c(z,w)$,  to simplify notations we can assume without loss of generality that $A=-2$  and $B=2$. 
Note that $v(z)=\frac{-z+\sqrt{z^2-4}}{2}$ conformally maps $\mathbb{C}\cup\{\infty\}\setminus [-2,2]$ into $\mathbb{D}$, so  $c(z,w)$ is well-defined.
Moreover,  the image of  $[2+\kappa,2-\kappa]\times [0,1]$ by $v$ is a subset of $\mathbb{H}\cap\mathbb{D}$
which has positive distance to $-1$ and $1$. Then for any $z,w\in[2+\kappa,2-\kappa]\times [0,1]$ we have
$|v(z)v(w)|<1-\e$ for some fixed $\e>0$, and hence $c(z,w)=\OO(1)$.

For the case $\Im z>0,\Im w<0$,
we denote $v(E^+)=\lim_{\eta\to 0^+}v(E+\ii\eta)=\frac{-E+\sqrt{4-E^2}\ii}{2}$ and similarly
$v(E^-)=\frac{-E-\sqrt{4-E^2}\ii}{2}$.
We have $v'(z)=-\frac{v(z)}{b(z)}=-\frac{1}{2}+\frac{z}{2\sqrt{z^2-4}}$.  We therefore denote $v'(E^+)=-\frac{1}{2}-\frac{E}{2\sqrt{4-E^2}}\ii$ and $v'(E^-)=-\frac{1}{2}+\frac{E}{2\sqrt{4-E^2}}\ii$.

Note that $v(E^+)v(E^-)=1$ and $v(E^+)v'(E^-)=-v(E_-)v'(E^+)=\frac{1}{b(E_+)}=-\frac{\ii}{\sqrt{4-E^2}}$. This implies, for $\im \e_1>0$ and $\Im \e_2<0$, that $v(E+\e_1)v(E+\e_2)$ is equal to
\[
\big(v(E^+)+v'(E^+)\e_1+\OO(\e_1)\big)\cdot\big(v(E^-)+v'(E^-)\e_2+\OO(\e_2)\big)
=1+\frac{\ii}{\sqrt{4-E^2}}(\e_1-\e_2)+\OO(|\e_1|^2+|\e_2|^2),
\]
which  implies $\log(1-v(z)v(w))=\log(z-w)+\OO(1)$  for 
$z\in[A+\kappa,B-\kappa]\times [0,1]$, and $w\in[A+\kappa,B-\kappa]\times [-1,0]$.
\end{proof}

\subsection{Rigidity under biased measures.} \label{subsec:rgidity}\ 
We start with some important preliminary bound on the Laplace transform, relying on \cite{BouModPai2020}.

\begin{lemma}\label{cor:momentsBound}
For any fixed $\kappa,\beta>0$, there exist $N_0(V,\kappa, \beta) \in \mathbb{N}$ and $C(V,\kappa,\beta) >0$ such that the following holds. Let $\tilde \eta$ be given by Theorem \ref{th:local_law_bulk}, and fix any $M >0$. For any $z\in \mathbb{C}$ such that $\re z\in[A+\kappa,B-\kappa]$ and $\im z\in [N^{-1},\tilde\eta]$, and any $N>N_0$ and  $\zeta\in[-M,M]$, 
\begin{equation}\label{e:b3}
\log \E_{\mu}\left[e^{\zeta\left(\sum_{k=1}^N 
\rlog_z(\lambda_k)-N\int \rlog_z\rd\rho_V\right)}\right]\in [- C M (\log N)^2,C M^2(\log N)^2].
\end{equation}
Moreover, the same estimate holds when considering $\ilog$ instead of $\rlog$.
\end{lemma}

\begin{proof}
For the lower bound, by Jensen's inequality and the hypothesis $\zeta\in[-M,M]$ we have 
\begin{multline}
\log \E_{\mu}\left[e^{\zeta\left(\sum_{k=1}^N \rlog_z(\lambda_k)-N\int \rlog_z\rd\rho_V\right)}\right]\geq \zeta\, \E_{\mu}\left[\sum_{k=1}^N \rlog_z(\lambda_k)-N\int \rlog_z\rd\rho_V\right]\\
\geq 
-M \sum_{k=1}^N\E_{\mu}\Big[ \big|\rlog_z(\lambda_k)-\rlog_z(\gamma_k)\big|\Big]-M \left|N\int \rlog_z\rd\rho_V-\sum_{k=1}^N \rlog_z(\gamma_k)\right|.\label{eqn:lower}
\end{multline}
The first summand above is easily bounded on the good set
\[\mathcal G=\bigcap_{\hat k>(\log N)^{2}}\{|\lambda_k-\gamma_k|\leq  C (\log N) N^{-2/3} \hat k^{-1/3}\}\bigcap_{\hat k\leq (\log N)^{2}}\{|\lambda_k-\gamma_k|\leq  C (\log N)^{10} N^{-2/3} \}\]
as follows. For $\hat k \le (\log N)^2$, we use that $\lambda_k, \gamma_k$ are near $A$ or $B$, where $\rlog_z$ is $C_1$-Lipschitz continuous for some constant $C_1 >0$ by the assumptions on $\Re z$ and $\Im z$. For the other values of $\hat k$, we use the mean value theorem. Then setting $E = \re z$ and $\eta = \im z$, and recalling that $E$ is in the bulk of the spectrum, we have 
\begin{align}
&\sum_{k=1}^N\E_{\mu}\Big[ \big|\rlog_z(\lambda_k)-\rlog_z(\gamma_k)\big|\mathds{1}_{\mathcal G}\Big]\notag\\ 
&\le C_1
\sum_{\hat k\leq (\log N)^2}|\lambda_k-\gamma_k|\mathds{1}_{\mathcal G}+
\sum_{\hat k\geq (\log N)^2}|\lambda_k-\gamma_k|\max_{|\lambda-\gamma_k|\leq C(\log N)N^{-2/3}(\hat k)^{-1/3}}|\rlog_z'(\lambda)|\mathds{1}_{\mathcal G}\notag\\
&\leq CN^{-1/2}+\sum_{|E-\gamma_k|>(\log N)/N} \frac{C(\log N)}{|E-\gamma_k|}N^{-\frac{2}{3}}(\hat k)^{-\frac{1}{3}}
+
\sum_{|E-\gamma_k|<(\log N)/N} \frac{C(\log N)}{\eta}N^{-1}\leq C(\log N)^2. \label{eqn:bound1}
\end{align}
Moreover,   for large enough $N$ we have 
\[\mathbb{P}(\mathcal G^{\rm c})\leq N^{-100}\]
from  \cite[Corollary 1.5]{BouModPai2020} for $\hat k>(\log N)^{2}$,  and \cite[Corollary 1.5 and Corollary 1.6]{BouErdYau2014} for $\hat k\leq (\log N)^{2}$.  This implies
\[\sum_{k=1}^N\E_{\mu}\Big[ \big|\rlog_z(\lambda_k)-\rlog_z(\gamma_k)\big|\mathds{1}_{\mathcal G^{\rm c}}\Big] \le N^{-10}\E[|\rlog_z(\lambda_1)|]^{1/2}\leq N^{-2}
\]
where we used that the latter expectation is finite,  from the estimate
$\mathbb{P}(|\lambda_1|>x)\leq (x-L)^N$ for some fixed $L$
\cite[Equation (3.3)]{de1995statistical}.

For the second, deterministic, term in (\ref{eqn:lower}), we have  by the intermediate value theorem  that $N\int \rlog_z\rd\mu_V-\sum \rlog_z(\gamma_k)=\sum (\rlog_z(\delta_k)- \rlog_z(\gamma_k))$ for some  $\delta_k\in[\gamma_{k-1},\gamma_{k+1}]$. Then the same reasoning as \eqref{eqn:bound1} applies, giving an analogous bound. This completes the proof of the lower bound in \eqref{e:b3}. 

For the upper bound, let $\eta'$ be a parameter to be fixed later. For all $x\in[\eta,\eta']$, we define 
\[g(x)=\sum_{k=1}^N\rlog_{E+\ii x}(\lambda_i)-N\int\rlog_{E+\ii x}\rd\rho_V, \qquad \delta(x)=N\big|s(E+\ii x)-m_V(E+\ii x)\big|.\]
From the local law,  Theorem \ref{th:local_law_bulk},  we have for $\eta'\in[\eta,\tilde\eta]$  that
\begin{align}
\E\left[e^{\xi\left(g(\eta)-g(\eta')\right)}\right]&\leq 2\sum_{p\geq 0}\frac{|\xi|^{2p}}{(2p)!}\E\left[\left(\int_{\eta}^{\eta'}\delta(x)\right)^{2p}\right]\notag \\
&\leq 2\sum_{p\geq 0}\frac{|\xi|^{2p}}{(2p)!}
\int_{[\eta,\eta']^{2p}}\E[ \delta(\eta_1)\dots\delta(\eta_{2p})] \rd\eta_1\dots\rd\eta_{2p}\notag \\
&\leq
2\sum_{p\geq 0}\frac{|\xi|^{2p}}{(2p)!}\,\int_{[\eta,\eta']^{2p}}(\E(\delta(\eta_1)^{2p})^{\frac{1}{2p}}\dots(\E(\delta(\eta_{2p})^{2p})^{\frac{1}{2p}})\rd\eta_1\dots\rd\eta_{2p}\notag \\
&\leq 
2\sum_{p\geq 0}\frac{|\xi|^{2p}}{(2p)!}\left(\int_{[\eta,\eta']}\left(\frac{(Cp)^p}{x^{2p}}+ (CN)^{2p} e^{-cN}\right)^{\frac{1}{2p}}\rd x\right)^{2p}\notag \\
&\leq 
\sum_{p\geq 0}\frac{|2\xi|^{2p}}{(2p)!}\left(\int_{[\eta,\eta']}\frac{(Cp)^{1/2}}{x}\rd x\right)^{2p}+\sum_{p\geq 0}\frac{|2\xi|^{2p}}{(2p)!}\left(\int_{[\eta,\eta']}CN e^{-c\frac{N}{2p}}\rd x\right)^{2p}\label{ubd3},
\end{align}
where  we successively used H\"older's inequality, the inequality $(a+b)^c\leq a^c+b^c$ for $a,b>0$,  $c\in[0,1]$ (with $c=1/(2p)$ above), and $(x+y)^{2p}\leq 2^{2p-1}(x^{2p}+y^{2p})$.
The first sum above is of order
\begin{equation}\label{upperbd1}
\sum_{p\geq 0}\frac{|C\xi|^{2p}}{p^p}(\log N)^{2p}\leq \exp(C|\xi|^2(\log N)^2),
\end{equation}
while the second one is bounded by
\begin{equation}\label{upperbd2}
\sum_{p\geq 0}\frac{|C\xi|^{2p}}{(2p)!}({\eta'}N)^{2p}e^{-cN}\leq e^{C|\xi|{\eta'} N-cN},
\end{equation}
where on the right-hand sides of \eqref{upperbd1} and \eqref{upperbd2}, $C,c>0$ are constants that depend on $V$ but not on $\eta'$. We now fix $\eta'$ small enough, as a function of $M$, so that these upper bound in \eqref{upperbd2} is $o(1)$ when $|\xi|<2M$.  In conclusion, we have shown (recalling \eqref{ubd3}) that there exist constants $C,\eta'>0$  such that
for any $|\xi|<2M$ we have
\begin{equation}\label{eqn:increment}
\E\left[e^{\xi(g(\eta)-g(\eta'))}\right]\leq \exp(C M^2(\log N)^2).
\end{equation}
Moreover,  as $\xi g(\eta')$ is a smooth linear statistic on a macroscopic scale,   uniformly in $|\xi|<2M$ we have (see e.g.\  \cite[Theorem 1(i)]{Shc2013})
\begin{equation}\label{eqn:scale1}
\E\left[e^{\xi g(\eta')}\right]=\OO(1),
\end{equation}
where the implicit constant depends only on the choice of $g$ and $V$. 
Equations (\ref{eqn:increment}) and (\ref{eqn:scale1}) conclude the proof of the upper bound when $\eta<\eta'$, after combining them with the Cauchy--Schwarz inequality to estimate the left side of \eqref{e:b3}. For $\eta\in[\eta',{\tilde \eta}]$, we can directly use 
$\E\left[e^{\xi g(\eta)}\right]=\OO(1)$, as explained before the previous equation.
\end{proof}

For fixed $\bzeta,\boxi$, and a function $h\colon \R \rightarrow \R$, we will need the following (complex) measures. They are modifications of the measure \eqref{eq:beta_ensembles_density} and depend on the parameter $0\leq t\leq 1$:
\[
\rd\mu_h^{t}(\bla):=
\frac{e^{t S_N(h)}}{Z_h(t)}\rd\mu(\bla),
\]
where we assumed 
\begin{equation}
  \label{eq-of1}
  Z_h(t):=\E_{\mu}[e^{ t S_N(h)}] \neq 0.
\end{equation}
In the next section we will use rigidity under biased measures under the following form.

\begin{lemma}\label{lem:rigandcons}
For any $t$ such that $Z_h(t)\neq 0$, and any function $f\colon \R \rightarrow \R$ and measurable set $\mathscr{G}$,
\begin{equation}\label{eqn:trivial}
|\E_{\mu_h^t}[f\mathds{1}_{\mathscr{G}}]|\leq \frac{Z_{\Re h}(t)}{|Z_h(t)|}\sup_{\mathscr{G}} |f|.
\end{equation}
Moreover, any integer $1\le k \le N$, we define $\mathscr{G}=\mathscr{G}_k=\{|\la_k-\gamma_k|<N^{-\frac{2}{3}}(\widehat k)^{-\frac{1}{3}}(\log N)^{100}\}$. Then for
 all $M,p>1$ there exists $c(M,p)>0$ such that for any  $(\bzeta,\boxi) \in \C^{2p}$ such that $\Re(\bzeta,\boxi)\in[-M,M]^{2p}$, $t\in[0,1]$ such that $Z_h(t)\neq 0$, and $N\geq 1$, we have
\begin{equation}\label{eqn:nontrivial}
|\E_{\mu_h^t}[f\mathds{1}_{\mathscr{G}_k^{\rm c}}]|\leq \frac{\E_\mu[|f|^2]^{1/2}}{|Z_h(t)|}\cdot e^{-(\log N)^{4}}.
\end{equation}
\end{lemma}

\begin{proof} The first statement follows directly from 
\[
|\E_{\mu_h^t}[f\mathds{1}_{\mathscr{G}}]|\leq \frac{\int e^{t S_N(\Re h)}|f|\mathds{1}_{\mathscr{G}}\rd\mu}{|Z_h(t)|}.
\]
For the second statement, we use the Cauchy--Schwarz inequality twice,  writing
\[
|\E_{\mu_h^t}[f\mathds{1}_{\mathscr{G}^{\rm c}}]|\leq
 \frac{\E_{\mu}[|f|\mathds{1}_{\mathscr{G}^{\rm c}} e^{t S_N(\Re h)}]}{|Z_h(t)|}\leq  \frac{\E_{\mu}[|f|^2]^{1/2}}{|Z_h(t)|}
 \mathbb{P}_\mu(\mathscr{G}^{\rm c})^{1/4}\E_{\mu}[e^{4t S_N(\Re h)}]^{1/4}.\]
One concludes with lemmas \ref{cor:momentsBound} and \ref{lem:rig}.
\end{proof}

\subsection{Analysis of the loop equation.}\ 
To prove Theorem \ref{thm:GaussFluct} we start with the identity
$
\frac{\rd}{\rd t}\log Z(t)=\E_{\mu_h^{t}}\left( S_N(h)\right)$, and therefore want to estimate expectation of general linear statistics for the measure $\mu_h^t$. 
The starting point for the proof of Theorem \ref{thm:GaussFluct}  will be the first order loop equation (\ref{eq:loop_equation_mu_theta}) below.

Let $\rho_1^{(N,t)}(s)$ be the 1-point correlation function for the measure $\mu_h^t$ (see e.g.\ \cite[Definition 2.4]{erdos2017dynamical}), and let $m_{N,t}(z)$ be the Stieljes transform of $\rho_1^{(N,t)}(s)$. 
We introduce 
\begin{equation}\label{eqn:phitheta}
	\varphi(z) = \varphi_{N,t}(z) 
	\coloneqq m_{N,t}(z) - m(z).
\end{equation}

For $z \in \mathbb{C} \setminus \R$, we define 
\begin{align} \label{eq:def_psiwhy}
	\widetilde{\varphi}(z)  
	\coloneqq \frac{1}{2\pi b(z)} \left(
	\frac{2t}{\beta} \int_A^B \frac{h'(s)}{s-z} \tau(s) \diff s
	- \pa{\frac{2}{\beta}-1} 
	\left( \pi (b'(z)-1) + \int_A^B \frac{r'(s) \tau(s)}{r(s)(s-z)} \diff s \right)
	\right).
\end{align}
The main result of this subsection is the following one.

\begin{lemma} \label{lem:estimate_varphi} For any $\kappa,M>0$ and $p\geq 1$ the following holds.
Uniformly in any $z = E + \ii \eta$, with $\eta_0\leq \eta\leq N^{10}$  (see (\ref{eqn:eta0})) and $A+\kappa\leq E \leq B-\kappa$, $(\bzeta,\boxi) \in \C^{2p}$ such that $\Re(\bzeta,\boxi)\in[-M,M]^{2p}$, 
and $t$ such that $ |Z_h(t)|\geq e^{-(\log N)^2}$, we have 
\begin{align}\label{eqn:phi}
	\varphi(z) 
	& =\frac{\widetilde\varphi(z)}{N}
	+\OO_{\kappa,M,p} \left(\frac{(\|\boxi\|_\infty+\|\bzeta\|_\infty)(\log N)^{500}}{N^2\eta_0\eta}\cdot \frac{Z_{\Re h}(t)^2}{|Z_h(t)|^2}\right).
\end{align}
\end{lemma}

We now prove two lemmas which will be used as input to the proof of Lemma \ref{lem:estimate_varphi}.
\begin{lemma} \label{lem:first_estimates_1}
For any $\kappa,M>0$,   $p\geq 1$,   we have uniformly in $\Re(\bzeta,\boxi)\in[-M,M]^{2p}$ and $z = E + \ii \eta$ such that $\eta \neq 0$, $A+\kappa\leq E \leq B-\kappa$, that
\begin{align*}
	\varphi(z) = \OO_{\kappa,M,p} \left( \frac{(\log N)^{200}}{N\eta}\cdot\frac{Z_{\Re h}(t)}{|Z_h(t)|}  \right), 
	\qquad 
	\Var_{\mu_h^t} \left( s_N(z) \right) 
	= \OO_{\kappa,M,p} \left( \frac{(\log N)^{400}}{ (N\eta)^2}\cdot\frac{Z_{\Re h}(t)^2}{|Z_h(t)|^2} \right).
\end{align*}
\end{lemma}
\begin{proof} 
These bounds are immediate consequences of Lemma \ref{lem:rigandcons}.
Indeed, rewriting \eqref{eqn:phitheta} gives
\[
\phi(z)=\E_{\mu_h^t}\left[\frac{1}{N}\sum_k \left(\frac{1}{\lambda_k-z}-\frac{1}{\gamma_k-z}\right)\right]+
\frac{1}{N}\sum_k\frac{1}{\gamma_k-z}-\int\frac{\rd\rho_V(\lambda)}{\lambda-z}.
\]
The deterministic sum in the second term is easily seen to be $\OO((N\eta)^{-1})$ by definition of the $\gamma_k$'s. Denoting 
$\mathscr{G}_k=\{|\la_k-\gamma_k|< N^{-\frac{2}{3}}(\widehat k)^{-\frac{1}{3}}(\log N)^{100}\}$,
 the expectation in the first term is bounded by
\begin{multline*}
\E_{\mu_h^t}\left[\frac{1}{N}\sum_k\Big|\frac{1}{\lambda_k-z}-\frac{1}{\gamma_k-z}\Big|\mathds{1}_{\cap_k\mathscr{G}_k}\right]+
\sum_k\E_{\mu_h^t}\left[\frac{1}{N}\sum_k\Big|\frac{1}{\lambda_k-z}-\frac{1}{\gamma_k-z}\Big|\mathds{1}_{\mathscr{G}_k^{\rm c}}\right]\\
\leq\frac{(\log N)^{200}}{N\eta}\cdot\frac{Z_{\Re h}(t)}{|Z_h(t)|}  +\frac{Ne^{-(\log N)^{4}}}{\eta|Z_h(t)|}, 
\end{multline*}
where we have used \Cref{lem:rig},   (\ref{eqn:trivial}) and  (\ref{eqn:nontrivial}).
The second term is negligible because $Z_{\Re h}(t)\geq \exp(-CM(\log N)^2)$  from Lemma \ref{cor:momentsBound}. This concludes our estimate on $\varphi$, after using the elementary estimate $|Z_h(t)| \le Z_{\Re h}(t)$. 

The bound on $\Var_{\mu_h^t}(s_N(z))$ proceeds similarly,  starting with
\[
|\Var_{\mu_h^t} \left( s_N(z) \right) |\leq 2 \big|\E_{\mu_h^t}[|s_N(z)-m(z)|^2]\big|+2|\varphi(z)|^2.
\]
One applies the same reasoning as before to $\E_{\mu_h^t}[|s_N(z)-m(z)|^2]$ and obtains the bound $\frac{(\log N)^{300}}{(N\eta)^2}\cdot\frac{Z_{\Re h}(t)}{|Z_h(t)|}$, so that the bound on $\varphi(z)^2$ dominates. 
\end{proof}

\begin{lemma} \label{lem:first_estimates_2}
For any $M>0$,   $p\geq 1$,   uniformly in  $z = E + \ii \eta$ with $\eta \neq 0$, $A\leq E \leq B$,  $A>0$,  $\Re(\bzeta,\boxi)\in[-M,M]^{2p}$, and uniformly in $h\in\mathscr{C}^2(\mathbb{R})$ we have
\begin{multline}
	\int_\R \frac{h'(s)}{s-z} \left( \rho_1^{(N,t)}(s) - \rho_V(s) \right) \diff s
	 =\\  \frac{(\log N)^{200}}{N}\cdot\frac{Z_{\Re h}(t)}{|Z_h(t)|}\OO_{M,p} \left(\int\frac{|h''(s)|}{|z-s|}\rd s
	 +\int \frac{|h'(s)|}{|z-s|^2}\rd s+\frac{e^{-(\log N)^2}}{\eta^2}(\|h'\|_\infty+\|h''\|_\infty)\right) .
	\label{eqn:elementary1}
\end{multline}
\end{lemma}
\begin{proof} 
We apply Lemma \ref{lem:rigandcons}, and the proof is almost the same as Lemma 5.3 in \cite{BouErdYauYin2016}, so we omit the details. 
The only differences are that the rigidity estimate is now known with multiplicative error $(\log N)^{100}$ instead of $N^\xi$,   the
probability error is now $e^{-(\log N)^2}$ instead of $e^{-N^\xi}$.
\end{proof}

\begin{proof}[Proof of Lemma \ref{lem:estimate_varphi}] We closely follow some steps in the proof of \cite[Lemma 4.6]{BouModPai2020},
with the  difference that we now work under complex measures. 
We first define
\begin{align} \label{eq:def_psi}
	\psi(z)  
	& \coloneqq \frac{2t}{\beta N} \int_A^B \frac{h'(s)}{s-z} \rho_{V}(s) \diff s - \frac{1}{N}\pa{\frac{2}{\beta}-1} m'_{V}(z)
	- \int_\R \frac{V'(s) - V'(z)}{s-z} \pa{\rho_1^{(N,t)}(s) - \rho_V(s)} \diff s \\
	\label{eq:def_Err}
	\mathrm{Err}(z) 
	& \coloneqq \varphi(z)^2
	- \frac{2t}{\beta N} 
		\int_\R \frac{h'(s)}{s-z}\pa{ \rho_1^{(N,t)}(s) - \rho_V(s)} \diff s 
	+ \frac{1}{N} \left( \frac{2}{\beta} - 1 \right) \varphi'(z)
	+ \Var_{\mu_h^t} \left( s_N(z) \right).
\end{align}
Then,  by the same proof as \cite[Equation (4.7)]{BouModPai2020} but with complex measures, we have
\begin{equation} \label{eq:loop_equation_mu_theta}
	\big(2 m_V(z) + V'(z)\big) \varphi(z) - \psi(z) + \mathrm{Err}(z) = 0.
\end{equation}
For the proof, we also need to work under the rigidity event $\mathscr{R} \coloneqq \bigcap_{1\leq k\leq N}\{|\lambda_k-\gamma_k|<(\log N)^{100}N^{-\frac{2}{3}}(\hat k)^{-\frac{1}{3}}\}$, by introducing the new probability measure
\begin{align*}
	\diff \mu_h^{t,\mathscr{R}}(\lambda_1,\dots,\lambda_N) 
	= \frac{\1_{\mathscr{R}}}{\P_{\mu_h^t}(\mathscr{R})} 
	\diff \mu_h^t(\lambda_1,\dots,\lambda_N).
\end{align*}
Moreover, let  $\rho_1^{(N,t,\mathscr{R})}(s)$ be the 1-point function under $\mu_h^{t,\mathscr{R}}$,  
$
\varphi^\mathscr{R}(z) \coloneqq \E_{\mu_h^{t,\mathscr{R}}} [s_N(z)] - m_V(z),
$
and $\mathrm{Err}^\mathscr{R}(z)$ be defined as $\mathrm{Err}(z)$ but with $\mu_h^{t,\mathscr{R}}$, $\rho_1^{(N,t,\mathscr{R})}(s)$ and $\varphi^\mathscr{R}(z)$ instead of $\mu_h^t$, $\rho_1^{(N,t)}(s)$ and $\varphi(z)$.
Note that 
\begin{equation}\label{scrR}
\P_{\mu_h^t}(\mathscr{R})= 1 +\OO\big(N |Z_h(t)|^{-1}e^{-(\log N)^{4}}\big) =1+\OO(e^{-(\log N)^{4}/2} )
\end{equation} 
by (\ref{eqn:nontrivial}) and our assumption on $Z_h(t)$.  This easily implies that lemmas \ref{lem:first_estimates_1} and \ref{lem:first_estimates_2} still hold under $\mu_h^{t,\mathscr{R}}$, giving
\begin{align}
&	\varphi^{\mathscr{R}}(z) = \OO \left( \frac{(\log N)^{200}}{N\eta}\cdot\frac{Z_{\Re h}(t)}{|Z_h(t)|} \right), 
	\qquad \notag
	\Var_{\mu_h^{t,\mathscr{R}}} \left( s_N(z) \right) 
	= \OO \left( \frac{(\log N)^{400}}{ (N\eta)^2}\cdot\frac{Z_{\Re h}(t)^2}{|Z_h(t)|^2} \right),\\
&	\int_\R \frac{h'(s)}{s-z} \left( \rho_1^{(N,t,\mathscr{R})}(s) - \rho_V(s) \right) \diff s\notag
	 \\ &\quad = \frac{(\log N)^{200}}{N}\cdot\frac{Z_{\Re h}(t)}{|Z_h(t)|}\OO \left(\int\frac{|h''(s)|}{|z-s|}\rd s
	+\int \frac{|h'(s)|}{|z-s|^2}\rd s %
	+\frac{e^{-(\log N)^2}}{\eta^2} \big(\|h'\|_\infty+\|h''\|_\infty\big) \right).\label{eqn:rigBias}
\end{align}

Fix some $z = E + \ii \eta$ with $\eta_0 \leq\eta$ and $A-\kappa \leq E \leq B+\kappa$. We also assume $\eta<\kappa$ first.
We consider the rectangle with vertices $A-\kappa \pm \ii e^{-(\log N)^3}$, $B+\kappa \pm \ii e^{-(\log N)^3}$, and denote by $\cC$ the corresponding closed contour with positive orientation.
We decompose this contour into $\cC_{\text{hor}}$, which consists only in the horizontal pieces, and $\cC_{\text{ver}}$, which consists only in the vertical pieces.
By the loop equation \eqref{eq:loop_equation_mu_theta} and (\ref{eqn:2mplusV}), we have
\begin{align*} 
	\int_{\cC_{\text{hor}}} \frac{2r(w)b(w) \varphi(w) - \psi(w) + \mathrm{Err}(w)}{r(w)(z-w)} \diff w
	= 0.
\end{align*}
Using \eqref{scrR}, the hypothesis that $|Z_h(t)|>e^{-(\log N)^2}$, and \eqref{eqn:nontrivial},
on $\cC_{\text{hor}}$
we have 
\begin{equation}\label{eqn:tophiR}
\varphi^\mathscr{R}(w)
	=  \varphi(w) + \OO \left(e^{-(\log N)^4/5} \right),
\end{equation}
and similarly $\mathrm{Err}^\mathscr{R}(w)=\mathrm{Err}(w)+\OO \left(e^{-(\log N)^4/5} \right)$.
Together with the facts that $|z-w|>\eta/2$ and $c\le r,b\le C$ on  $\cC_{\text{hor}}$ (remember that $r$ is continuous and has no zero on $[A,B]$) this implies
\begin{align} \label{eq:C_hor}
	\int_{\cC_{\text{hor}}} \frac{2r(w)b(w) \varphi^\mathscr{R}(w) - \psi(w) + \mathrm{Err}^\mathscr{R}(w)}{r(w)(z-w)} \diff w
	= \OO \left(e^{-(\log N)^4/10}\right).
\end{align}
On the other hand, for $w$ on $\cC_{\text{ver}}$, we have $2r(w)b(w) \varphi^\mathscr{R}(w) - \psi(w) + \mathrm{Err}^\mathscr{R}(w) = \OO(e^{(\log N)^{5/2}})$,  an easy estimate based on the following facts: (1) by the definition of, $\mu_h^{t,\mathscr{R}}$, all particles are at a distance larger than $\kappa/2$ from $\cC_{\text{ver}}$),  (2) $c \le r,b \le C$ on $\cC_{\text{ver}}$,  (3)  $|Z_h(t)|>e^{-(\log N)^2}$ by assumption,  (4) $Z_{\rm {Re} h}(t)\leq e^{CM(\log N)^2}$ by Lemma \ref{cor:momentsBound}.
This implies 
\begin{align} \label{eq:C_ver}
	\int_{\cC_{\text{ver}}} \frac{2r(w)b(w) \varphi^\mathscr{R}(w) - \psi(w) + \mathrm{Err}^\mathscr{R}(w)}{r(w)(z-w)} \diff w
	= \OO \left(e^{-(\log N)^3/2} \right).
\end{align}
Combining \eqref{eq:C_hor} and \eqref{eq:C_ver}, we get
\begin{align} \label{eq:C}
	\int_\cC \frac{2r(w)b(w) \varphi^\mathscr{R}(w) - \psi(w) + \mathrm{Err}^\mathscr{R}(w)}{r(w)(z-w)} \diff w
	= \OO \left(e^{-(\log N)^3/2}\right).
\end{align}
We now estimate each term in the integral in \eqref{eq:C} successively.

We start with the part involving $\varphi^\mathscr{R}(w)$. 
The function $w \mapsto 2b(w) \varphi^\mathscr{R}(w)/(z-w)$ is analytic on and outside $\cC$, except for the pole at $z$, and it behaves as $\OO(w^{-2})$ as $\abs{w} \to \infty$ because $b(w)=\OO(w)$ and $\varphi^\mathscr{R}(w)=\OO(w^{-2})$. Therefore, by the Cauchy integral formula with residue at infinity, we get
\begin{align} \label{eq:C_varphi}
	\int_\cC \frac{2b(w)\varphi^\mathscr{R}(w)}{(z-w)} \diff w
	= 4 \ii \pi b(z) \varphi^\mathscr{R}(z).
\end{align}

Now we evaluate the part involving $\psi(w)$. Recall the definition of $\psi(w)$ in \eqref{eq:def_psi} and note that the third term there is analytic in $w \in \mathbb{C}$. 
Moreover from (\ref{eqn:2mplusV}) we have $m_V'(w) = -\frac{1}{2} V''(w) + (rb)'(w)$, where $V''(w)$ is also analytic in $w \in \mathbb{C}$.
Since $z$ is exterior to $\cC$ and $r$ has no zero inside $\cC$ for $\kappa$ chosen small enough, these analytic terms disappear  in \eqref{eq:C} and we get
\begin{align}
  \label{eq:ElAl}
	\int_\cC \frac{\psi(w)}{r(w)(z-w)} \diff w
	& = \int_\cC \left( \frac{2t}{\beta N} \int_A^B \frac{h'(s)}{s-w} \rho_{V}(s) \diff s - \frac{1}{N}\pa{\frac{2}{\beta}-1} (rb)'(w) \right) \frac{\diff w}{r(w)(z-w)} \\
	& = -\frac{4\pi t}{\beta N} \int_A^B \frac{h'(s)}{r(s)(z-s)} \rho_{V}(s) \diff s
	- \frac{1}{N}\pa{\frac{2}{\beta}-1} 
	\int_A^B \frac{-2\ii (r\tau)'(s)}{r(s)(z-s)} \diff s,
	\nonumber
\end{align}
where, for the first term, we applied Cauchy's integral formula and, for the second term, we let the contour approach the segment $[A,B]$ and used $\lim_{y \to 0+} (rb)'(x\pm \ii y) = \pm \ii(r\tau)'(x)$ for all $x \in (A,B)$, recalling $\tau(x) = \sqrt{(x-A)(B-x)}$.
Recalling the definition of $\widetilde{\varphi}(z)$ and that $\rho_V = \frac{1}{\pi} r \tau$ (recall \eqref{eqn:r} and \eqref{e:taunotations}), we get
\begin{align} \label{eq:C_psi}
	\int_\cC \frac{\psi(w) \diff w}{r(w)(z-w)}
	& = -2 \left(
		\frac{2t}{\beta N} \int_A^B \frac{h'(s)}{s-z} \tau(s) \diff s
		+ \frac{\ii}{N}\pa{\frac{2}{\beta}-1} 
		\int_A^B \left( \frac{\tau'(s)}{s-z} + \frac{r'(s) \tau(s)}{r(s)(s-z)} \right) \diff s
		\right) 
	= \frac{4 \ii \pi b(z)}{N} \widetilde{\varphi}(z),
\end{align}
where we used that $\int_A^B \frac{\tau'(s)}{s-z} \diff s = \int_A^B \frac{\tau(s)}{(s-z)^2} \diff s = \pi (b'(z)-1)$, which follows from differentiating the identity (\ref{eqn:linear}).

Finally, we deal with the part involving $\mathrm{Err}^\mathscr{R}(w)$. We deform the contour $\cC$ into $\cC'$, the positively oriented rectangle with vertices $A-\kappa \pm \ii \eta/2$, $B+\kappa \pm \ii \eta/2$, which does not change the value of the integral because the deformation does not cross any poles.
The function $w \mapsto \mathrm{Err}^\mathscr{R}(w) / (r(w)(z-w))$ is analytic on and between these contours (remember we assume $\eta<\kappa$ first, and $\kappa$ is fixed and small enough),  so
\begin{multline} \label{eq:C_Err}
	\int_\cC \frac{\mathrm{Err}^\mathscr{R}(w)}{r(w)(z-w)} \diff w
	= \int_{\cC'} \frac{\mathrm{Err}^\mathscr{R}(w)}{r(w)(z-w)} \diff w
	= \OO \left( \frac{(\log N)^{400}}{(N\eta)^2}\cdot\frac{Z_{\Re h}(t)^2}{|Z_h(t)|^2} \right)\\
	+ \OO \left( \frac{(\log N)^{200}}{N^2} 
			\sum_{\ell=1}^m \frac{1}{\eta_\ell(\eta_{\ell}\vee|z-z_\ell|)}\cdot\frac{Z_{\Re h}(t)}{|Z_h(t)|} \right),
\end{multline}
where we used that $\abs{r(w)}$ is uniformly lower bounded and we applied  (\ref{eqn:rigBias}) on the horizontal pieces of $\cC'$,  and on the vertical pieces we used that  the same estimates hold substituting  $\eta$ with $\kappa$.
Coming back to \eqref{eq:C} and combining \eqref{eq:C_varphi}, \eqref{eq:C_psi} and \eqref{eq:C_Err},  for $\eta\leq \kappa$ we have proved that 
\begin{align}
	\varphi^\mathscr{R}(z) 
	& =\frac{\widetilde\varphi(z)}{N}
	+ \OO \left(\frac{(\log N)^{400}}{N^2\eta_0\eta}\cdot\frac{Z_{\Re h}(t)^2}{|Z_h(t)|^2}\right).\label{eqn:12}
\end{align}
For $\eta\geq \kappa$, the result follows from the Cauchy formula $\frac{1}{s-z}=\frac{1}{2\pi \ii}\int_{\mathcal{C}''}\frac{\rd w}{(w-z)(w-s)}$ where $\mathcal{C}''$ is the rectangle with vertices $A-\kappa\pm\ii\kappa/10$ and $B+\kappa\pm \ii\kappa/10$. More precisely we use (\ref{eqn:12}) in the portion $|{\rm Im}z|>\eta_0$ of $\mathcal{C}''$ and the estimates from (\ref{eqn:rigBias}) with $\eta$ replaced by $1$ on the part $|{\rm Im}z|\leq\eta_0$.

Finally,  the result follows from (\ref{eqn:12}) and (\ref{eqn:tophiR}).
\end{proof}

\subsection{Proof of Theorem \ref{thm:GaussFluct}.}\ 
Let $Z_j=z_j+\ii N^{9}$.  With $Z_h$ as in \eqref{eq-of1},
we can write
\begin{multline}
  \frac{\rd}{\rd t}\log Z_h(t)=\E^{\mu_h^{t}} \left[\sum_{i=1}^p(\zeta_i \rlog_{z_i}+\xi_i \ilog_{z_i}) \right]\\
=N\sum_{j=1}^p\frac{\zeta_j-\ii\xi_j}{2}\int_{z_j}^{Z_j} \varphi(z)\rd z
-N\sum_{j=1}^p\frac{\zeta_j+\ii\xi_j}{2} \int_{z_j}^{Z_j}\varphi(\bar z)\rd z
+
\E^{\mu_h^{t}}\left[\sum_{j=1}^p(\zeta_j \rlog_{Z_j}+\xi_j \ilog_{Z_j})\right],\label{eqn:differential}
\end{multline}
where we have used $\rd z=\ii\, \rd y$ and \[
\frac{\rd}{\rd y}\log\big|(x+\ii y)-\lambda\big|=-\frac{1}{2\ii}\left(\frac{1}{\lambda-(x+\ii y)}-\frac{1}{\lambda-(x-\ii y)}\right), \] \[ \frac{\rd}{\rd y}\im\log\big((x+\ii y)-\lambda\big)=-\frac{1}{2}\left(\frac{1}{\lambda-(x+\ii y)}+\frac{1}{\lambda-(x-\ii y)}\right).\]

To bound the last term in \eqref{eqn:differential},  we note that the inequality  $\log(1+\e)=\OO(\e)$ for $|\e|<1/2$ gives
\begin{multline*}
\E^{\mu_h^t} \left[\sum_i\log(Z_j-\lambda_i)-N\int\log(Z_j-\lambda)\rd\rho_V(\lambda)\right]\\ =\E^{\mu_h^t} \left[\sum_i\log(1-\lambda_i/Z_j) \right]-N\int\log(1-\lambda/Z_j)\rd\rho_V(\lambda)=\OO(N^{-7}).
\end{multline*}
We now insert the asymptotics from Lemma \ref{lem:estimate_varphi} in (\ref{eqn:differential}). The error term in (\ref{eqn:phi}) contributes 
\[
\frac{Z_{\Re h}(t)^2}{|Z_h(t)|^2}\sum_{j=1}^p\frac{|\zeta_j|+|\xi_j|}{2}\int_{z_j}^{Z_j}\frac{(\|\boxi\|_\infty+\|\bzeta\|_\infty)(\log N)^{500}}{N\eta_0\eta_z} |\rd z|\leq C
\frac{Z_{\Re h}(t)^2}{|Z_h(t)|^2}\frac{(\|\boxi\|_\infty^2+\|\bzeta\|_\infty^2)(\log N)^{600}}{N\eta_0}.
\]
Moreover,  one easily checks that $\int_{Z_j}^{Z_j+\ii\infty}\tilde\phi(z)\, \rd z=\OO(N^{-6})$, so that 
denoting \[p(z)=\frac{1}{b(z)} \pa{\frac{1}{4}-\frac{1}{2\beta}} 
	\left( (b'(z)-1) + \int_A^B \frac{r'(s) \tau(s)}{ r(s)(s-z)} \frac{\diff s}{\pi} 
	\right)\]
we have proved that, as long as $ |Z_h(t)|\geq e^{-(\log N)^2}$ (necessary to apply Lemma \ref{lem:estimate_varphi}) we have
\begin{align}
  \frac{\rd}{\rd t}\log Z_h(t)=&\sum_{j=1}^p\frac{\zeta_j-\ii\xi_j}{2}\int_{z_j}^{z_j+\ii\infty} \frac{t}{\pi\beta b(z)}
	\int_A^B \frac{h'(s)}{s-z} \tau(s) \diff s\rd z\notag\\
	& \notag
-\sum_{j=1}^p\frac{\zeta_j+\ii\xi_j}{2} \int_{z_j}^{z_j+\ii\infty}\frac{t}{\pi\beta b(\bar z)}
	 \int_A^B \frac{h'(s)}{s-\bar z} \tau(s) \diff s\rd z\\
	 &+\sum_{j=1}^p(\zeta_j-\ii\xi_j)\int_{z_j}^{z_j+\ii\infty} p(z)\rd z\notag
-\sum_{j=1}^p(\zeta_j+\ii\xi_j) \int_{z_j}^{z_j+\ii\infty}p(\bar z)\rd z\\
&+\OO\left(\frac{Z_{\Re h}(t)^2}{|Z_h(t)|^2}\frac{(\|\boxi\|_\infty^2+\|\bzeta\|_\infty^2)(\log N)^{600}}{N\eta_0}\right).\label{eqn:interm}
\end{align}
While the third line above cannot be simplified for general $V$, for our particular choice 
$$h=\sum_{i=1}^p(\zeta_i \rlog_{z_i}+\xi_i \ilog_{z_i}),$$
the first and second lines can.  Indeed,
\[
h'(s)=\frac{1}{2}\sum_i\left(\frac{\zeta_i-\ii\xi_i}{s-z_i}+\frac{\zeta_i+\ii\xi_i}{s-\bar z_i}\right),
\]
so that
\[
\int_A^B\frac{h'(s)}{s-z}\tau(s)\rd s=\frac{1}{2}\sum_i\int_A^B\left(\frac{\zeta_i-\ii\xi_i}{z-z_i}\left(\frac{1}{s-z}-\frac{1}{s-z_i}\right)+
\frac{\zeta_i+\ii\xi_i}{z-\bar z_i}\left(\frac{1}{s-z}-\frac{1}{s-\bar z_i}\right)
\right)\tau(s)\rd s.
\]
From (\ref{eqn:linear}) and the definition of $v$,  we can write
\[
\frac{1}{2\pi}\int_A^B\frac{h'(s)}{s-z}\tau(s)\rd s=\frac{1}{2}\sum_i\left(\frac{(\zeta_i-\ii\xi_i)(v(z)-v(z_i))}{z-z_i}+
\frac{(\zeta_i+\ii\xi_i)(v(z)-v(\bar z_i))}{z-\bar z_i}
\right).
\]
Note that $v'(z)=-v(z)/b(z)$ and $v^2+(z-\frac{A+B}{2})v+\frac{1}{4}\left(\frac{A-B}{2}\right)^2=0$ , so that, abbreviating $v_i=v(z_i)$ and $\gamma=((A-B)/2)^2/4$ we obtain
\[
\int_{w}^{w+\ii\infty}\frac{v(z)-v(z_i)}{z-z_i}\frac{\rd z}{b(z)}=\int_{v(w)}^{0}\frac{v-v_i}{v+\frac{\gamma}{v}-(v_i+\frac{\gamma}{v_i})}\frac{\rd v}{v}
=\int_{v(w)}^{0}\frac{1}{v-\frac{\gamma}{v_i}}\rd v=-\log(1-\frac{v(w) v(z_i)}{\gamma})
\]
and similarly 
\[
\int_{w}^{w+\ii\infty}\frac{v(\bar z)-v(z_i)}{\bar z-z_i}\frac{\rd z}{b(\bar z)}
=-\log\left(1-\frac{v(\bar w) v(z_i)}{\gamma}\right).
\]
We have therefore proved, using the notations (\ref{eqn:sigma}) and (\ref{eqn:shift}),
\[
\frac{\rd}{\rd t}\log Z_h(t)=t\,\sigma(\bzeta,\boxi,\bz)+\mu(\bzeta,\boxi,\bz)+\OO\left(\frac{Z_{\Re h}(t)^2}{|Z_h(t)|^2}\frac{(\log N)^{600}}{N\eta_0}\right)
\]
where here and below, we abbreviate $\OO=\OO_{M,p,\kappa}$.
From this equation we first conclude about the case of real-valued $h$. 
Then trivially $Z_{\Re h}(t)=Z_h(t)$
so integrating the above equation gives
\[
Z_h(t)=\exp\left(\frac{t^2}{2}\sigma(\bzeta,\boxi,\bz)+t\mu(\bzeta,\boxi,\bz)+\OO\left(\frac{(\log N)^{600}}{N\eta_0}\right)\right).
\]
From our assumption on $\eta_0$  the above error term is $\OO(1)$.
For the general complex case, we  now have
\begin{equation}\label{eqn:diffeqn}
\frac{\rd}{\rd t}\log Z_h(t)=t\,\sigma(\bzeta,\boxi,\bz)+\mu(\bzeta,\boxi,\bz)+\OO\left(\frac{(\log N)^{600}}{N\eta_0}
\cdot \frac{e^{\frac{t^2}{2}\sigma(\re\bzeta,\re\boxi,\bz)}}{|Z_h(t)|}
\right),
\end{equation}
so that, taking the real part above, we have 
\[
\frac{\frac{\rd}{\rd t}|Z_h(t)|}{|Z_h(t)|}=t\,\sigma(\re\bzeta,\re\boxi,\bz)-t\,\sigma(\im\bzeta,\im\boxi,\bz)+\Re\mu(\bzeta,\boxi,\bz)+\OO\left(\frac{(\log N)^{600}}{N\eta_0}
\cdot \frac{e^{t^2\sigma(\re\bzeta,\re\boxi,\bz)}}{|Z_h(t)|^2}
\right).
\]
Defining $g(t)=|Z_h(t)|^2e^{t^2 \sigma(\im\bzeta,\im\boxi,\bz)-t^2\,\sigma(\re\bzeta,\re\boxi,\bz)-t\mu(\bzeta,\boxi,\bz)}$, the above equation implies
\[
g'(t)=\OO\left(\frac{(\log N)^{600}}{N\eta_0}
\cdot e^{t^2\sigma(\im\bzeta,\im\boxi,\bz)}
\right).
\]
From the assumption $\Im(\bzeta,\boxi)\in\sqrt{\beta}\cdot [-\tfrac{1}{10p},\tfrac{1}{10p}]^{2p}$, we have
$\sigma(\im\bzeta,\im\boxi,\bz)\leq \frac{1}{5}\log (N\eta_0)$, so that
$
g'(t)=\OO\left(\frac{1}{\sqrt{N\eta_0}}\right)
$
and we have proved that
\[
|Z_h(t)|=e^{-\frac{t^2}{2}\sigma(\im\bzeta,\im\boxi,\bz)+\frac{t^2}{2}\,\sigma(\re\bzeta,\re\boxi,\bz)+t\Re\mu(\bzeta,\boxi,\bz)}\cdot \left(1+\frac{1}{\sqrt{N\eta_0}}\right).
\]
Inserting this estimate in (\ref{eqn:diffeqn}) finally gives
\[
Z_h(t)=e^{\frac{t^2}{2}\,\sigma(\bzeta,\boxi,\bz)+t\mu(\bzeta,\boxi,\bz)}\cdot \left(1+\frac{1}{\sqrt{N\eta_0}}\right),
\]
We note that all equations since (\ref{eqn:interm}) hold only provided that
$ |Z_h(s)|\geq e^{-(\log N)^2}$ for $s\in[0,t]$, which is necessary to apply Lemma \ref{lem:estimate_varphi}.
Therefore, denoting $t_0=\max\{t\in[0,1]:Z_h(t)> e^{-(\log N)^2}\}$, for large enough $N$ we have
\[Z_h(t_0)=e^{\frac{t_0^2}{2}\,\sigma(\bzeta,\boxi,\bz)+t_0\mu(\bzeta,\boxi,\bz)}\cdot (1+\frac{1}{\sqrt{N\eta_0}})>e^{-(\log N)^2},\]
where we have used in the above inequality the easy estimates $\sigma(\bzeta,\boxi,\bz)=\OO(\log N)$ and $\mu(\bzeta,\boxi,\bz)=\OO(1)$.
By continuity this implies $t_0=1$.
The expected result therefore holds by taking $t=1$. 

\subsection{Generalization to further potentials.}\label{s:otherpotentials}\ 
The proof of the  local law \Cref{th:local_law_bulk} is the only place requiring the sub-quadratic growth assumption from (\ref{eq:alternative_assumption}).  Theorem \ref{thm:Beta} also holds for
$V$ growing at least linearly,  as in Assumption (A2) (ii) through the following steps.
\begin{enumerate}
\item Denoting $\E^{[A-\delta,B+\delta]}$ for the expectation conditional on all particles remaining in $[A-\delta,B+\delta]$, by \cite[Equations (2.25), (2.14)]{BouModPai2020}, the following local law holds:
\[
		\E^{[A-\delta,B+\delta]}\left[ \absa{ s_N(z) - m_V(z) }^{2q} \right]
		\leq \frac{(Cq)^{q}}{(N\eta)^{2q}}
			+ \frac{C^qe^{-cN}}{\abs{z-A}^{q} \abs{z-B}^{q}}.
	\]
When compared to \cite[Theorem 1.1]{BouModPai2020},  note the exponentially small second error term,  possible thanks to working under the conditioned measure.  This improvement is essential to the proof of Theorem \ref{thm:Beta}.
\item Based on this local law for the conditioned measure,  an analogue of the previous quantitative central limit theorem, \Cref{thm:GaussFluct}, can be proved under  $\E^{[A-\delta,B+\delta]}$ for a function ${\tilde L}_{N}$ coinciding with $L_N$ on $[A-\delta/2,B+\delta/2]$ but compactly supported on  $[A-\delta,B+\delta]$. This gives Theorem \ref{thm:Beta} for ${\tilde L}_N$, and then for $L_N$ as the probability of a particle outside $[A-\delta/2,B+\delta/2]$ is $\oo(1)$ by rigidity.
\end{enumerate}


\begin{bibdiv}
\begin{biblist}

\bib{albeverio2011}{article}{
  title={On the 1/n expansion for some unitary invariant ensembles of random matrices},
  author={Albeverio, S.},
  author={Pastur, L.},
  author= {Shcherbina, M.},
  journal={Communications in Mathematical Physics},
  volume={224},
  number={1},
  pages={271--305},
  year={2001},
  publisher={Springer}
}


\bib{anderson2010introduction}{book}{
   author={Anderson, G.},
   author={Guionnet, A.},
   author={Zeitouni, O.},
   title={An introduction to random matrices},
   series={Cambridge Studies in Advanced Mathematics},
   volume={118},
   publisher={Cambridge University Press, Cambridge},
   date={2010},
   pages={xiv+492}
}



\bib{ArgBelBou2017}{article}{
   author={Arguin, L.-P.},
   author={Belius, D.},
   author={Bourgade, P.},
   title={Maximum of the characteristic polynomial of random unitary
   matrices},
   journal={Comm. Math. Phys.},
   volume={349},
   date={2017},
   number={2},
   pages={703--751}
}


\bib{ArgBelBouRadSou2016}{article}{
   author={Arguin, L.-P.},
   author={Belius, D.},
   author={Bourgade, P.},
   author={Radziwi\l \l , M.},
   author={Soundararajan, K.},
   title={Maximum of the Riemann zeta function on a short interval of the
   critical line},
   journal={Comm. Pure Appl. Math.},
   volume={72},
   date={2019},
   number={3},
   pages={500--535}
}

\bib{arguin2020fyodorov1}{article}{
  title={The Fyodorov-Hiary-Keating Conjecture. I},
   author={Arguin, L.-P.},
   author={Bourgade, P.},
   author={ Radziwi{\l}{\l}, M.},
  journal={arXiv preprint arXiv:2007.00988},
  year={2020}
}

\bib{arguin2020fyodorov2}{article}{
  title={The Fyodorov-Hiary-Keating Conjecture. II},
   author={Arguin, L.-P.},
   author={Bourgade, P.},
   author={ Radziwi{\l}{\l}, M.},
  journal={arXiv preprint arXiv:2307.00982},
  year={2023}
}



\bib{augeri2020clt}{article}{
  title={A CLT for the characteristic polynomial of random Jacobi matrices, and the G$\beta$E},
  author={Augeri, F.},
  author={Butez, R.},
  author={Zeitouni, O.},
  journal={Prob. Theory Rel. Fields},
  year={2023},
  volume={186},
  pages={1--89}
}


\bib{bailey2022maxima}{article}{
  title={Maxima of log-correlated fields: some recent developments},
  author={Bailey, E.},
  author= {Keating, J.P.},
  journal={Journal of Physics A: Mathematical and Theoretical},
  volume={55},
  number={5},
  pages={053001},
  year={2022},
  publisher={IOP Publishing}
}


\bib{BauHof2022}{article}{
   author={Bauerschmidt, R.},
   author={Hofstetter, M.},
   title={Maximum and coupling of the sine-Gordon field},
   journal={Ann. Probab.},
   volume={50},
   date={2022},
   number={2},
   pages={455--508}
}




\bib{BelKis14}{article}{
   author={Belius, D.},
   author={Kistler, N.},
   title={The subleading order of two dimensional cover times},
   journal={Probab. Theory Related Fields},
   volume={167},
   date={2017},
   number={1-2},
   pages={461--552}
}

\bib{BelRosZei20}{article}{
AUTHOR = {Belius, D.},
author={Rosen, J.},
author={Zeitouni, O.},
     TITLE = {Tightness for the cover time of the two dimensional sphere},
   JOURNAL = {Probab. Theory Related Fields},
    VOLUME = {176},
      YEAR = {2020},
    NUMBER = {3-4},
     PAGES = {1357--1437},
      ISSN = {0178-8051,1432-2064}
}


\bib{BelWu20}{article}{
AUTHOR = {Belius, D.},
author={Wu, W.},
     TITLE = {Maximum of the {G}inzburg-{L}andau fields},
   JOURNAL = {Ann. Probab.},
    VOLUME = {48},
      YEAR = {2020},
    NUMBER = {6},
     PAGES = {2647--2679},
      ISSN = {0091-1798,2168-894X}
}



\bib{BenBou2013}{article}{
   author={Ben Arous, G.},
   author={Bourgade, P.},
   title={Extreme gaps between eigenvalues of random matrices},
   journal={Ann. Probab.},
   volume={41},
   date={2013},
   number={4},
   pages={2648--2681}
}

\bib{benaych2016lectures}{article}{
  author={Benaych-Georges, F.},
  author={Knowles, A.},
  title={Local semicircle law for Wigner matrices},
  journal={Panor. Synth\`{e}ses},
  volume={53},
  publisher={Soci\'{e}t\'{e} Math\'{e}matique de France},
  year={2017}
}

\bib{BenLop2022}{article}{
   author={Benigni, L.},
   author={Lopatto, P.},
   title={Optimal delocalization for generalized Wigner matrices},
   journal={Adv. Math.},
   volume={396},
   date={2022}
}
\bib{BenLopFluc2022}{article}{
AUTHOR = {Benigni, L.},
author={Lopatto, P.},
     TITLE = {Fluctuations in local quantum unique ergodicity for
              generalized {W}igner matrices},
   JOURNAL = {Comm. Math. Phys.},
    VOLUME = {391},
      YEAR = {2022},
    NUMBER = {2},
     PAGES = {401--454},
      ISSN = {0010-3616,1432-0916}
}


\bib{BerWebWon2017}{article}{
   author={Berestycki, N.},
   author={Webb, C.},
   author={Wong, M. D.},
   title={Random Hermitian matrices and Gaussian multiplicative chaos},
   journal={Probab. Theory Related Fields},
   volume={172},
   date={2018},
   number={1-2},
   pages={103--189}
}




\bib{Bis20}{book}{
    AUTHOR = {Biskup, M.},
     TITLE = {Extrema of the two-dimensional discrete {G}aussian free field},
 BOOKTITLE = {Random graphs, phase transitions, and the {G}aussian free
              field},
    SERIES = {Springer Proc. Math. Stat.},
    VOLUME = {304},
     PAGES = {163--407},
 PUBLISHER = {Springer, Cham},
      YEAR = {2020}
}

\bib{BolDeuGia01}{article}{
    AUTHOR = {Bolthausen, E.},
   AUTHOR = {Deuschel, J.-D.},
   AUTHOR = {Giacomin,  G.},
     TITLE = {Entropic repulsion and the maximum of the two-dimensional
              harmonic crystal},
   JOURNAL = {Ann. Probab.},
    VOLUME = {29},
      YEAR = {2001},
    NUMBER = {4},
     PAGES = {1670--1692}
}





\bib{Bou2018}{article}{
   author={Bourgade, P.},
   title={Extreme gaps between eigenvalues of Wigner matrices},
   journal={J. Eur. Math. Soc.},
   year={2022},
   volume={8},
   pages={2823--2873}
}






\bib{BouErdYau2014}{article}{
   author={Bourgade, P.},
   author={Erd\H os, L.},
   author={Yau, H.-T.},
   title={Edge universality for beta ensembles},
   journal={Communications in Mathematical Physics},
   number={1},
   volume={332},
   date={2014},
   pages={261--353}
}


\bib{BouErdYauYin2016}{article}{
   author={Bourgade, P.},
   author={Erd{\H{o}}s, L.},
   author={Yau, H.-T.},
   author={Yin, J.},
   title={Fixed energy universality for generalized Wigner matrices},
   journal={Comm. Pure Appl. Math.},
   volume={69},
   date={2016},
   number={10},
   pages={1815--1881}
}





\bib{BouFal2022}{article}{
  title={Liouville quantum gravity from random matrix dynamics},
  author={Bourgade, P.},
  author={Falconet, H.},
  journal={arXiv preprint arXiv:2206.03029},
  year={2022}
}


\bib{BouMod2018}{article}{
  title={Gaussian fluctuations of the determinant of Wigner matrices},
  author={Bourgade, P.},
  author={Mody, K.},
  journal={Electronic Journal of Probability},
  volume={24},
  pages={1--28},
  year={2019},
  publisher={Institute of Mathematical Statistics and Bernoulli Society}
}


\bib{BouModPai2020}{article}{
   author={Bourgade, P.},
   author={Mody, K.},
   author={Pain, M.},
   title={Optimal local law and central limit theorem for $\beta$-ensembles},
   journal={Communications in Mathematical Physics},
  volume={390},
  number={3},
  pages={1017--1079},
  year={2022},
  publisher={Springer}
}



\bib{de1995statistical}{article}{
  title={On the statistical mechanics approach in the random matrix theory: integrated density of states},
  author={Boutet De Monvel, A.},
  author={Pastur, L.},
  author={Shcherbina, M.},
  journal={Journal of Statistical Physics},
  volume={79},
  number={3},
  pages={585--611},
  year={1995},
  publisher={Springer}
}




\bib{Bra1978}{article}{
   author={Bramson, M.},
   title={Maximal displacement of branching Brownian motion},
   journal={Comm. Pure Appl. Math.},
   volume={31},
   date={1978},
   number={5},
   pages={531--581}
}




\bib{Bra1983}{article}{
   author={Bramson, M.},
   title={Convergence of solutions of the Kolmogorov equation to travelling
   waves},
   journal={Mem. Amer. Math. Soc.},
   volume={44},
   date={1983},
   number={285},
   pages={iv+190}
}


\bib{BraDinZei2016}{article}{
   author={Bramson, M.},
   author={Ding, J.},
   author={Zeitouni, O.},
   title={Convergence in law of the maximum of the two-dimensional discrete
   Gaussian free field},
   journal={Comm. Pure Appl. Math.},
   volume={69},
   date={2016},
   number={1},
   pages={62--123}
}


\bib{BraDinZei2016bis}{article}{
   author={Bramson, M.},
   author={Ding, J.},
   author={Zeitouni, O.},
   title={Convergence in law of the maximum of nonlattice branching random
   walk},
   journal={Ann. Inst. Henri Poincar\'{e} Probab. Stat.},
   volume={52},
   date={2016},
   number={4},
   pages={1897--1924}
}






\bib{cacciapuoti2015bounds}{article}{
  title={Bounds for the Stieltjes transform and the density of states of Wigner matrices},
  author={Cacciapuoti, C.},
  author={Maltsev, A.},
  author={Schlein, B.},
  journal={Probability Theory and Related Fields},
  volume={163},
  number={1},
  pages={1--59},
  year={2015},
  publisher={Springer}
}







\bib{CarSunZyg20}{article}{
AUTHOR = {Caravenna, F.},
author={Sun, R.},
author={Zygouras, N.},
     TITLE = {The two-dimensional {KPZ} equation in the entire subcritical
              regime},
   JOURNAL = {Ann. Probab.},
    VOLUME = {48},
      YEAR = {2020},
    NUMBER = {3},
     PAGES = {1086--1127}
}

\bib{ChaFahWebWon2021}{article}{
   author={Charlier, C.},
   author={Fahs, B.},
   author={Webb,  C.},
   author={Wong,  M.  D.},
   title={Asymptotics of Hankel determinants with a multi-cut regular potential and Fisher-Hartwig singularities},
   journal={to appear in Mem. Amer. Math. Soc.},
   date={2021}
}



\bib{ChhMadNaj2018}{article}{
   author={Chhaibi, R.},
   author={Madaule, T.},
   author={Najnudel, J.},
   title={On the maximum of the ${\rm C}\beta {\rm E}$ field},
   journal={Duke Math. J.},
   volume={167},
   date={2018},
   number={12},
   pages={2243--2345}
}


\bib{CipErdWu2023}{article}{
   author={Cipolloni,  G.},
   author={Erd{\H o}s,  L.},
   author={Xu, Y.},
   title={Universality of extremal eigenvalues of large random matrices},
   journal={arXiv preprint arXiv:2312.08325}
}

\bib{CipolloniLandon}{article}{
   author={Cipolloni,  G.},
   author={Landon,  B.},
   title={Maximum of the Characteristic Polynomial of {IID} Matrices},
   year={2024},
   journal={arXiv preprint arXiv:2405.05045}
}


\bib{claeys2021much}{article}{
  title={How much can the eigenvalues of a random Hermitian matrix fluctuate?},
  journal={Duke Math. J.},
  author={Claeys, T.},
  author={Fahs, B.},
  author={Lambert, G.},
  author={Webb, C.},
  volume={170},
  number={9},
  pages={2085--2235},
  year={2021},
  publisher={Duke University Press}
}

\bib{CooZei20}{article}{
AUTHOR = {Cook, N.},
author={Zeitouni, O.},
     TITLE = {Maximum of the characteristic polynomial for a random
              permutation matrix},
   JOURNAL = {Comm. Pure Appl. Math.},
    VOLUME = {73},
      YEAR = {2020},
    NUMBER = {8},
     PAGES = {1660--1731},
      ISSN = {0010-3640,1097-0312}
}

\bib{CosZei23}{article}{
AUTHOR = {Cosco, C.},
author={Zeitouni, O.},
     TITLE = {Moments of partition functions of 2d {G}aussian polymers in
              the weak disorder regime---{I}},
   JOURNAL = {Comm. Math. Phys.},
    VOLUME = {403},
      YEAR = {2023},
    NUMBER = {1},
     PAGES = {417--450},
}

\bib{DemPerRosZeo04}{article}{
    AUTHOR = {Dembo, A.},
    AUTHOR = {Peres, Y.},
    AUTHOR = {Zeitouni, O.},
     TITLE = {Cover times for {B}rownian motion and random walks in two
              dimensions},
   JOURNAL = {Ann. of Math. (2)},
    VOLUME = {160},
      YEAR = {2004},
    NUMBER = {2},
     PAGES = {433--464}
}

\bib{Dia03}{article}{
AUTHOR = {Diaconis, P.},
     TITLE = {Patterns in eigenvalues: the 70th {J}osiah {W}illard {G}ibbs
              lecture},
   JOURNAL = {Bull. Amer. Math. Soc. (N.S.)},
    VOLUME = {40},
      YEAR = {2003},
    NUMBER = {2},
     PAGES = {155--178},
      ISSN = {0273-0979,1088-9485},
}


\bib{ding2017convergence}{article}{
  title={Convergence of the centered maximum of log-correlated Gaussian fields},
  author={Ding, J.},
  author={Roy, R.},
  author={Zeitouni, O.},
  journal={The Annals of Probability},
  volume={45},
  number={6A},
  pages={3886--3928},
  year={2017},
  publisher={Institute of Mathematical Statistics}
}



\bib{ErdSchYau2011}{article}{
   author={Erd\H os, L.},
   author={Schlein, B.},
   author={Yau, H.-T.},
   title={Universality of random matrices and local relaxation flow},
   journal={Invent. Math.},
   volume={185},
   date={2011},
   number={1},
   pages={75--119}
}

\bib{erdos2017dynamical}{book}{
   author={Erd\H{o}s, L.},
   author={Yau, H.-T.},
  title={A Dynamical Approach to Random Matrix Theory},
  series={Courant Lecture Notes in Mathematics},
  volume={28},
  year={2017}
}


\bib{ErdYauYin2012}{article}{
   author={Erd\H os, L.},
   author={Yau, H.-T.},
   author={Yin, J.},
   title={Rigidity of eigenvalues of generalized Wigner matrices},
   journal={Adv. Math.},
   volume={229},
   date={2012},
   number={3},
   pages={1435--1515}
}


\bib{FengWei2018II}{article}{
   author={Feng, R.},
   author={Wei, D.},
   title={Large gaps of CUE and GUE},
   journal={arXiv preprint arXiv:1807.02149},
   date={2018}
}



\bib{FyoHiaKea2012}{article}{
   author={Fyodorov, Y.},
   author={Hiary, G.},
   author={Keating, J.},
   title={Freezing Transition, Characteristic Polynomials of Random Matrices, and the Riemann Zeta Function},
   journal={Physical Review Letters},
   volume={108},
   date={2012}
}


\bib{fyodorov2014freezing}{article}{
  title={Freezing transitions and extreme values: random matrix theory, and disordered landscapes},
  author={Fyodorov, Y.},
  author={Keating, J.},
  journal={Philosophical Transactions of the Royal Society A},
  volume={372},
  number={2007},
  pages={20120503},
  year={2014},
  publisher={The Royal Society Publishing.}
}




\bib{fyodorov2016distribution}{article}{
  title={On the distribution of the maximum value of the characteristic polynomial of GUE random matrices},
  author={Fyodorov, Y.},
  author={Simm, N.},
  journal={Nonlinearity},
  volume={29},
  number={9},
  pages={2837},
  year={2016},
  publisher={IOP Publishing}
}



\bib{GotNauTikTim2013}{article}{
   author={G\"{o}tze, F.},
   author={Tikhomirov, A.},
   title={On the rate of convergence to the semi-circular law},
   conference={
      title={High dimensional probability VI},
   },
   book={
      series={Progr. Probab.},
      volume={66},
      publisher={Birkh\"{a}user/Springer, Basel},
   },
   date={2013},
   pages={139--165}
}
\bib{GotNauTikTim2018}{article}{
   author={G\"{o}tze, F.},
   author={Naumov, A.},
   author={Tikhomirov, A.},
   author={Timushev, D.},
   title={On the local semicircular law for Wigner ensembles},
   journal={Bernoulli},
   volume={24},
   date={2018},
   number={3},
   pages={2358--2400}
}

\bib{GuiZei2000}{article}{
   author={Guionnet, A.},
   author={Zeitouni, O.},
   title={Concentration of the spectral measure for large matrices},
   journal={Electron. Comm. Probab.},
   volume={5},
   date={2000},
   pages={119--136}
}

\bib{Gus2005}{article}{
   author={Gustavsson, J.},
   title={Gaussian fluctuations of eigenvalues in the GUE},
   language={English, with English and French summaries},
   journal={Ann. Inst. H. Poincar\'{e} Probab. Statist.},
   volume={41},
   date={2005},
   number={2},
   pages={151--178}
}

\bib{Har2019}{article}{
    AUTHOR = {Harper, A. J.},
     TITLE = {On the partition function of the Riemann zeta function, and the Fyodorov-Hiary-Keating conjecture},
   JOURNAL = {arXiv preprint arXiv:1906.05783},
      YEAR = {2019},
}


\bib{Har2019II}{article}{
   author={Harper, A. J.},
   title={The Riemann zeta function in short intervals [after Najnudel, and
   Arguin, Belius, Bourgade, Radziwi\l \l  and Soundararajan]},
   journal={Ast\'{e}risque, S\'{e}minaire Bourbaki. Vol. 2018/2019.},
   volume={422},
      date={2020},
   pages={Exp. No. 1161, 391--414}
}


\bib{ItsKra2008}{article}{
   author={Its, A.},
   author={Krasovsky, I.},
   title={Hankel determinant and orthogonal polynomials for the Gaussian
   weight with a jump},
   conference={
      title={Integrable systems and random matrices},
   },
   book={
      series={Contemp. Math.},
      volume={458},
      publisher={Amer. Math. Soc., Providence, RI},
   },
   date={2008},
   pages={215--247}
}


\bib{Joh1998}{article}{
   author={Johansson, K.},
   title={On fluctuations of eigenvalues of random Hermitian matrices},
   journal={Duke Math. J.},
   volume={91},
   date={1998},
   number={1},
   pages={151--204}}


\bib{KnoYin2013}{article}{
   author={Knowles, A.},
   author={Yin, J.},
   title={The isotropic semicircle law and deformation of Wigner matrices},
   journal={Comm. Pure Appl. Math.},
   volume={66},
   date={2013},
   number={11},
   pages={1663--1750}
}




\bib{lambert2020maximum}{article}{
 title={Maximum of the characteristic polynomial of the Ginibre ensemble},
  author={Lambert, G.},
  journal={Communications in Mathematical Physics},
  volume={378},
  number={2},
  pages={943--985},
  year={2020},
  publisher={Springer}
}

\bib{lambert2021mesoscopic}{article}{ 
  title={Mesoscopic central limit theorem for the circular $\beta$-ensembles and applications},
  author={Lambert, G.},
  journal={Electronic Journal of Probability},
  volume={26},
  pages={1--33},
  year={2021},
  publisher={Institute of Mathematical Statistics and Bernoulli Society}
}



\bib{lambert2023law}{article}{
  title={Law of large numbers for the maximum of the two-dimensional Coulomb gas potential},
     author={Lambert, G.},
        author={Lebl\'e,  T.},
           author={Zeitouni,  O.},
  journal={arXiv preprint arXiv:2303.09912},
  year={2023}
}
   

\bib{LamOstSim2018}{article}{
   author={Lambert, G.},
   author={Ostrovsky, D.},
   author={Simm, N.},
   title={Subcritical multiplicative chaos for regularized counting
   statistics from random matrix theory},
   journal={Comm. Math. Phys.},
   volume={360},
   date={2018},
   number={1},
   pages={1--54}
}

\bib{LamPaq2019}{article}{
   author={Lambert, G.},
   author={Paquette, E.},
   title={The law of large numbers for the maximum of almost Gaussian log-correlated random fields coming from random matrices},
   journal={Probab. Theory Related Fields},
   date={2019},
   volume={173},
   note={See arXiv 1611.08885 version 3},
   pages={157--209}
}

\bib{lambert2020strong1}{article}{
  title={Strong approximation of Gaussian $\beta$-ensemble characteristic polynomials: the edge regime and the stochastic Airy function},
  author={Lambert, G.},
  author={Paquette, E.},
  journal={arXiv preprint arXiv:2009.05003},
  year={2020}
}

\bib{lambert2020strong2}{article}{
  title={Strong approximation of Gaussian $\beta$-ensemble characteristic polynomials: the hyperbolic regime},
  author={Lambert, G.},
  author={Paquette, E.},
  journal={Ann. Appl. Probab. 33},
  year={2023},
  volume={33},
  pages={549--612}
}


\bib{LanLopMar2018}{article}{
   author={Landon, B.},
   author={Lopatto, P.},
   author={Marcinek, J.},
   title={Comparison theorem for some extremal eigenvalue statistics},
   journal={Annals Probab.},
   date={2020},
   volume={48},
   pages={2894--2919}
}

\bib{LanSos2018}{article}{
   author={Landon, B.},
   author={Sosoe, P.},
   title={Applications of mesoscopic CLTs in random matrix theory},
   journal={Annals Appl.  Probab.},
   date={2020},
   volume={30},
   pages={2769--2795}
}

\bib{LanSosYau2016}{article}{
   author={Landon, B.},
   author={Sosoe, P.},
   author={Yau, H.-T.},
   title={Fixed energy universality for Dyson Brownian motion},
   journal={Advances in Mathematics},
   date={2019},
   volume={346},
   pages={1137--1332}
}



\bib{lee2018local}{article}{
  title={Local law and Tracy--Widom limit for sparse random matrices},
  author={Lee, J. O.},
  author={Schnelli, K.},
  journal={Probability Theory and Related Fields},
  volume={171},
  number={1},
  pages={543--616},
  year={2018},
  publisher={Springer}
}


\bib{LytPas2009}{article}{
   author={Lytova, A.},
   author={Pastur, L.},
   title={Central limit theorem for linear eigenvalue statistics of random
   matrices with independent entries},
   journal={Ann. Probab.},
   volume={37},
   date={2009},
   number={5},
   pages={1778--1840}
}



\bib{Naj16}{article}{
   author={Najnudel, J.},
   title={On the extreme values of the Riemann zeta function on random
   intervals of the critical line},
   journal={Probab. Theory Related Fields},
   volume={172},
   date={2018},
   number={1-2},
   pages={387--452}
}



\bib{Ngu2018}{article}{
   author={Nguyen, H.},
   title={Random matrices: overcrowding estimates for the spectrum},
   journal={J. Funct. Anal.},
   volume={275},
   date={2018},
   number={8},
   pages={2197--2224}
}


\bib{nikula2020multiplicative}{article}{
  title={Multiplicative chaos and the characteristic polynomial of the CUE: The $L^1$-phase},
  author={Nikula, M.},
  author={Saksman, E.},
  author={Webb, C.},
  journal={Transactions of the American Mathematical Society},
  volume={373},
  number={6},
  pages={3905--3965},
  year={2020}
}





\bib{Oro2010}{article}{
   author={O'Rourke, S.},
   title={Gaussian fluctuations of eigenvalues in Wigner random matrices},
   journal={J. Stat. Phys.},
   volume={138},
   date={2010},
   number={6},
   pages={1045--1066}
}


\bib{PaqZei2018}{article}{
   author={Paquette, E.},
   author={Zeitouni, O.},
   title={The maximum of the CUE field},
   journal={Int. Math. Res. Not. IMRN},
   date={2018},
   number={16},
   pages={5028--5119}
}



\bib{PaqZei2022}{article}{
   author={Paquette, E.},
   author={Zeitouni, O.},
   title={The extremal landscape for the C$\beta$E ensemble},
   journal={arXiv preprint arXiv:2209.06743},
   date={2022}
}




\bib{Shc2013}{article}{
   author={Shcherbina, M.},
   title={Fluctuations of linear eigenvalue statistics of $\beta$ matrix
   models in the multi-cut regime},
   journal={J. Stat. Phys.},
   volume={151},
   date={2013},
   number={6},
   pages={1004--1034}
}





\bib{TaoVu2011}{article}{
   author={Tao, T.},
   author={Vu, V.},
   title={Random matrices: universality of local eigenvalue statistics},
   journal={Acta Math.},
   volume={206},
   date={2011},
   number={1},
   pages={127--204}
}


\bib{TaoVu2013}{article}{
   author={Tao, T.},
   author={Vu, V.},
   title={Random matrices: sharp concentration of eigenvalues},
   journal={Random Matrices Theory Appl.},
   volume={2},
   date={2013},
   number={3}
}


\bib{Web2015}{article}{
   author={Webb, C.},
   title={The characteristic polynomial of a random unitary matrix and
   Gaussian multiplicative chaos---the $L^2$-phase},
   journal={Electron. J. Probab.},
   volume={20},
   date={2015},
   pages={no. 104, 21}
}

\bib{Zei16}{incollection}{
    AUTHOR = {Zeitouni, O.},
     TITLE = {Branching random walks and {G}aussian fields},
 BOOKTITLE = {Probability and statistical physics in {S}t. {P}etersburg},
    SERIES = {Proc. Sympos. Pure Math.},
    VOLUME = {91},
     PAGES = {437--471},
 PUBLISHER = {Amer. Math. Soc., Providence, RI},
      YEAR = {2016}
}


\end{biblist}
\end{bibdiv}
\end{document}